\documentclass[12pt,reqno]{amsart}
\usepackage{amsmath,amssymb,amsfonts,amscd,latexsym,amsthm,mathrsfs,verbatim,textcomp}
\usepackage[colorlinks,linkcolor=black,citecolor=black]{hyperref}
\usepackage{cite}
\usepackage{hypbmsec}
\usepackage{enumerate}
\usepackage{bm}
\usepackage{tikz}
\usepackage{mathtools}
\usepackage{verbatim}
\usepackage{xfrac, mathabx}
\theoremstyle{plain}
\usepackage{manfnt}
\usepackage{tikz}
\usetikzlibrary{matrix,arrows,decorations.pathmorphing}
\usepackage[margin=1.1in]{geometry}
\usepackage{csquotes}
\newtheorem{theorem}{Theorem}
\newtheorem{lemma}[theorem]{Lemma}
\newtheorem{corollary}[theorem]{Corollary}
\newtheorem{prop}[theorem]{Proposition}
\theoremstyle{remark}
\newtheorem{remark}[theorem]{\bf Remark}

\newtheorem*{prob*}{\bf Problem}

\newtheorem*{quest*}{\bf Question}
\renewcommand{\d}{{\mathrm d}}
\renewcommand{\Im}{\operatorname{Im}}
\renewcommand{\Re}{\operatorname{Re}}

\newcommand{\N}{\mathbb{N}}

\usepackage{etoolbox}
\patchcmd{\section}{\scshape}{\bfseries}{}{}
\makeatletter
\renewcommand{\@secnumfont}{\bfseries}
\makeatother

\numberwithin{theorem}{section}
\numberwithin{equation}{section}

\newcommand{\pMatrix}[4]{\left(\begin{matrix}#1 & #2 \\ #3 & #4\end{matrix}\right)}

\renewcommand{\pmatrix}[4]{\left(\begin{smallmatrix}#1 & #2 \\ #3 & #4\end{smallmatrix}\right)}

\makeatletter
\@namedef{subjclassname@2020}{\textup{2020} Mathematics Subject Classification}
\makeatother

\begin{document}
\author{Alexander Dunn}
\address{School of Mathematics, Georgia Institute of Technology, 
Atlanta, USA}
\email{adunn61@gatech.edu}

\title{Metaplectic cusp forms and the large sieve}

\dedicatory{Dedicated to Chantal David on the occasion of her 60th birthday.}

\begin{abstract}
We prove a power saving upper bound for the sum of Fourier coefficients $\rho_f(\cdot)$ of a fixed cubic metaplectic cusp form
$f$ over primes. Our result is the cubic analogue of a celebrated 1990 Theorem of Duke and Iwaniec,
and the cuspidal analogue of a Theorem due to the author and Radziwi{\l\l} for the bias in cubic Gauss sums.

The proof has two main inputs, both 
of independent interest.
Firstly, we prove a new large sieve estimate for a bilinear form 
whose kernel function is $\rho_f(\cdot)$. The proof of the bilinear estimate uses
a number field version of circle method due to Browning and Vishe, Voronoi summation,
and Gauss-Ramanujan sums. 
Secondly, we use Voronoi summation and the cubic large sieve of Heath-Brown 
to prove an estimate for a linear form involving $\rho_f(\cdot)$. Our linear estimate
overcomes a bottleneck occurring at level of distribution $2/3$.
\end{abstract}

\maketitle
\tableofcontents
\section{Introduction} \label{introduction}

\subsection{Background and statement of results} \label{introsec}

Arithmetic functions that arise from the Fourier coefficients of
automorphic forms on congruence subgroups of $\operatorname{SL}_2(\mathbb{Z})$
encode deep arithmetic and analytic information. A famous example is the Modularity
theorem for elliptic curves $E/\mathbb{Q}$ \cite{BCDT}, and its resolution of the Hasse--Weil conjecture for such curves.

At a fundamental level, automorphic forms on congruence subgroups of $\operatorname{SL}_2(\mathbb{Z})$ are nice objects
because 
there is an ``adequate Hecke theory" available. By this, we mean the basic property
that the sequence of  Fourier coefficients of an integer weight 
cusp form restricted to values coprime to the level can be expressed as a linear
combination of multiplicative 
functions given by the Hecke eigenvalues! 
It is well-known that a power saving upper bound for the sum of Hecke eigenvalues $\lambda_g(\cdot)$
over primes would yield a rectangular zero-free region in the critical strip for the associated $L$-function $L(s,g)$
(thanks to the Euler product).
Unfortunately, the proof of such a bound is well out of reach of current technology!

The Fourier coefficients of half-integer weight modular forms also play a
key role in arithmetic. An important example is
the use of Dedekind's $\eta$-function (holomorphic cusp form of weight $1/2$ on $\operatorname{SL}_2(\mathbb{Z}))$ 
in the proof of Rademacher's formula \cite{Rade} for the partition function $p(n)$. 
Hecke observed in \cite[pg.~639]{Hecke1} and \cite {Hecke2}
that there is not an ``adequate Hecke theory" (in the naive sense above) for
modular forms of half-integer weight.
Wohlfahrt \cite{Woh} confirmed Hecke's observations
and essentially showed that there is an algebra of Hecke operators
$\mathbb{C}[\{T_{n^2}\}_{n=1}^{\infty}]$ acting on half-integer weight modular forms of weight $k$
such that 
 $T_{m^2} \circ T_{n^2}=T_{m^2 n^2}=T_{n^2} \circ T_{m^2}$ for $(m,n)=1$,
$T_{p^{2 a}}$ is a polynomial in $T_{p^2}$ for each $a \in \mathbb{Z}_{\geq 1}$ and odd prime $p$,
and that each Hecke operator is Hermitian (on cusp forms) with respect to
the standard Petersson inner product.
In general, the Fourier coefficients of 
half-integer weight Hecke eigenforms at general integer indices are not multiplicative, unless they are squares!
In foundational works, Shimura \cite{Shimura} and Kohnen--Zagier \cite{KZ} studied this phenomenon in more detail.
For a comprehensive summary of the theory, the reader can consult \cite[\S 4.3]{Kob}.

Duke and Iwaniec \cite{DI} in 1990 gave striking quantitative evidence
that the Fourier coefficients of half-integer weight holomorphic cusp forms along squarefree integers are not multiplicative (unless their values are zero).
In particular, suppose that $g$ is a holomorphic cusp form on $\Gamma_0(N)$ ($N \equiv 0 \pmod{4}$) having weight $k=1/2+2 \ell$, $\ell \in \mathbb{Z}_{\geq 2}$,
and Fourier expansion (at $\infty$)
\begin{equation}
g(z)=\sum_{n=1}^{\infty} c_g(n) n^{(k-1)/2} e(nx) e^{-2 \pi n y}, \quad z=x+iy \in \mathbb{H},
\end{equation}
where $c_g(n) \in \mathbb{C}$, $e(x):=e^{2 \pi i x}$ for all $x \in \mathbb{R}$, and $\mathbb{H}:=\mathbb{R} \times \mathbb{R}_{>0}$
is the complex upper-half plane. 
For $\varepsilon>0$ and $A,B \geq 10$, Duke and Iwaniec \cite{DI}
proved that 
\begin{equation} \label{DIbilinear}
\sum_{a \leq A} \sum_{b \leq B} \mu^2(a) \alpha_a \beta_b c_g(ab) \ll_{\varepsilon,g}
(AB)^{\varepsilon} (B^{1/2}+A B^{1/4} ) \|\boldsymbol{\alpha} \|_2
\|\boldsymbol{\beta} \|_2,
\end{equation}
where $\boldsymbol{\alpha}$,$\boldsymbol{\beta}$ are $\mathbb{C}$-valued sequences 
and $\|\bullet \|$ denotes the usual $\ell_2$-norm. Using \eqref{DIbilinear}
together with appropriate linear estimates, Duke and Iwaniec \cite{DI} also proved that 
\begin{equation} \label{diineq}
\sum_{\substack{ p \leq X \\ p \text{ prime} }} c_g(p)  \ll_{\varepsilon,g} X^{1-1/156+\varepsilon},
\end{equation}
as $X \rightarrow \infty$.
The result in \eqref{diineq} allows for twists by primitive characters of conductor divisible by $N \equiv 0 \pmod{4}$,
and so one can restrict to sum to primes in an arithmetic progression (with the implied constant depending on the modulus).

The goal of this paper is to generalise the results of Duke and Iwaniec
to cusp forms
on the cubic metaplectic cover of $\operatorname{GL}_2$ (in the sense of Kubota \cite{Kubota1,Kubota2}).
This is the complementary case to recent work by the author and Radziwi{\l\l} \cite{DR}
on Patterson's conjecture for the bias of cubic Gauss sums over primes
(cubic Gauss sums are the Fourier coefficients of the cubic theta function \cite{Pat1} which is non-cuspidal).
The spectral theory of cubic metaplectic forms have played a key role
in the work of Livn\'{e}--Patterson \cite{LP}, and Louvel \cite{Lo},
on the distribution of certain cubic exponential sums. In their 2004
PhD thesis \cite{Moh}, M\"{o}hring numerically investigated the 
Fourier coefficients of some cuspidal cubic metaplectic forms. 

Before stating our results we briefly introduce some notation.
Let $\mathbb{H}^{3}:=\mathbb{C} \times \mathbb{R}_{>0}$ denote hyperbolic $3$-space.
Let $\omega=e^{2 \pi i/3}$, and $\mathbb{Q}(\omega)$ denote the Eisenstein quadratic field (class number $1$).
This number field has
ring of integers $\mathbb{Z}[\omega]$, discriminant $-3$, and the unique ramified prime is $\lambda:=\sqrt{-3}=1+2 \omega$.
Let $(\frac{\bullet}{c})_3$ denote the cubic symbol over $\mathbb{Z}[\omega]$,
and $\Lambda(c)$ denote the usual von Mangoldt function
on $\mathbb{Z}[\omega]$. Consider the following
congruence groups:
\begin{align*}
\Gamma&:=\operatorname{SL}_2(\mathbb{Z}[\omega]);  \\
\Gamma_1(3)&:=\{ \gamma \in \Gamma : \gamma \equiv I \pmod{3} \}; \\
\Gamma_2&:=\langle \operatorname{SL}_2(\mathbb{Z}),\Gamma_1(3) \rangle.
\end{align*}
The cubic Kubota \cite{Kubota1,Kubota2} character
$\chi: \Gamma_1(3) \rightarrow \{1, \omega, \omega^2 \}$ is defined by
\begin{equation} \label{cubekubota}
\chi(\gamma):=\begin{cases}
\big( \frac{c}{a} \big)_3  & \text{if } c \neq 0 \\
1 & \text{if } c=0
\end{cases},
\quad \gamma=\pMatrix a b c d \in \Gamma_1(3),
\end{equation}
and extends to a well-defined homomorphism 
$\chi: \Gamma_2 \rightarrow \{1,\omega,\omega^2\}$
when one defines
$\chi \lvert_{\operatorname{SL}_2(\mathbb{Z})} \equiv 1$ \cite[\S 2]{Pat1}. 
The group $\Gamma_2$
is the lowest possible level for cubic metaplectic forms.
Let $f$ be a cuspidal cubic metaplectic form on $\Gamma_2$
i.e.
\begin{itemize}
\item  $f$ vanishes at all cusps of $\Gamma_2$;
\item $f(\gamma w)=\chi(\gamma) f(w)$ for all $\gamma \in \Gamma_2$ and $w \in \mathbb{H}^3$;
\item $f$ is an eigenfunction of the hyperbolic Laplacian: $\Delta f=-\tau_f(2- \tau_f)f$ for some $\tau_f \in \mathbb{C}$.
\end{itemize}
There is an algebra of Hecke operators 
$\mathbb{C}[\{T_{\nu^3}\}_{\nu \in \mathbb{Z}[\omega] \setminus \{0 \} }]$ acting on cubic metaplectic forms
such that 
 $T_{\mu^3} \circ T_{\nu^3}=T_{\mu^3 \nu^3}=T_{\nu^3} \circ T_{\mu^3}$ for $(\mu,\nu)=1$,
$T_{\varpi^{3 a}}$ is a polynomial in $T_{\varpi^3}$ for each $a \in \mathbb{Z}_{\geq 1}$ and prime $\varpi \equiv 1 \pmod{3}$,
and that each Hecke operator is Hermitian (on cusp forms) with respect to
the standard Petersson inner product \cite[\S 0.3.12]{Pro}.
The Fourier expansion of $f$ (at $\infty$) is given by
\begin{equation} \label{introfourier}
f(w)=\sum_{\substack{\nu \neq 0 \\ \nu \in \lambda^{-3} \mathbb{Z}[\omega] }} 
\rho_f(\nu) v K_{\tau_f-1} (4 \pi \lvert \nu \rvert v) 
\check{e}(\nu z), \quad w=(z,v) \in \mathbb{H}^3,
\end{equation}
where
$K_{\alpha}(\cdot)$ is the standard $K$-Bessel function of order $\alpha \in \mathbb{C}$,
$\check{e}(z):=e^{2 \pi i(z+\overline{z})}$ for $z \in \mathbb{C}$, and $\rho_f(\nu) \in \mathbb{C}$.

\begin{remark} \label{cubicshimurarem}
The cubic Shimura lift of Patterson \cite[Theorem~3.4]{Pat3} guarantees that
one always has $\tau_f \in 1+i \mathbb{R}$ for cuspidal cubic metaplectic forms $f$ on $\Gamma_2$ (see \S \ref{cubeshim}).
\end{remark}

Let $K,M \geq 1$, and $W_{K,M}:(0,\infty) \rightarrow \mathbb{C}$ be a smooth function with compact support in $[1,2]$ 
such that for each $j \in \mathbb{Z}_{\geq 0}$ we have
\begin{equation} \label{Kderivativebound}
W_{K,M}^{(j)}(x) \ll_j M K^{j} \quad \text{for all} \quad x>0.
\end{equation}
If $M=1$ then $M$ is omitted from the notation, and we write $W_{K}$. Let
$\| \bullet \|_q$ with $1 \leq q \leq \infty$ denote the $\ell_q$-norm of a $\mathbb{C}$-valued sequence
indexed over elements of $\mathbb{Z}[\omega]$. 

The main sums of interest in this paper are
\begin{equation} \label{PXdefn}
\mathscr{P}_f(X,v,u;W_K):=\sum_{\substack{\nu \in \lambda^{-3} \mathbb{Z}[\omega] \\ \lambda^3 \nu \equiv u \pmod{v} }} 
\rho_f(\nu)  \Lambda(\lambda^3 \nu) W_K \Big( \frac{N(\nu)}{X} \Big);
\end{equation}
\begin{equation} \label{PXdefn2}
\mathscr{P}_f(X,v,u):=\sum_{\substack{\nu \in \lambda^{-3} \mathbb{Z}[\omega]  \\ \lambda^3 \nu \equiv u \pmod{v} \\ N(\nu) \leq X }} 
\rho_f(\nu)  \Lambda(\lambda^3 \nu);
\end{equation}
and 
\begin{equation} \label{PXdefn3}
\widetilde{\mathscr{P}_f}(X,v,u):=\sum_{\substack{\varpi \in \mathbb{Z}[\omega] \\ \varpi \text{ prime} \\ \varpi \equiv u \pmod{v} \\ N(\lambda^{-3} \varpi) \leq X }} 
\rho_f(\lambda^{-3} \varpi) \log N(\varpi).
\end{equation}
where $0 \neq v \in \mathbb{Z}[\omega]$ is such that $v \equiv 0 \pmod{3}$, 
and $u \in \mathbb{Z}[\omega]/v \mathbb{Z}[\omega]$ is such that $(u,v)=1$ and $u \equiv 1 \pmod{3}$.
It is technically convenient to restrict attention to $u \equiv 1 \pmod{3}$. The other congruence classes 
modulo $3$ can be treated by a mild adaption of the methods of this paper.

\begin{theorem} \label{mainthm}
Let $\varepsilon>0$ and the notation be as above. Then
\begin{equation*}
\mathscr{P}_f(X,v,u;W_K) \ll_{\varepsilon,f} (XKN(v))^{\varepsilon} K^8 N(v)^{4} X^{1-1/34},
\end{equation*}
as $X \rightarrow \infty$.
\end{theorem}

\begin{corollary} \label{maincor}
Let $\varepsilon>0$. In the notation above we have
\begin{equation} \label{PXbd}
\mathscr{P}_f(X,v,u)  \ll_{\varepsilon,f,v}  X^{1-1/578+\varepsilon},
\end{equation}
and 
\begin{equation} \label{tildePXbd}
\widetilde{\mathscr{P}_f}(X,v,u)  \ll_{\varepsilon,f,v}  X^{1-1/578+\varepsilon},
\end{equation}
as $X \rightarrow \infty$.
\end{corollary}

Theorem \ref{mainthm} follows from new estimates for linear and bilinear sums which we now 
describe. A brief sketch of the new difficulties and ideas that arise in our case (as opposed to the case in
\cite{DI}) is given in \S \ref{sketch}.
Let
\begin{equation} \label{type1point}
\mathscr{T}_f(a,X,v,u;W_K):=\sum_{\substack{b \in \mathbb{Z}[\omega] \\  ab \equiv u \pmod{v}}} 
\rho_f(\lambda^{-3} ab) W_K \Big( \frac{N(\lambda^{-3} ab)}{X} \Big)
\end{equation}
denote the pointwise Type-I sum,
where $X \geq 10$ and  $a \in \mathbb{Z}[\omega]$ with $a \equiv 1 \pmod{3}$.
Let
\begin{equation} \label{type1avg}
\mathscr{A}_f(\boldsymbol{\alpha},X,v,u;W_K):=\mathop{\sum \sum}_{\substack{a,b \in \mathbb{Z}[\omega] \\ ab \equiv u \pmod{v}}}
\mu^2(a) \alpha_a \rho_f(\lambda^{-3} ab) W_K \Big( \frac{N(\lambda^{-3} ab)}{X} \Big)
\end{equation}
denote the average (over squarefree $a$) Type-I sum, where $A,X \geq 10$
and $\boldsymbol{\alpha}:=(\alpha_a)$ is
a $\mathbb{C}$-valued sequence
supported on $a \in \mathbb{Z}[\omega]$ with $a \equiv 1 \pmod{3}$ and $N(a) \asymp A$.
Let
\begin{equation} \label{type2}
\mathscr{B}_f(\boldsymbol{\alpha},\boldsymbol{\beta},X,v,u;W_K):=\mathop{\sum \sum}_{\substack{a,b \in \mathbb{Z}[\omega]  
\\ ab \equiv u \pmod{v} }}
\mu^2(a) \alpha_a \beta_b \rho_f (\lambda^{-3} ab) W_K \Big(\frac{N(\lambda^{-3} ab)}{X} \Big)
\end{equation}
denote the Type-II sum, where $A,B \geq 10$, $(\boldsymbol{\alpha}_a)$
is as above, and $\boldsymbol{\beta}:=(\beta_b)$
is a $\mathbb{C}$-valued sequence
supported on $b \in \mathbb{Z}[\omega]$ with $N(b) \asymp B$.
Note that we necessarily have $X \asymp AB$ in \eqref{type2}, otherwise the 
double sum is empty.

In \S \ref{type1sec} we use Voronoi summation to prove the following ``trivial" pointwise Type-I bound.
\begin{lemma} \label{trivpoint}
Let $\varepsilon>0$ and the notation be as above. Then
\begin{equation*}
\mathscr{T}_f(a,X;v,u;W_K) \ll_{\varepsilon,f}
(XKN(v))^{\varepsilon} K^{4} N(v)^{1/2} N(a)^{1/2}.
\end{equation*}
\end{lemma}
When $\mathscr{T}_f(a,\cdots)$ is multiplied by 
a weight $\alpha_a$ and the estimate in Lemma \ref{trivpoint} is summed trivially over $a \in \mathbb{Z}[\omega]$ with $N(a) \asymp A$,
the resulting bound is acceptable when $A \ll X^{2/3-\varepsilon}$.

In \S \ref{type2convsec} we use the circle method to prove the following new bilinear estimate.
\begin{theorem} \label{type2a}
Let $\varepsilon>0$ and the notation be as above. Then for $A,B \geq 10$ and $X \asymp AB$ we have
\begin{equation*}
\mathscr{B}_f(\boldsymbol{\alpha},\boldsymbol{\beta},X,v,u;W_K) \ll_{\varepsilon,f} (XKN(v))^{\varepsilon} K^8 N(v)^4 ((AB)^{1/2}+A^{3/2} B^{1/4})
\| \boldsymbol{\mu}^2 \boldsymbol{\alpha} \|_\infty \| \boldsymbol{\beta} \|_2.
\end{equation*}
\end{theorem}
Theorem \ref{type2a} is acceptable when $\|\boldsymbol{\mu^2} \boldsymbol{\alpha} \|_{\infty} \ll A^{\varepsilon} $ 
and $X^{2/3+\varepsilon} \ll B \ll X^{1-\varepsilon}$. 

We point out that
Lemma \ref{trivpoint} and Theorem \ref{type2a} together \emph{barely} misses primes.  
To overcome the bottleneck at level of distribution $\asymp X^{2/3}$,
we use Voronoi summation and Heath-Brown's cubic large sieve \cite{HB} to prove the following estimate.

\begin{prop} \label{type1avgbd}
Let $\varepsilon>0$ and the notation be as above. Then for 
$X, A \geq 10$ we have
\begin{equation*}
\mathscr{A}_f(\boldsymbol{\alpha},X,v,u;W_K) \ll_{\varepsilon,f} 
(XKN(v))^{\varepsilon} K^{14/3} N(v)^{17/6}  (AX)^{1/3}  \| \boldsymbol{\mu}^2 \boldsymbol{\alpha} \|_2.
\end{equation*}
\end{prop}

\subsection{Brief sketch of the method} \label{sketch}
We close with a brief outline of the proofs of Theorem \ref{type2a}
and Proposition \ref{type1avgbd}. For simplicity,
we suppress smooth functions, and ignore both the units of $\mathbb{Z}[\omega]$ and 
the congruence condition $u \pmod{v}$.

\subsubsection{Linear sums} 
We apply Voronoi summation to the $b$-sum in \eqref{type1avg} and perform
a computation with the arithmetic exponential sums that appear on the dual side. We obtain a bilinear form
\begin{equation} \label{type1bilinear}
\frac{X}{A^{2}}  \sum_{N(a) \asymp  A} \sum_{N(\nu) \ll A^2/X  }    \mu^2(a) \overline{g(a)} \alpha_a  \rho_f(\nu)
\Big( \frac{\nu}{a} \Big)_3, 
\end{equation}
where $g(a)$ denotes the un-normalised cubic Gauss sum over $\mathbb{Z}[\omega]$ with modulus $a$.
The use of Heath-Brown's cubic large sieve \cite{HB} (with the squarefree condition on one variable relaxed)
leads to our average Type-I estimate.

\subsubsection{Bilinear sums}
After application of Cauchy-Schwarz in the $b$-variable to \eqref{type2}, the sum of interest is
\begin{equation} \label{postcauchy}
\sum_{\substack{ N(a_1),N(a_2) \asymp A \\ a_1,a_2 \equiv 1 \pmod{3} }}
\mu^2(a_1) \alpha_{a_1} \mu^2(a_2) \overline{\alpha_{a_2}} \sum_{N(b) \sim B} \rho_f(a_1 b) \overline{\rho_f( a_2 b)}.
\end{equation}
The natural approach would be to ignore the averaging over $a_1$ and $a_2$,
and estimate each convolution sum $\sum_{N(b) \sim B} \rho_f(a_1 b) \overline{\rho_f(a_2 b)}$ directly.
Duke and Iwaniec \cite{DI}
proved that each convolution sum is
$\ll_{\varepsilon} \delta_{a_1=a_2} B+ (AB)^{\varepsilon} A B^{1/2}$
for the case of holomorphic half-integer weight cusp forms.
We explain below why the additional averaging over $a_1$ and $a_2$ is crucial
in the Maass case.

The initial move of \cite{DI} is to open one of the Fourier coefficients 
in terms of sums of half-integer weight Kloosterman sums. This comes about by expressing 
the holomorphic cusp form as a finite $\mathbb{C}$-linear combination of
Poincar\'{e} series.
This opening move is not available for 
Maass forms!
Instead, we separate oscillations using the circle method of Browning and Vishe \cite{BroVis}
to obtain
\begin{align} \label{circlexp}
\sum_{N(b) \sim B} \rho_f( a_1 b) \overline{\rho_f( a_2 b)}
\approx \frac{1}{B} \sum_{\substack{ N(\nu_1),N(\nu_2) \asymp AB \\  \nu_1 \equiv 0 \pmod{a_1} \\ \nu_2 \equiv 0 \pmod{a_2}  }}
\rho_f(\nu_1)  \overline{\rho_f(\nu_2)}
 \sum_{\substack{ N(c) \sim B^{1/2} \\ (c,\lambda a_1 a_2)=1 }} r(\nu_1/a_1-\nu_2/a_2,c),
 \end{align}
 where $r(n,c)$ denotes the un-normalised Ramanujan sum over $\mathbb{Z}[\omega]$ with modulus $c$ and shift $n$.
 In reality, one must also consider moduli $c$ that are not coprime to $\lambda a_1 a_2$. This
can be handled with an modification of the method below with an additional local computation involving cubic Gauss sums with moduli dividing
$\text{rad}(a_1 a_2)^{\infty}$.

We detect the congruence conditions on the $\nu_1,\nu_2$ using additive characters, apply Voronoi summation 
to each $\nu_1,\nu_2$ sum, and perform a considerable
computation with the exponential sums on the dual side.  This leads to an expression of the 
shape
\begin{equation}  \label{dualsummary}
\sum_{\substack{ s_1 \mid a_1 \\ s_2 \mid a_2 }}
\frac{1}{N(s_1 s_2)^{1/2}} \sum_{\substack{ N(\nu_1) \ll N(s_1)^2/A \\
 N(\nu_2) \ll  N(s_2)^2/A  \\ (\nu_1,s_1)=1 \\  (\nu_2,s_2)=1  }} \rho_f(\nu_1) \overline{\rho_f(\nu_2)} \sum_{\substack{ N(c) \sim B^{1/2} \\ (c,\lambda a_1 a_2)=1 }}
r(s_2^2 a_1 \nu_1 -s_1^2 a_2 \nu_2,c).
\end{equation}
We highlight that the squarefree property of $a_1$ and $a_2$ simplifies 
the computations considerably.
One can apply Cauchy-Schwarz
and Rankin-Selberg bounds to estimate the off-diagonal ($s_2^2 a_1 \nu_1 \neq s_1^2 a_2 \nu_2$)
contribution in \eqref{dualsummary} by $(AB)^{\varepsilon} A B^{1/2}$. The diagonal term is more subtle.
The diagonal equation  $s_2^2 a_1 \nu_1= s_1^2 a_2 \nu_2$  
is equivalent to  $s_2 (a_1/s_1) \nu_1=s_1 (a_2/s_2) \nu_2$. The conditions $(\nu_1,s_1)=(\nu_2,s_2)=1$ together with
the squarefree hypothesis on $a_1$ and $a_2$ imply that $s_1=s_2=:s \mid (a_1,a_2)$. Thus the diagonal contribution in \eqref{dualsummary} 
has the shape
\begin{equation} \label{diagreduceexp}
B \sum_{s \mid (a_1,a_2) }
\frac{1}{N(s)} \sum_{\substack{ N(\nu) \ll N(s)^{3} N((a_1/s,a_2/s))/A^2 \\ (\nu,s)=1 } } \rho_f \Big( \frac{a_2/s}{(a_1/s,a_2/s)} \nu \Big)
\overline{\rho_f \Big( \frac{a_1/s}{(a_1/s,a_2/s)} \nu  \Big)}.
 \end{equation}
 At this point there is no cancellation to be realistically exploited in \eqref{diagreduceexp},
 and so we apply the triangle inequality and place absolute values around
 the Fourier coefficients.
  It is tempting to 
 apply a ``Deligne-type" bound for $\rho_f(\cdot)$ to estimate the diagonal by $(AB)^{\varepsilon} \cdot B \cdot (N((a_1,a_2))/A)^2$ (which is of acceptable
 size).
 However, no such bound for $\rho_f(\cdot)$ is known unconditionally,
 and the author is not aware of \emph{any} non-trivial 
 bound for $\rho_f(\cdot)$ stronger than the bound implied by Rankin-Selberg. 
 There is no ``Waldspurger-type" formula known for the coefficients of cubic metaplectic cusp forms (on $\operatorname{GL}_2$).
Hence the strategy for bounding these Fourier coefficients via subconvexity for twisted $L$-values
 is not available (this strategy is used the half-integer weight case, see \cite{CI}).
 To overcome this, we substitute \eqref{diagreduceexp} into \eqref{postcauchy}, take absolute values and
the supremum norm of the $\boldsymbol{\alpha}$ terms,
and exploit the additional averaging over $a_1$ and $a_2$ using
Cauchy-Schwarz and Rankin-Selberg bounds.
This yields the acceptable estimate $(AB)^{\varepsilon} AB \| \boldsymbol{\mu}^2 \boldsymbol{\alpha} \|_{\infty}^2$
for the diagonal of the averaged sum. It is interesting to note that an argument of Nelson \cite{Nel} could potentially be adapted
to estimate the sparse convolution sum in \eqref{diagreduceexp}. We refrain from this additional work.

\subsection{Acknowledgements}
The author thanks both Ikuya Kaneko and Maksym Radziwi{\l\l} for making them aware of the reference \cite{BroVis} 
and for useful discussions. The author thanks the referee for helpful feedback on the manuscript. 

\subsection{Conventions}
For $n \in \mathbb{N}$ and $N>0$, we use $n \sim N$ to mean $N<n \leq 2N$,
and $n \asymp N$ to mean that there exists constants $c_1,c_2>0$ such that 
$c_1 N \leq n \leq c_2 N$.

Dependence of implied constants on parameters will be indicated in statements of results,
but suppressed throughout the body of the paper (i.e. proofs).
Implied constants in the body of the paper 
are allowed to depend on $f \in L^2(\Gamma_2 \backslash \mathbb{H}^3,\chi,\tau)$, 
$\varepsilon,D>0$ ($\varepsilon,D>0$ possibly different in each instance),
and the implicit constants in the statements $N(a) \asymp A$ and $N(b) \asymp B$.

Whenever we write $r \mid q$ with $0 \neq r,q \in \mathbb{Z}[\omega]$ and $q \equiv 1 \pmod{3}$,
it is our convention that $r \equiv 1 \pmod{3}$.
For any integer $b$ we let $\mathbb{Z}_{\geq b}:=\{n \in \mathbb{Z}: n \geq b \}$.

Unless otherwise specified, it should be clear from context whether $\overline{x}$ means 
modular inverse (with respect to an appropriate modulus) or complex conjugation.

Unless otherwise specified, it should be clear from context whether $v$ refers to 
the modulus of an arithmetic progression or the real component of a quaternion element $w=(z,v)$.

\section{Preliminaries and background}
\subsection{Eisenstein quadratic field and cubic Gauss sums}
We include some brief background on $\mathbb{Q}(\omega)$ and cubic Gauss sums. 
More details can be found in \cite{DR,Pat1,Pro}.

Let $\mathbb{Q}(\omega)$ be the Eisenstein quadratic number field,
where $\omega$ is identified with $e^{2\pi i/3} \in \mathbb{C}$.
This quadratic number field has ring of integers $\mathbb{Z}[\omega]$, discriminant $-3$,
and class number $1$. Let 
$N(x):=N_{\mathbb{Q}(\omega)/\mathbb{Q}}(x)=|x|^2$
denote the norm form on $\mathbb{Q}(\omega)/\mathbb{Q}$. 
The dual of $\mathbb{Z}[\omega]$ is
\begin{equation*}
\mathbb{Z}[\omega]^{*}:=
\{z \in \mathbb{C}: \check{e}(zz^{\prime})=1 \text{ for all } z^{\prime} \in \mathbb{Z}[\omega] \} =\lambda^{-1} \mathbb{Z}[\omega].
\end{equation*}
It is well known that any non-zero element of $\mathbb{Z}[\omega]$
can be uniquely 
written as $\zeta \lambda^{k} c$ with $\zeta \in \langle -\omega \rangle$ a unit (i.e. $\zeta^6=1$), 
$\lambda:=\sqrt{-3} = 1 + 2 \omega$ the unique ramified prime in $\mathbb{Z}[\omega]$, $k
\in \mathbb{Z}_{\geq 0}$, and $c \in \mathbb{Z}[\omega]$ with $c \equiv 1 \pmod{3}$. 
If $p \equiv 1 \pmod{3}$ is a rational prime, then 
$p=\varpi \overline{\varpi}$ in $\mathbb{Z}[\omega]$
with $N(\varpi)=p$ and $\varpi$ a prime in $\mathbb{Z}[\omega]$.
If $p \equiv 2 \pmod{3}$ is a rational prime, then $p=\varpi$ is inert in
$\mathbb{Z}[\omega]$, and $N(\varpi)=p^2$. Thus we have $N(\varpi) \equiv 1 \pmod{3}$
for all primes $\varpi$ with $(\varpi) \neq (\lambda)$.

The cubic Jacobi symbol defined for $a \equiv 1 \pmod{3}$ and $\varpi \equiv 1 \pmod{3}$ prime
is defined by 
$$
\Big ( \frac{a}{\varpi} \Big )_{3} \equiv a^{(N(\varpi) - 1) / 3} \pmod{\varpi},
$$
and the condition it take values in $\{0,1,\omega,\omega^2 \}$.
The cubic symbol is clearly multiplicative in $a$ and can be extended multiplicatively to all $b \equiv 1 \pmod{3}$ by setting
$
\big ( \frac{a}{b} \big )_{3} = \prod_{i} \big ( \frac{a}{\varpi_i} \big )
$
for any $b = \prod_{i} \varpi_i$ with $\varpi_i \equiv 1 \pmod{3}$ primes. 
The cubic symbol obeys cubic reciprocity: given $a,b \equiv 1 \pmod{3}$ 
we have 
\begin{equation} \label{cuberep}
\Big( \frac{a}{b} \Big)_3=\Big( \frac{b}{a} \Big)_3.
\end{equation}
There are also supplementary laws for units and the ramified prime.
Given 
\begin{equation} \label{supp1}
d \equiv 1 + \alpha_2 \lambda^2 + \alpha_3 \lambda^3 \pmod{9} \quad \text{with} \quad \alpha_2, \alpha_3 \in \{-1,0,1\},
\end{equation}
then
\begin{equation} \label{supp2}
\Big ( \frac{\omega}{d} \Big )_{3} = \omega^{\alpha_2} \quad \text{and} \quad \quad \Big ( \frac{\lambda}{d} \Big )_{3}
 = \omega^{-\alpha_3}.
\end{equation}
We follow the standard convention for an empty product,
\begin{equation} \label{empty}
\Big( \frac{a}{1} \Big)_3=1 \quad \text{for all} \quad a \in \mathbb{Z}[\omega].
\end{equation}

Let
\begin{equation*}
\check{e}(z):=e^{2 \pi i \text{Tr}_{\mathbb{C}/\mathbb{R}}(z)}
=e^{2 \pi i(z+\overline{z})}, \quad z \in \mathbb{C}.
\end{equation*}
For $\mu \in \mathbb{Z}[\omega]^{*}=\lambda^{-1} \mathbb{Z}[\omega]$
and $c \in \mathbb{Z}[\omega]$ with $c \equiv 1 \pmod{3}$, the cubic Gauss sum (with shift $\mu$) is defined by
\begin{equation} \label{generalgauss}
g(\mu,c):=\sum_{d \pmod{c}} \Big( \frac{d}{c} \Big)_3 \check{e} \Big( \frac{\mu d}{c} \Big).
\end{equation}
We write $g(c):=g(1,c)$ for short. 
Making a change of variable in the Gauss sum we see that
\begin{equation} \label{dualremark}
g(\mu,c):=\Big(\frac{\lambda }{c} \Big)_3 g(\lambda \mu,c),
\end{equation}
and so for the rest of this section it suffices to consider only $\mu \in \mathbb{Z}[\omega]$,
which we now assume.
We have 
\begin{equation} \label{coprimerel}
g(\mu,c)=\overline{\Big( \frac{\mu}{c} \Big)_3} g(1,c) \quad \text{for} \quad (\mu,c)=1.
\end{equation}
The Chinese Remainder Theorem implies the
twisted multiplicativity property
\begin{equation} \label{twistmult}
g(\mu,ab)=\overline{\Big( \frac{a}{b}  \Big)_3}
g(\mu,a) g(\mu,b), \quad \text{for } a,b \equiv 1 \pmod{3} \quad \text{such that} \quad (a,b)=1.
\end{equation}
By \eqref{coprimerel} and \eqref{twistmult} is suffices to 
understand $g(\varpi^{k},\varpi^{\ell})$ for $\varpi \equiv 1 \pmod{3}$ prime
and $k,\ell \in \mathbb{Z}_{\geq 0}$.
A specialisation of \cite[property (h), pg.~7]{Pro} yields
\begin{equation} \label{localcomp}
g(\varpi^{k},\varpi^{\ell})=\begin{cases}
1 & \text{if } \ell=0, \\
\varphi(\varpi^{\ell}) & \text{if } 1 \leq \ell \leq k, \quad \ell \equiv 0 \pmod{3} \\ 
-N(\varpi)^{k} & \text{if } \ell=k+1, \quad \ell \equiv 0 \pmod{3} \\
N(\varpi)^{k} g(\varpi) & \text{if } \ell=k+1, \quad \ell \equiv 1 \pmod{3} \\
N(\varpi)^{k} \overline{g(\varpi)} & \text{if } \ell=k+1, \quad \ell \equiv 2 \pmod{3} \\
0 & \text{otherwise}
\end{cases}.
\end{equation}

For $\varpi \equiv 1 \pmod{3}$ prime we have the formula for the cube \cite[pp. 443--445]{H},
\begin{equation} \label{cuberel}
g(\varpi)^3=-\varpi^2 \overline{\varpi}.
\end{equation}
Observe that \eqref{coprimerel}--\eqref{localcomp} and \eqref{cuberel}
imply that
\begin{equation} \label{sqrootcancel}
|g(c)|=\mu^2(c) N(c)^{1/2}
\end{equation}
for $c \equiv 1 \pmod{3}$.
We denote the normalised cubic Gauss sum (with shift $\mu \in \mathbb{Z}[\omega]$) by 
\begin{equation} \label{normalised}
\widetilde{g}(\mu,c):=N(c)^{-1/2} g(\mu,c).
\end{equation}

The following two lemmas follow directly from combining \eqref{coprimerel}--\eqref{localcomp}.

\begin{lemma} \label{elemlem}
Let $\mu \in \mathbb{Z}[\omega]$ and
$c \in \mathbb{Z}[\omega]$ such that $c \equiv 1 \pmod{3}$ is squarefree. 
Then
\begin{equation*}
g(\mu,c)=0 \quad \emph{unless} \quad (\mu,c)=1.
\end{equation*}
\end{lemma} 

\begin{lemma} \label{sqsupport}
Let $c \in \mathbb{Z}[\omega]$ with $c \equiv 1 \pmod{3}$ and $\varpi,\mu \in \mathbb{Z}[\omega]$ be such that
$\varpi \equiv 1 \pmod{3}$ is prime and $\varpi^2 \mid c$.
Then
\begin{equation*}
g(\mu,c)=0 \quad \emph{unless} \quad \varpi \mid \mu.
\end{equation*}
\end{lemma}

The next Lemma follows directly from combining \eqref{coprimerel}--\eqref{localcomp} 
and \eqref{sqrootcancel}.

\begin{lemma} \label{sqrootbd}
Let $\mu,c \in \mathbb{Z}[\omega]$ with $c \equiv 1 \pmod{3}$. Then
\begin{equation*}
|g(\mu,c)| \leq N(c)^{1/2} \cdot N((\mu,c))^{1/2}.
\end{equation*}
\end{lemma}

\begin{remark}
We emphasise that $c \in \mathbb{Z}[\omega]$ is not necessarily squarefree in Lemma \ref{sqrootbd}.
\end{remark}

For $b \in \mathbb{R}$, and $q \in \mathbb{Z}[\omega]$ with $q \equiv 1 \pmod{3}$, let 
\begin{equation} \label{sigmabq}
\sigma_b(q):=\sum_{d \mid q} N(d)^{b}.
\end{equation}
For a given
$\varepsilon>0$, we have the
standard divisor bound
\begin{equation} \label{divis}
\sigma_0(q) \ll_{\varepsilon} N(q)^{\varepsilon}.
\end{equation}
The following Lemma is immediate.

\begin{lemma} \label{gcdflat}
Let $q \in \mathbb{Z}[\omega]$ with $q \equiv 1 \pmod{3}$ and $b \in \mathbb{R}$. Then for $Y \geq 1$ we have
\begin{equation*}
\sum_{\substack{\mu \in \mathbb{Z}[\omega] \\ 1 \leq N(\mu) \leq Y}} N((\mu,q))^{b} \ll Y \sigma_{b-1}(q),
\end{equation*}
$\sigma_{b}(q)$ is given in \eqref{sigmabq}.
\end{lemma}

\begin{lemma} \label{radlemma}
Let $q \in \mathbb{Z}[\omega]$ with $q \equiv 1 \pmod{3}$. Then
for $X \geq 1$ and $\varepsilon>0$ we have 
\begin{equation}
\sum_{\substack{ N(r) \leq X \\ r \mid \emph{rad}(q)^{\infty} }} 1 \ll_{\varepsilon} (N(q) X)^{\varepsilon}.
\end{equation}
\end{lemma}
\begin{proof}
Without loss of generality we can assume $X$ is an odd half-integer.
By Perron's formula (truncated) we have 
\begin{equation*}
\sum_{\substack{ N(r) \leq X \\ r \mid \text{rad}(q)^{\infty} }} 1 =\frac{1}{2\pi i}\int_{2-i (XN(q))^{100}}^{2+i (XN(q))^{100}} X^{s} \prod_{\substack{ \varpi \mid \text{rad}(q) \\ \varpi \text{ prime} \\ \varpi \equiv 1 \pmod{3} }} (1-N(\varpi)^{-s})^{-1} \frac{ds}{s} +O((XN(q))^{-50}).
\end{equation*}
The integrand is holomorphic in the half-plane $\Re(s)>0$.
We move the contour $\Re(s)=\varepsilon$. 
Taking the logarithm of the Euler product and then using 
the pointwise bound
\begin{equation} \label{distinctprime}
\omega(q) \ll \frac{\log N(q)}{\log \log N(q)},
\end{equation}
we obtain (after exponentiation)
\begin{equation*}
\prod_{\substack{ \varpi \mid \text{rad}(q) \\ \varpi \text{ prime} \\ \varpi \equiv 1 \pmod{3} }} (1-N(\varpi)^{-s})^{-1} \ll N(q)^{\varepsilon} \quad \text{for} 
\quad \Re(s) \geq \varepsilon.
\end{equation*}
The result follows from Cauchy's Residue Theorem.
\end{proof}

\subsection{Group action on $\mathbb{H}^3$ and Laplacian}
Let $\mathbb{H}^{3}$ denote the hyperbolic 3-space $\mathbb{C} \times \mathbb{R}_{>0}$.
Embed $\mathbb{C}$ and $\mathbb{H}^3$ in the Hamilton quaternions  by identifying 
$i=\sqrt{-1}$ with $\hat{i}$ and 
$w=(z,v)=(x+iy,v) \in \mathbb{H}^3$ with $x+y \hat{i}+v \hat{k}$, where 
$1,\hat{i},\hat{j},\hat{k}$ denote the unit quaternions.  
The continuous action 
of $\operatorname{SL}_2(\mathbb{C})$ on $\mathbb{H}^3$ (in quaternion arithmetic)
is given by
\begin{equation*}
\gamma w=(aw+b)(cw+d)^{-1}, \quad \gamma=\begin{pMatrix}
a b
c d 
\end{pMatrix} \in \operatorname{SL}_2(\mathbb{C}) \quad \text{and} \quad
w \in \mathbb{H}^3.
\end{equation*}
The action of $\operatorname{SL}_2(\mathbb{C})$ on $\mathbb{H}^3$ is transitive,
and the stabiliser of a point is $\operatorname{SU}_2(\mathbb{C})$.
In coordinates,
\begin{equation} \label{coordinates}
\gamma w= \bigg( \frac{(az+b) \overline{(cz+d)}+a \overline{c} v^2}{|cz+d|^2 +|c|^2 v^2}, 
\frac{v}{|cz+d|^2+|c|^2 v^2} \bigg),
\quad w=(z,v).
\end{equation}
The Laplace operator 
$\Delta:=v^2 ( \partial^2/\partial x^2+
\partial^2/\partial y^2+\partial^2/\partial v^2)-v \partial/\partial v$
acts on $C^{\infty}(\mathbb{H}^3)$ and commutes with the
action of $\operatorname{SL}_2(\mathbb{C})$ on $C^{\infty}(\mathbb{H}^3)$. 

Consider the subgroup $\Gamma:=\operatorname{SL}_2(\mathbb{Z}[\omega])$  
of $\operatorname{SL}_2(\mathbb{C})$. It has finite volume (but is not co-compact)
with respect to the 
$\operatorname{SL}_2(\mathbb{C})$-invariant Haar measure $v^{-3} dx dy dv$ 
on $\mathbb{H}^3$. In what follows, let $\Gamma^{\prime} \subseteq \Gamma$ be a subgroup
with $[\Gamma:\Gamma^{\prime}]<\infty$. Let $P(\Gamma^{\prime}) \subset \mathbb{Q}(\omega) \cup \{\infty \}$
be a complete inequivalent (finite) set of cusps for $\Gamma^{\prime}$. Each cusp of $\Gamma^{\prime}$ can be written as
$\sigma \infty$ for some $\sigma \in \Gamma$, and let
\begin{equation*}
\Gamma^{\prime}_{\sigma}:=\{\gamma \in \Gamma^{\prime}: \gamma \sigma \infty=\sigma \infty \},
\end{equation*}
denote the stabiliser group of the cusp $\sigma \infty$ in $\Gamma^{\prime}$. 
We have $\Gamma^{\prime}_{\sigma}:=\sigma \Gamma^{\prime}_I \sigma^{-1} \cap \Gamma^{\prime}$,
and let
\begin{align*}
\Lambda_{\sigma}&:=\{ \mu \in \mathbb{C}: \sigma \pmatrix 1 {\mu} 0 1 \sigma^{-1} \in \Gamma^{\prime} \}; \\
\Lambda_{\sigma}^{*}&:=\{\nu \in \mathbb{C}: \operatorname{Tr}(\mu \nu) \in \mathbb{Z} \quad \text{for all} \quad \mu \in \Lambda_{\sigma} \}.
\end{align*}
It is well known that $\Lambda_{\sigma}$ and $\Lambda_{\sigma}^{*}$ are lattices in $\mathbb{C}$,
and that $\Lambda_{\sigma}^{*}$ is dual to $\Lambda_{\sigma}$. 

A fundamental domain for the action of $\Gamma$ 
on $\mathbb{H}^3$ is the set
\begin{equation*}
\mathcal{F}:=\{w=(z,v) \in \mathbb{H}^3: |z|^2+v^2>1 \quad \text{and} \quad z \in \pm \triangle \},
\end{equation*}
where $\triangle$ is the interior of the triangle with vertices $0$, $(1-\omega)^{-1}$ 
and $(1-\omega^2)^{-1}$. The set of cusps for $\Gamma$ is $P(\Gamma):=\{\infty\}$.

Other congruence subgroups of significance to this paper are given in \S \ref{cubicsec}.

\subsection{Automorphic forms (for general multipliers)}
We record some facts about automorphic forms on $\Gamma^{\prime} \backslash \mathbb{H}^3$
that transform with general unitary character $\kappa:\Gamma^{\prime} \rightarrow \mathbb{C}^{\times}$. 
For more details one may consult \cite{LP, Lo,Pat3,Pro}. We specialise to
cubic metaplectic forms in \S \ref{cubicsec}.

Let $\kappa: \Gamma^{\prime}
\rightarrow \mathbb{C}^{\times}$ be a unitary character that satisfies 
$\kappa(-I)=1$ if $-I \in \Gamma^{\prime}$.
The function defined by $\mu \rightarrow \kappa(\sigma \pmatrix 
1 {\mu} 0 1 \sigma^{-1}):\Lambda_{\sigma} \rightarrow 
\mathbb{C}^{\times}$ is a homomorphism on the lattice $\Lambda_{\sigma}$. There exists $h_{\sigma} \in \mathbb{C}$ such that
\begin{equation*}
 \kappa(\sigma \pmatrix 
1 {\mu} 0 1 \sigma^{-1})=\check{e}(h_{\sigma} \mu) \quad \text{for all} \quad \mu \in \Lambda_{\sigma}.
\end{equation*} 
Essential cusps with respect to $\kappa$ are those $\sigma$ for which we can take $h_{\sigma}=0$.

Let 
\begin{equation*}
A(\Gamma^{\prime}  \backslash \mathbb{H}^3,\kappa)
:=\{u :\mathbb{H}^3 \rightarrow \mathbb{C} : u(\gamma w)=\kappa(\gamma) u(w) 
\text{ for all } \gamma \in \Gamma^{\prime} \text{ and } w \in \mathbb{H}^3 \}.
\end{equation*}
We say that $u \in A(\Gamma^{\prime} \backslash \mathbb{H}^3,\kappa)$ is an automorphic form 
under $\Gamma^{\prime}$ with character $\kappa$ if it satisfies the conditions:
\begin{itemize}
\item $u \in C^{\infty}(\mathbb{H}^3)$ and is an eigenfunction of the Laplacian i.e.
\begin{equation*} 
\Delta u=-\tau_u(2-\tau_u) u \quad \text{for some} \quad \tau_u \in \mathbb{C}.
\end{equation*}
The quantity $\tau_u \in \mathbb{C}$ is the spectral parameter for $u$, 
and is well-defined only up to $\tau_u \mapsto 2-\tau_u$. Without loss of generality 
one can assume that $\Re(\tau_u) \geq 1$.
\item $u$ has moderate growth at cusps: there exists a $D \in \mathbb{R}$ such that 
\begin{equation*}
|u(w)| < (v+(1+|z|^2) v^{-1} )^{D} \quad \text{for all} \quad w=(z,v) \in \mathbb{H}^3.
\end{equation*}
\end{itemize}
Let $L(\Gamma^{\prime} \backslash \mathbb{H}^3,\kappa,\tau)$ denote the $\mathbb{C}$-vector space
of automorphic forms under $\Gamma^{\prime}$ with character $\kappa$ and spectral parameter $\tau$.
The norm $\| \cdot \|_2$ on $L(\Gamma^{\prime} \backslash \mathbb{H}^3,\kappa,\tau)$ 
is induced by the standard Petersson inner product
\begin{equation*}
\langle u_1,u_2 \rangle:=
\int_{\Gamma^{\prime} \backslash \mathbb{H}^3} u_1(w) \overline{u_2(w)} \frac{dx dy dv}{v^3}.
\end{equation*}
Let 
\begin{equation*}
L^2(\Gamma^{\prime} \backslash \mathbb{H}^3,\kappa,\tau):=\{u \in L(\Gamma^{\prime} \backslash \mathbb{H}^3,\kappa,\tau)
: \|u\|_2<\infty \},
\end{equation*}
denote the finite dimensional Hilbert space 
of square integrable automorphic forms having character $\kappa$ and spectral parameter $\tau$.
We demand that
$\left( \begin{smallmatrix}
{-1} &  0 \\
0 &  1
\end{smallmatrix} \right)$ act on an automorphic form $u$ by $\pm 1$, and we speak of 
$u$ being even or odd respectively.

Consulting \cite[Theorem~0.3.1]{Pro}, each $u \in L(\Gamma^{\prime} \backslash \mathbb{H}^3,\kappa,\tau)$ 
has Fourier expansion at the cusp $\sigma \infty$ given by 
\begin{equation} \label{fouriersigma}
U_{\sigma}(w):=u(\sigma w)=c_{u,\sigma}(v)+\sum_{\substack{\nu \neq 0 \\ \nu \in h_{\sigma}+\Lambda_{\sigma}^{*} }}
\rho_{u,\sigma}(\nu) v K_{\tau-1}(4 \pi |\nu| v) \check{e}(\nu z), \quad w \in \mathbb{H}^3,
\end{equation}
where $\rho_{u,\sigma}(\nu) \in \mathbb{C}$, and 
\begin{equation*}
c_{u,\sigma}(v)=\begin{cases}
\rho_{u,\sigma,+}(0) v^{\tau}+\rho_{u,\sigma,-}(0) v^{2-\tau} & \text{if} \quad \tau \neq 1 \\
\rho^{\sigma}_{u,\sigma,+}(0) v \log v+\rho_{u,\sigma,-}(0) v & \text{if} \quad \tau=1,
\end{cases}
\end{equation*}
and $\rho_{u,\sigma,+}(0), \rho_{u,\sigma,-}(0) \in \mathbb{C}$.
If $\sigma \infty$ is essential, then one can take $h_{\sigma}=0$.
If $\sigma \infty$ is not essential, then $c_{u,\sigma}(v) \equiv 0$ by \cite[Theorem~0.3.1]{Pro}.
By convention, if $\sigma=I$ then we omit it from the subscripts on
the Fourier coefficients.

If $c_{u,\sigma}(v) \equiv 0$ for all cusps $\sigma \infty$, then $u$ is a cusp form (it is necessarily a Maass form
since $\mathbb{H}^3$ does not have an invariant complex structure).
In particular, all cusp forms have exponential decay at the cusps, 
and consequently are square integrable on $\Gamma^{\prime} \backslash \mathbb{H}^3$.

The following crude Rankin-Selberg bound follows from a standard argument that uses Plancherel's theorem.
The proof is analogous to that of \cite[Theorem~3.2]{IwaSpec}, and is omitted.

\begin{lemma} \label{rankinselbound}
Let $\tau \in \mathbb{C}$ with $\Re(\tau) \geq 1$,
$u \in L^2(\Gamma^{\prime} \backslash \mathbb{H}^3,\kappa,\tau)$ be a cusp form, $\sigma$ a cusp of $\Gamma^{\prime}$, and $\varepsilon>0$. Then
for all $X \geq 100$ we have 
\begin{equation*}   
\sum_{\substack{\nu \in h_{\sigma} + \Lambda^{*}_{\sigma} \\ N(\nu) \leq X }} 
|\rho_{u,\sigma}(\nu)|^2 \ll_{u,\sigma,\varepsilon} X^{1+\varepsilon}.
\end{equation*}
\end{lemma}

An application of the Cauchy-Schwarz inequality and Lemma \ref{rankinselbound} give the following $L^1$-bound.
\begin{lemma} \label{L1bound}
In the notation of Lemma \ref{rankinselbound} we have 
\begin{equation*}   
\sum_{\substack{\nu \in h_{\sigma} + \Lambda^{*}_{\sigma} \\ N(\nu) \leq X }} 
|\rho_{u,\sigma}(\nu)| \ll_{u,\sigma,\varepsilon} X^{1+\varepsilon}.
\end{equation*}
\end{lemma}

The following Wilton-type bound follows from a standard argument using
Fourier convolution with the Dirichlet kernel. 
The proof is analogous to that of \cite[Theorem~3.1]{EHS}, and is omitted.

\begin{lemma} \label{wiltonbound}
Let the notation be as in Lemma \ref{rankinselbound} and suppose that $\Re(\tau)=1$.
Then
\begin{equation*}
\sum_{\substack{\nu \in h_{\sigma}+\Lambda^{*}_{\sigma} \\ N(\nu) \leq X }} 
\rho_{u,\sigma}(\nu) \check{e}(\alpha \nu) \ll_{u,\sigma,\varepsilon} X^{1/2+\varepsilon},
\end{equation*}
for any $\alpha \in \mathbb{C}$.
The implied constant is uniform with respect to $\alpha$.
\end{lemma}

A direct consequence of partial summation and Lemma \ref{wiltonbound} is the following smoothed Wilton bound. 

\begin{lemma} \label{smoothwilton}
Let the notation be as in Lemma \ref{rankinselbound}, $\Re(\tau)=1$, $K,M \geq 1$, and $W_{K,M}:(0,\infty) \rightarrow \mathbb{C}$ be a smooth function with compact support in $[1,2]$ 
that satisfies \eqref{Kderivativebound}. Then 
 \begin{equation*} 
\sum_{\substack{0\neq\nu \in h_{\sigma}+\Lambda^{*}_{\sigma}}}
\rho_{u,\sigma}(\nu) \check{e}(\alpha \nu) W_{K,M} \Big( \frac{N(\nu)}{X} \Big) \ll_{u,\sigma,\varepsilon} M K X^{1/2+\varepsilon},
\end{equation*}
for any $\alpha \in \mathbb{C}$.
The implied constant is uniform with respect to $\alpha$.
\end{lemma}

\section{Cubic metaplectic forms} \label{cubicsec}
\subsection{Cubic Kubota character}
Recall that $\Gamma:=\operatorname{SL}_2(\mathbb{Z}[\omega])$.
It is well known that $\Gamma=\langle P,T,E \rangle $, where
\begin{equation*}
P:=\begin{pMatrix}
\omega  0 
0  {\omega^2}
\end{pMatrix},
\quad 
T:=\begin{pMatrix}
1  1 
0  1
\end{pMatrix}, 
\quad 
E:=\begin{pMatrix}
0 {-1}
1  0
\end{pMatrix}.
\end{equation*}
Let $0 \neq C \in \mathbb{Z}[\omega]$
satisfy $C \equiv 0 \pmod{3}$, and
\begin{equation*}
\Gamma_1(C):=\{ \gamma \in \Gamma : \gamma \equiv I \pmod{C}  \}.
\end{equation*}
Observe that $\Gamma_1(C)$ is a normal subgroup of $\Gamma$ since it is the kernel of the 
reduction modulo $C$ map. Let 
\begin{equation} \label{Gamma2}
\Gamma_2:=\langle \operatorname{SL}_2(\mathbb{Z}), \Gamma_1(3) \rangle=\operatorname{SL}_2(\mathbb{Z}) \Gamma_1(3)=\Gamma_1(3)
\operatorname{SL}_2(\mathbb{Z}),
\end{equation}
where the last two equalities follow because 
$\Gamma_1(3)$ is normal in $\Gamma$. 
We also have $[\Gamma:\Gamma_2]=27$ (see \cite[\S 2]{Pat1} for the calculation).
Recall that $\chi: \Gamma_2 \rightarrow \{1,\omega, \omega^2 \}$
is the cubic Kubota character defined in \S \ref{introsec}.
The cusps of $\Gamma_2$ are 
$P(\Gamma_2)=\{\infty,\omega,\omega^2 \}$, and the only essential cusp
of $\Gamma_2$ with respect to $\chi$ is $\infty$. 

\subsection{Cubic Shimura lift} \label{cubeshim}
Suppose $\Gamma^{\prime} \subseteq \Gamma_2$ is a subgroup with 
$[\Gamma_2: \Gamma^{\prime}]<\infty$. If $h \in L(\Gamma^{\prime} \backslash 
\mathbb{H}^3,\chi,\tau)$, then $h$ is said to be a cubic metaplectic form on $\Gamma^{\prime}$
with spectral parameter $\tau$ (abbreviated to cubic metaplectic form).
In this section we specialise to the lowest possible level $\Gamma^{\prime}=\Gamma_2$,
and focus on the finite dimensional subspace 
$L^2(\Gamma_2 \backslash \mathbb{H}^3,\chi,\tau) \subset L(\Gamma_2 \backslash 
\mathbb{H}^3,\chi,\tau)$ that contains square integrable cubic metaplectic forms.

We say that $h \in L^2(\Gamma_2 \backslash \mathbb{H}^3,\chi,\tau)$ 
is a Hecke eigenform if it is an eigenfunction for all Hecke operators $\{T_{\nu^3}\}_{\nu \in \mathbb{Z}[\omega] \setminus \{0\} }$
i.e. $T_{\nu^3} h=\widetilde{\lambda}_h(\nu^3) h$ for some $\widetilde{\lambda}_h(\nu^3) \in \mathbb{C}$ and
all $\nu \in \mathbb{Z}[\omega] \setminus \{0\}$.
There is an orthonormal basis (with respect to the Petersson inner product) of 
$L^2(\Gamma_2 \backslash \mathbb{H}^3,\chi,\tau)$ consisting of Hecke eigenforms.
Two automorphic forms are identified with each other if they are constant multiples of one another.
The discrete spectrum of $\Delta$ on $L^2(\Gamma_2 \backslash \mathbb{H}^3,\chi)$
is completely determined via the cubic Shimura correspondence of Flicker \cite{Flick} and Patterson \cite[Theorem~3.4]{Pat3}.
In particular, there is a bijective correspondence
between even (resp. odd) Hecke eigenforms 
$h \in L^2(\Gamma_2 \backslash \mathbb{H}^3,\chi,\tau)$ and even (resp. odd) 
Hecke eigenforms $g \in  L^2(\Gamma \backslash \mathbb{H}^3,\mathbf{1},3 \tau-2)$,
where in the latter case the Hecke operators are the standard ones $\{\mathcal{T}_{\nu} \}_{\nu \in \mathbb{Z}[\omega] \setminus \{0\}}$ 
on the trivial cover of $\Gamma$ i.e. $\mathcal{T}_{\nu} g=\lambda_g(\nu) g$ for some $\lambda_g(\nu) \in \mathbb{C}$ and
all $\nu \in \mathbb{Z}[\omega] \setminus \{0\}$. 
Under this correspondence
one also has
\begin{equation*}
N(\nu^3)^{-1/2} \widetilde{\lambda}_h(\nu^3)=N(\nu)^{-1/2} \lambda_g(\nu).
\end{equation*}

The only non-cuspidal Hecke eigenform in $L^2(\Gamma_2 \backslash \mathbb{H}^3,\chi)$
is the cubic theta function of Patterson \cite{Pat1},
\begin{equation*}
\vartheta_3(w):=\text{Res}_{s=4/3} E_3(w,s) \in L^2 (\Gamma_2 \backslash \mathbb{H}^3,\chi,4/3),
\end{equation*}
where $E_3(w,s)$ is the Kubota cubic Eisenstein series for $\Gamma_1(3)$ 
at the cusp $\infty$. Its Shimura correspondent 
is the constant function $1 \in L^2(\Gamma \backslash \mathbb{H}^3,\mathbf{1},2)$.
The countably many other Hecke eigenforms $h_k \in L^2(\Gamma_2 \backslash \mathbb{H}^3,\chi)$
are Maass cusp forms, whose Shimura correspondents
$g_k \in L^2(\Gamma \backslash \mathbb{H}^3,\mathbf{1})$ are also Maass cusp forms.
All spectral parameters are non-exceptional i.e. $\Re(\tau_{h_k})=\Re(\tau_{g_k})=1$ for $k=1,2,\ldots$. We also have
$0 \leq \Im(\tau_{f_1}) \leq \Im(\tau_{f_2}) \leq \ldots$,
where $\Im(\tau_{h_k}) \rightarrow \infty$ as $k \rightarrow \infty$. 

\subsection{Cubic Kloosterman sums}
We will encounter cubic Kloosterman sums attached to the cubic Kubota character in our computations.

Let $\Gamma^{\prime} \subseteq \Gamma_2$ with $[\Gamma_2: \Gamma^{\prime}]<\infty$, and
let $\sigma,\xi \in \operatorname{SL}_2(\mathbb{Z}[\omega])$ denote 
cusps of $\Gamma^{\prime}$.  Let 
\begin{equation*}
\mathcal{C}(\sigma,\xi):=\big \{c \in \mathbb{Z}[\omega] \setminus \{0\}:  \sigma \pmatrix {*} {*} c {*} \xi^{-1} \in \Gamma^{\prime}  \big \}
\end{equation*}
be the set of allowable moduli for the cusp pair $(\sigma, \xi)$.
For $m \in \Lambda_{\sigma}^{*}$, $n \in \Lambda_{\xi}^{*}$,
and $c \in \mathcal{C}(\sigma,\xi)$, the cubic Kloosterman sum is
\begin{equation} \label{cubickloosterman}
K_{\Gamma^{\prime},\sigma,\xi}(m,n,c):=\sum_{\substack{a \hspace{-0.15cm} \pmod{c \Lambda_{\sigma}} \\ d \hspace{-0.15cm}
\pmod{c \Lambda_{\xi}} \\ \sigma \pmatrix a {*} c d \xi^{-1} \in \Gamma^{\prime} }} \overline{\chi \Big(\sigma \pMatrix a * c d \xi^{-1} \Big)} \check{e} \Big(\frac{ma+nd}{c} \Big),
\end{equation}
where $\chi: \Gamma_2 \rightarrow \{1,\omega,\omega^2\}$ is the cubic Kubota character.
We have the following Weil bound \cite{Weil}.

\begin{lemma}{\emph{\cite[Proposition~5.1]{LP} and \cite[(2.6)]{Lo}}} \label{weilbd}
Let the notation be as above. Then for $m,n \in \mathbb{Z}[\omega]$ and $c \in \mathcal{C}(\sigma,\xi)$,
we have 
\begin{equation*} 
|K_{\Gamma^{\prime},\sigma,\xi}(m,n,c)| \leq 2^{\omega(c)} N((m,n,c)) N(c)^{1/2}, 
\end{equation*}
where $\omega(c)$ denotes the number of distinct prime divisors of $c$.
\end{lemma}

\begin{remark}
In \cite[Proposition~5.1]{LP} (and propagated in \cite[\S 2]{Lo}), it appears the bound
in Lemma \ref{weilbd} is stated sub-optimally with a factor $N((m,n,c))$ instead of
$N((m,n,c))^{1/2}$. This makes no difference to us because $(m,n,c)=1$
in any instance when Lemma \ref{weilbd} is used in this paper.
\end{remark}

\begin{lemma} \label{kloostermanexp1}
Suppose that $\Gamma^{\prime}=\Gamma_1(3)$ and $\sigma=\pmatrix 1 0 0 1$.
Then
\begin{equation}
K_{\Gamma_1(3),\sigma,\sigma}(m,n,c)=\sum_{\substack{a,d \hspace{-0.15cm} \pmod{3c} \\ a,d \equiv 1 \pmod{3} \\ ad \equiv 1 \pmod{3c}}} \Big( \frac{c}{d} \Big)_3  \check{e} \Big(\frac{ma+nd}{c} \Big),
\end{equation}
for any $c \in 3 \mathbb{Z}[\omega] \setminus \{0\}$, and $m,n \in \lambda^{-3} \mathbb{Z}[\omega]$.
\end{lemma}

\begin{proof}
Observe that
$\Lambda_{\sigma}=3 \mathbb{Z}[\omega]$,
$\Lambda^{*}_{\sigma}=3^{-1} \mathbb{Z}[\omega]^*=\lambda^{-3} \mathbb{Z}[\omega]$,
and $\mathcal{C}(\sigma,\sigma)=3 \mathbb{Z}[\omega] \setminus \{0\}$.
Observe that $\gamma=\pmatrix a b c d \in \Gamma_1(3)$ if and only if
$a \equiv d \equiv 1 \pmod{3}$, $b,c \equiv 0 \pmod{3}$, and $ad-bc=1$. 
For $\gamma=\pmatrix a b c d \in \Gamma_1(3)$ with $c \neq 0$ we have 
$\chi(\gamma)=(c/a)_3$ by \eqref{cubekubota}.
The condition $b=(ad-1)/c\equiv0\pmod3$ is equivalent to $ad\equiv1\pmod{3c}$. The claim now follows from \eqref{cubickloosterman}, \eqref{supp1}, and \eqref{supp2}.
\end{proof}

\begin{lemma} \label{kloostermanexp2}
Suppose that $\Gamma^{\prime}=\Gamma_1(3)$, $\sigma=\pmatrix 1 0 0 1$,
and $\xi=\pmatrix 0 {-1} 1 0$.
Then
\begin{equation}
K_{\Gamma_1(3),\sigma,\xi}(m,n,c)=\sum_{\substack{a,d  \hspace{-0.15cm} \pmod{3c} \\ a,d \equiv 0 \pmod{3} \\ ad \equiv 1 \pmod{c}}} \Big( \frac{d}{c} \Big)_3 \check{e} \Big(\frac{ma+nd}{c} \Big),
\end{equation}
for any $c \in \mathbb{Z}[\omega]$ such that $c \equiv 1 \pmod{3}$, and $m,n \in \lambda^{-3} \mathbb{Z}[\omega]$.
\end{lemma}

\begin{proof}
Observe that
$\Lambda_{\sigma}=\Lambda_{\xi}=3 \mathbb{Z}[\omega]$ and that
$\Lambda^{*}_{\xi}=\Lambda^{*}_{\sigma}=3^{-1} \mathbb{Z}[\omega]^*=\lambda^{-3} \mathbb{Z}[\omega]$.  Let 
$\gamma=\pmatrix a b c d \in \operatorname{SL}_2(\mathbb{Z}[\omega])$. Observe that $\sigma \gamma \xi^{-1} \in \Gamma_1(3)$ if and only if
$a \equiv d \equiv 0 \pmod{3}$, $c \equiv 1 \pmod{3}$, $b \equiv -1 \pmod{3}$, and $ad-bc=1$.
After recalling that $\chi$ is homomorphism on $\Gamma_2$ such that $\chi \vert_{\operatorname{SL}_2(\mathbb{Z})} \equiv 1$,
we see that $\chi(\gamma \xi^{-1})=\chi(\xi^{-1} \gamma)=(-a/c)_3=(a/c)_3$ by \eqref{cubekubota}
and the convention \eqref{empty}.
The claim now follows from \eqref{cubickloosterman}.
\end{proof}

\section{The cubic large sieve}
Implicit in the work of Heath--Brown \cite{HB} is a version of cubic large sieve 
where one of the variables is not required to be squarefree. Here we record the relevant results.

\begin{theorem} \label{HBcubic} 
Let $\varepsilon>0$ be given, $M,N \geq 1/2$ and $\boldsymbol{\Psi}=(\Psi_c)$ be
a $\mathbb{C}$-valued sequence supported on $c \in \mathbb{Z}[\omega]$ with $c \equiv 1 \pmod{3}$
and $N(c) \sim N$.
Then
\begin{equation*}
\sum_{N(d) \sim M} \bigg | \sum_{\substack{N(c) \sim N \\ c \equiv 1 \pmod{3} }} \mu^2(c) \Psi_c \Big(\frac{c}{d} \Big)_3 \bigg |^2
\ll_{\varepsilon} (MN)^{\varepsilon} M^{1/3} (M+N) \| \boldsymbol{\mu}^2 \boldsymbol{\Psi} \|_2^2. 
\end{equation*}
\end{theorem}

\begin{proof}
This follows from \cite[(22)]{HB} (and the display above it), \cite[(28)]{HB},
and the second display on \cite[pg.~123]{HB}.
\end{proof}

\begin{corollary} \label{cubiccor}
Let the notation be as in Theorem \ref{HBcubic} and $\boldsymbol{\Omega}=(\Omega_d)$ be a $\mathbb{C}$-valued sequence supported on
$d \in \mathbb{Z}[\omega]$ with $N(d) \sim M$.
Then
\begin{equation*}
\sum_{N(d) \sim M} \sum_{\substack{N(c) \sim N \\ c \equiv 1 \pmod{3}}}
\Omega_d \mu^2(c) \Psi_{c} \Big( \frac{d}{c} \Big)_3 
\ll_{\varepsilon} (MN)^{\varepsilon} M^{1/6}  (M^{1/2}+N^{1/2}) 
\| \boldsymbol{\Omega} \|_2 \| \boldsymbol{\mu}^2 \boldsymbol{\Psi} \|_2.
\end{equation*}
\end{corollary}

\begin{proof}
Application of the Cauchy--Schwarz inequality, unique factorisation in $\mathbb{Z}[\omega]$, 
\eqref{cuberep}, 
and Theorem \ref{HBcubic} gives
\begin{align*} 
& \Big | \sum_{N(d) \sim M} \sum_{\substack{N(c) \sim N \\ c \equiv 1 \pmod{3} }}
\Omega_d \mu^2(c) \Psi_c \Big( \frac{d}{c} \Big)_3 \Big |^2 \\
& \leq \| \boldsymbol{\Omega} \|_2^2 \cdot
\Big(\sum_{\zeta} \sum_{k \geq 0} \sum_{\substack{N(\zeta \lambda^{k} m) \sim M \\ m \equiv 1 \pmod{3}}}
\hspace{0.1cm}
\Big | \sum_{\substack{N(c) \sim N \\ c \equiv 1 \pmod{3} }} \mu^2(c) \Big( \frac{\zeta \lambda^{k}}{c} \Big)_3 \Psi_c  \Big(\frac{c}{m} \Big)_3 \Big |^2 \Big) \\
& \ll (MN)^{\varepsilon} M^{1/3} (M+N) 
\| \boldsymbol{\Omega} \|_2^2 \| \boldsymbol{\mu}^2 \boldsymbol{\Psi} \|_2^2,
\end{align*}
as required.
\end{proof}

\section{The Browning--Vishe circle method for number fields}
The proof of our Type-II estimates will use a circle method over number fields 
due to Browning and Vishe \cite[Theorem~1.2]{BroVis}. Their work  
generalises earlier work of Heath-Brown \cite[Theorem~1]{HB2} (over $\mathbb{Q}$), 
and ultimately relies on the $\delta$-function
technology of Duke, Friedlander, and Iwaniec \cite{DFI}.
 
Let $L/\mathbb{Q}$ be a number field of degree $d \geq 2$ with ring of 
integers $\mathcal{O}_L$ and unit group $\mathcal{O}^{\times}_L$. Let $\mathfrak{a} \unlhd \mathcal{O}_L$ be an integral ideal,
$N(\mathfrak{a}):=\# \mathcal{O}_L/\mathfrak{a}$ denote the ideal norm of $\mathfrak{a}$,
and
\begin{equation*}
\delta_L(\mathfrak{a}):=\begin{cases}
1 & \text{if } \mathfrak{a}=(0) \\
0 & \text{otherwise}.
\end{cases}
\end{equation*}

\begin{remark} \label{elementindicator}
One obtains an indicator function on $\mathcal{O}_L$ 
by restricting to principal ideals, 
in which case one writes $\delta_L((\nu))=\delta_L(\nu)$
for any $\nu \in \mathcal{O}_L$. We also have 
$N((\nu))=N(\nu)$, where the latter is the norm of an element of $\mathcal{O}_L$.
\end{remark}

\begin{theorem} \label{DFIcirc} \emph{\cite[Theorem~1.2]{BroVis}} 
Let $L/\mathbb{Q}$ be a number field of degree $d \geq 2$,
$C \geq 1$, and $\mathfrak{a} \unlhd \mathcal{O}_L$ be an integral ideal. Then there exists 
a positive constant $k_C$ and an infinitely differentiable function 
$h(x,y):(0,\infty) \times \mathbb{R} \rightarrow \mathbb{R}$ (depending on $L/\mathbb{Q}$) such that 
\begin{equation} \label{deltacirc}
\delta_L(\mathfrak{a})=\frac{k_C}{C^{2d}} \sum_{(0) \neq \mathfrak{c} \subseteq \mathcal{O}_L}
\hspace{0.1cm} \sideset{}{^*} \sum_{\sigma \pmod{\mathfrak{c}}} \sigma(\mathfrak{a}) h \Big(\frac{N(\mathfrak{c})}{C^d},\frac{N(\mathfrak{a})}{C^{2d}}
\Big),
\end{equation}
where the notation $\Sigma^{*}_{\sigma \pmod{\mathfrak{c}}}$ means that the sum 
is taken over primitive additive characters (extended to ideals) modulo $\mathfrak{c}$. The constant 
$k_C$ satisfies 
\begin{equation} \label{kC}
k_C=1+O_{L/\mathbb{Q},D}(C^{-D}) \quad \emph{for any} \quad D>0.
\end{equation}
 Furthermore, we have 
 \begin{equation} \label{hbound}
 h(x,y) \ll_{L/\mathbb{Q}} x^{-1} \quad \emph{for all} \quad y \in \mathbb{R},
\end{equation} 
and 
\begin{equation} \label{hsupport}
h(x,y) \neq 0 \quad \emph{only if} \quad x \leq \emph{max} \{1,2|y|\}.
\end{equation}
\end{theorem}

\begin{remark} \label{oscillation}
In practice one usually
chooses $C:=X^{1/(2d)}$ 
to detect the condition $\mathfrak{a}=(0)$ for a sequence of
ideals of $\mathcal{O}_L$ with norm less than or equal to $X$. This means that
for $\mathfrak{c}$ (see \eqref{deltacirc}) in the generic range $N(\mathfrak{c}) \asymp X^{1/2}$ there is no
oscillation in the weight function $h(x,y)$.
\end{remark}

\begin{lemma} \label{hyderiv} \emph{\cite[Lemma~3.1]{BroVis}}
Let the notation be as in Theorem \ref{DFIcirc}.
The function $h(x,y)$ vanishes when $x \geq 1$ and $|y| \leq x/2$. When $x \leq 1$ and $|y| \leq x/2$, we have
\begin{equation} \label{partialybound}
\frac{\partial}{\partial y} h(x,y)=0.
\end{equation}
\end{lemma}

\begin{lemma} \label{hlemma} \emph{\cite[Lemma~3.2]{BroVis}}
Let the notation be as in Theorem \ref{DFIcirc}.
Then for $i,j,D \in \mathbb{Z}_{\geq 0}$ we have,
\begin{equation} \label{partialbound}
\frac{\partial^{i+j}}{\partial x^i \partial y^j} h(x,y) \ll_{L/\mathbb{Q},i,j,D} x^{-i-j-1} \Big(x^{D}+ \min \Big \{ 1,\Big ( \frac{x}{|y|} \Big)^{D} \Big \}  \Big).
\end{equation}
The term $x^{D}$ on the right side of \eqref{partialbound} can be omitted if $j \neq 0$.
\end{lemma}

\begin{corollary} \label{yderivcor}
Let the notation be as in Theorem \ref{DFIcirc}.
Then for any $j \in \mathbb{Z}_{\geq 1}$ we have
\begin{equation}
\frac{\partial^{j}}{\partial y^{j}} h(x,y) \ll_{L/\mathbb{Q},j} x^{-j-1}.
\end{equation}
\end{corollary}
\begin{proof}
This follows from Lemma \ref{hlemma} with $i=D=0$.
\end{proof}

\section{Vaughan's identity}
Here we record a celebrated identity of Vaughan \cite{Vau} adapted to our situation.
\begin{prop} \label{vaughanid}
Let $R,S \geq 1$.
Then for any $\nu \in \mathbb{Z}[\omega]$ with $\nu \equiv 1 \pmod{3}$ 
and $N(\nu)>S$, we have
\begin{equation} \label{vauid}
\Lambda(\nu)=\sum_{\substack{a \mid \nu \\ N(a) \leq R }} \mu(a) \log \Big( \frac{N(\nu)}{N(a)} \Big)
-\mathop{\sum \sum}_{\substack{ab  \mid \nu \\  N(a) \leq R \\ N(b) \leq S}}  \mu(a) \Lambda(b)
+\mathop{\sum \sum}_{\substack{ab \mid \nu \\ N(a)>R \\ N(b) >S }} \mu(a) \Lambda(b),
\end{equation}
If $N(\nu) \leq S$, the right side of \eqref{vauid} vanishes.
\end{prop}

\section{Proof of Theorem \ref{mainthm} and Corollary \ref{maincor}}
In this section we prove Theorem \ref{mainthm} and Corollary \ref{maincor}
assuming the truth of 
Lemma \ref{trivpoint} and the main inputs:
Theorem \ref{type2a} and Proposition \ref{type1avgbd}.  

\begin{proof}[Proof of Thm \ref{mainthm} assuming Lem \ref{trivpoint}, Thm \ref{type2a}, and Prop \ref{type1avgbd}]
Recall the definition of the quantity $\mathscr{P}_f(X,v,u;W_K)$ given in \eqref{PXdefn}.
We apply Proposition \ref{vaughanid} to \eqref{PXdefn}. The parameters $R,S \geq 1$ used in our application of Proposition \ref{vaughanid} 
will be chosen at a later point in the proof and will satisfy
\begin{equation} \label{RScond}
S<\frac{X}{10000} \quad \text{and} \quad 10000X<RS<10000000X \quad \text{say},
\end{equation} 
for all sufficiently large $X$. Since the support of $W_K$ is contained in $[1,2]$
and $S<X/10000$ by \eqref{RScond}, all summands in $\mathscr{P}_f(X,v,u;W_K)$
are automatically supported on the condition $N(\nu)>S$.
Note that the right most sum in \eqref{vauid} vanishes since
the support of $W_K$ is contained in $[1,2]$ and $RS>10000X$ by \eqref{RScond}.
We insert a smooth partition of unity in the $a$ and $b$ variables in the second sum in \eqref{vauid},
and then interchange these summations with the $\nu$ summation after substitution of \eqref{vauid}
into \eqref{PXdefn}. We obtain
\begin{align} \label{decomp}
& \mathscr{P}_f(X,v,u;W_K) \nonumber \\
&= \mathscr{P}_{1f}(X,R,v,u;W_K) - \mathop{\sum \sum}_{\substack{ 1/2 \leq M \leq 2R \\ 1/2 \leq N \leq 2S \\ M,N \text{ dyadic}}} \mathscr{P}_{2f}(X,M,N,v,u;W_K),
\end{align}
where 
\begin{equation} \label{P1def}
\mathscr{P}_{1f}(\cdots) 
:=\mathop{\sum \sum}_{\substack{ a,b \equiv 1 \pmod{3}  \\ ab \equiv u \pmod{v} \\ N(a) \leq R }} \mu(a) \log(N(b)) \rho_f(\lambda^{-3}ab) W_K \Big( \frac{N(\lambda^{-3} ab)}{X} \Big),
\end{equation}
and 
\begin{equation} \label{P2def}
\mathscr{P}_{2f}(\cdots) 
:=\mathop{\sum \sum \sum}_{ \substack{ a,b,c \equiv 1 \pmod{3} \\ abc \equiv u \pmod{v} \\ N(a) \leq R \\ N(b) \leq S  }}
\mu(a) \Lambda(b) \rho_f(\lambda^{-3} abc)  
W_K \Big(\frac{N(\lambda^{-3} abc)}{X}  \Big) U \Big( \frac{N(a)}{M} \Big) U \Big( \frac{N(b)}{N} \Big),
\end{equation}
and  $U: \mathbb{R} \rightarrow \mathbb{R}_{\geq 0}$ is a fixed smooth function with compact support
in $[1/2,2]$ such that
\begin{equation*}
\sum_{j\in\mathbb Z}U(x/2^j)=1\qquad (x>0).
\end{equation*}
Thus a factor $U(N(a)/M)$ restricts $M/2\leq N(a)\leq2M$.

\subsection{Estimate for $\mathscr{P}_{1f}(X,R,v,u;W_K)$}
Re-writing \eqref{P1def} using additive characters we obtain
\begin{align}
\mathscr{P}_{1f}(\cdots) 
&=\sum_{\substack{a \equiv 1 \pmod{3} \\ N(a) \leq R \\ (a,v)=1}} \frac{\mu(a)}{N(av)} \sum_{j \pmod{av}} \check{e} \Big({-\frac{j \eta}{\lambda av}} \Big) \nonumber \\
& \times \sum_{\nu \in \lambda^{-3} \mathbb{Z}[\omega]}   
\rho_f(\nu) \check{e} \Big( \frac{j  \lambda^3 \nu}{\lambda av} \Big)  \log \Big( \frac{N(\nu)}{N(\lambda^{-3} a)} \Big) W_K \Big(\frac{N(\nu)}{X} \Big),
\end{align}
where $\eta \in \mathbb{Z}[\omega]$ is such that $\eta \equiv u \pmod{v}$ and $\eta \equiv 0 \pmod{a}$.
The restriction $(a,v)=1$ follows from $ab\equiv u\pmod v$ and guarantees the existence of $\eta$.
The factor $\lambda$ in the additive characters accounts for the trace-dual lattice
$\mathbb Z[\omega]^*=\lambda^{-1}\mathbb Z[\omega]$.
Applying Lemma \ref{smoothwilton} (while noting Remark \ref{cubicshimurarem}) to the $\nu$ summation and estimating the other sums trivially using the triangle inequality we obtain
\begin{equation} \label{PRest}
\mathscr{P}_{1f}(\cdots)  \ll (RX)^{\varepsilon} K R X^{1/2},
\end{equation}
uniformly in the modulus $v$.

\subsection{Two estimates for $\mathscr{P}_{2f}(\cdots)$}

\subsubsection{First estimate}

For the first estimate we treat \eqref{P2def} as an average Type-I sum.
That is, in \eqref{P2def} we let $h=ab$, 
\begin{equation} \label{gammaprime}
\gamma^{\prime}_h(M,N):=\sum_{\substack{h=ab\\ a,b\equiv1\pmod3\\N(a)\leq R\\N(b)\leq S}} \mu(a) \Lambda(b) U \Big( \frac{N(a)}{M} \Big) U \Big( \frac{N(b)}{N} \Big),
\end{equation}
and interpret $c$ as the ``smooth" summation variable.
We then decompose
\begin{equation} \label{P2decompose}
\mathscr{P}_{2f}(\cdots)=\mathscr{P}^{\star}_{2f}(\cdots)+\mathscr{P}^{\dagger}_{2f}(\cdots),
\end{equation}
where $\mathscr{P}^{\star}_{2f}(\cdots)$ (resp.$\mathscr{P}^{\dagger}_{2f}(\cdots)$)
has the factor $\mu^2(h)$ (resp. $1-\mu^2(h)$) inserted.  
The weight $\mu^2(h) \gamma^{\prime}_h(M,N)$ in $\mathscr{P}^{\star}_2(\cdots)$ is supported on squarefree
elements of $\mathbb{Z}[\omega]$. We apply 
Proposition \ref{type1avgbd} (see \eqref{type1avg}) to obtain 
\begin{equation} \label{Pstarest}
\mathscr{P}^{\star}_{2f}(\cdots) \ll (XKN(v))^{\varepsilon} K^{14/3} N(v)^{17/6}  (MN)^{5/6} X^{1/3}.
\end{equation}
Applying Lemma \ref{trivpoint} (see \eqref{type1point}) to the $c$-sum in $\mathscr{P}^{\dagger}(\cdots)$
we obtain 
\begin{align} \label{Pdaggerest}
\mathscr{P}^{\dagger}_{2f}(\cdots) & \ll  (XKN(v))^{\varepsilon} K^4 N(v)^{1/2} (MN)^{1/2} \| (\boldsymbol{1}-\boldsymbol{\mu}^2) \boldsymbol{\gamma^{\prime}}(M,N) \|_1 
\nonumber \\
& \ll (XKN(v))^{\varepsilon} K^4 N(v)^{1/2} M^{3/2} N.
\end{align}
To justify the second line, $\mu(a)\ne0$ forces $a$ to be squarefree.
If $ab$ is not squarefree and $\Lambda(b)\ne0$, either $b=\varpi^k$ with $k\geq2$,
or $b=\varpi$ is a prime dividing $a$.
There are $O_\varepsilon(N^{1/2+\varepsilon})$ proper prime powers with $N(b)\asymp N$,
so the first case contributes $O_\varepsilon((MN)^{\varepsilon}MN^{1/2})$
to $\| (\boldsymbol1-\boldsymbol\mu^2)\boldsymbol\gamma'(M,N)\|_1$.
In the second case, sum the prime divisors of each squarefree $a$ with $N(a)\asymp M$.
Their logarithmic weights sum to at most $\log N(a)$.
This contributes $O_\varepsilon(M^{1+\varepsilon})$, which is covered by the same bound.
Substitution of \eqref{Pstarest} and \eqref{Pdaggerest} into \eqref{P2decompose} gives
\begin{equation} \label{P2est1}
\mathscr{P}_{2f}(\cdots) \ll (XKN(v))^{\varepsilon}(K^{14/3} N(v)^{17/6} (MN)^{5/6} X^{1/3}+ K^4 N(v)^{1/2} M^{3/2} N).
\end{equation}

\subsubsection{Second estimate}

For the second estimate we treat \eqref{P2def} as a Type-II sum.
That is, we let $h=bc$, and
\begin{equation*}
\gamma_h(N,X/MN):=
\mathbf1_{27X/(2M)\leq N(h)\leq108X/M}
\sum_{\substack{h=bc\\b,c\equiv1\pmod3\\N(b)\leq S}}
\Lambda(b) U \Big ( \frac{N(b)}{N} \Big).
\end{equation*}
The inserted range for $h$ contains every term in \eqref{P2def}, because
$N(a)\in[M/2,2M]$ and $N(\lambda^{-3}ah)\in[X,2X]$.
The other coefficient is
$\mathbf1_{a\equiv1\pmod3}\mathbf1_{N(a)\leq R}\mu(a)U(N(a)/M)$,
which is supported only on squarefree $a$.
Thus we apply Theorem \ref{type2a} (see \eqref{type2}) and obtain
\begin{equation} \label{P2est2}
\mathscr{P}_{2f}(\cdots) \ll (XKN(v))^{\varepsilon} K^8 N(v)^4 (XM^{-1/2}+(MX)^{3/4}).
\end{equation}

\subsection{Conclusion}
We use \eqref{PRest} to estimate the first term of \eqref{decomp}.
Let $1 \ll L \ll R$. We use \eqref{P2est1} (resp. \eqref{P2est2})
to estimate the second term in \eqref{decomp} when $M \leq L$ (resp. $M \geq L$).
The net result is
\begin{equation} \label{PXintermed}
  \mathscr{P}_f(X,v,u;W_K)  
 \ll (XKN(v))^{\varepsilon} K^8 N(v)^{4} ( R X^{1/2} +
 (LS)^{5/6}X^{1/3}+ L^{3/2} S+XL^{-1/2}+ (RX)^{3/4})
\end{equation}
for any $R,S \geq 1$ satisfying \eqref{RScond} and $1 \ll L \ll R$.
The choice of parameters 
\begin{equation*}
R=1000 X^{5/17}, \quad S=1000 X^{12/17}, \quad \text{and} \quad L=X^{1/17},
\end{equation*}
satisfies \eqref{RScond} for all sufficiently large $X$, and substitution into \eqref{PXintermed} yields 
\begin{equation*}
\mathscr{P}_f(X,v,u;W_K) \ll (XKN(v))^{\varepsilon} K^8 N(v)^{4} X^{1-1/34},
\end{equation*}
as required.
\end{proof}

We now remove the smoothing.

\begin{proof}[Proof of Corollary \ref{maincor}]
Let $Z$ be sufficiently large and $4\leq K\leq Z^{1/2}$.
Put $\Delta:=K^{-1}$ and choose a smooth function $W_K:(0,\infty) \rightarrow \mathbb{R}$ satisfying
\begin{align}
& \text{supp}(W_K) \subset [5/4-\Delta, 7/4+\Delta]\subset[1,2], \quad 0 \leq W_K(x) \leq 1 \quad \text{for all} \quad x>0, \nonumber \\
& \qquad W_K(x) =1  \quad \text{for} \quad x \in [5/4,7/4], \quad \text{and} \quad W_K^{(j)}(x) \ll_j K^{j}.
\end{align}
Then for any $Z \gg 1$ we have 
\begin{equation} \label{insertsmooth}
\sum_{\substack{\nu \in \lambda^{-3} \mathbb{Z}[\omega]  \\ \lambda^3 \nu \equiv u \pmod{v} \\ 5Z/4 < N(\nu) \leq 7Z/4  }} 
\rho_f(\nu)  \Lambda(\lambda^3 \nu) =  \sum_{\substack{\nu \in \lambda^{-3} \mathbb{Z}[\omega]  \\ \lambda^3 \nu \equiv u \pmod{v} }} 
\rho_f(\nu)  \Lambda(\lambda^3 \nu) W_K \Big(\frac{N(\nu)}{Z} \Big) + O_v(K^{-1/2} Z^{1+\varepsilon} ),
\end{equation}
Indeed, the discrepancy is supported on two boundary annuli whose squared-radius
intervals have total length $O(Z/K)$ and whose radii are $O(Z^{1/2})$.
The elementary lattice-point bound for these annuli is
$O(Z/K+Z^{1/2})=O(Z/K)$, using $K\leq Z^{1/2}$.
Cauchy--Schwarz, $\Lambda(\lambda^3\nu)\ll\log Z$, and Lemma \ref{rankinselbound}
therefore give the error term in \eqref{insertsmooth}.
Applying Theorem \ref{mainthm} to the right side \eqref{insertsmooth} gives
\begin{equation*}
\sum_{\substack{\nu \in \lambda^{-3} \mathbb{Z}[\omega]  \\ \lambda^3 \nu \equiv u \pmod{v} \\ 5Z/4 < N(\nu) \leq 7Z/4  }} 
\rho_f(\nu)  \Lambda(\lambda^3 \nu)  \ll_v (ZK)^{\varepsilon} (K^8 Z^{1-1/34} +K^{-1/2} Z).
\end{equation*}
We choose $K=Z^{1/289}$ to obtain 
\begin{equation}
\sum_{\substack{\nu \in \lambda^{-3} \mathbb{Z}[\omega]  \\ \lambda^3 \nu \equiv u \pmod{v} \\ 5Z/4 < N(\nu) \leq 7Z/4  }} 
\rho_f(\nu) \Lambda(\lambda^3 \nu)  \ll_v Z^{1-1/578+\varepsilon}.
\end{equation}
Apply this estimate with $Z=Z_j:=(4X/7)(5/7)^j$, $j\geq0$, while $Z_j$ is sufficiently large.
The intervals $(5Z_j/4,7Z_j/4]$ are disjoint and cover $(0,X]$ apart from a bounded initial range.
Summing the resulting geometric series, and treating that bounded range trivially, yields \eqref{PXbd}.

To prove \eqref{tildePXbd} we first observe that 
\begin{equation} \label{minuseq}
\widetilde{\mathscr{P}_f}(X;v,u)-\mathscr{P}_f(X;v,u) =-
\mathop{\sum \sum}_{\substack{ k \in \mathbb{Z}_{\geq 2}, \varpi \in \mathbb{Z}[\omega] \\ \varpi \text{ prime},\ \varpi\equiv1\pmod3 \\  \varpi^k \equiv u \pmod{v} \\ N(\lambda^{-3} \varpi^k) \leq X }}
\rho_f(\lambda^{-3} \varpi^k) \log N(\varpi).
\end{equation}
Taking the primary generator $\varpi\equiv1\pmod3$ ensures that each proper prime power is counted once.
There are $O_\varepsilon(X^{1/2+\varepsilon})$ such powers in this range.
Applying Cauchy--Schwarz to \eqref{minuseq} and absorbing logarithmic weights gives
\begin{align}
& |\widetilde{\mathscr{P}_f}(X;v,u)-\mathscr{P}_f(X;v,u)|
\ll X^{1/4+\varepsilon} \Big( \mathop{\sum \sum}_{\substack{ k \in \mathbb{Z}_{\geq 2}, \varpi \in \mathbb{Z}[\omega] \\ \varpi \text{ prime},\ \varpi\equiv1\pmod3 \\ N(\lambda^{-3} \varpi^k) \leq X }}
|\rho_f(\lambda^{-3} \varpi^k)|^2 \Big)^{1/2} \nonumber \\
& \ll X^{1/4+\varepsilon} \Big(\sum_{\substack{\nu \in \lambda^{-3} \mathbb{Z}[\omega] \\ N(\nu) \leq X }}
|\rho_f(\nu)|^2 \Big)^{1/2} \ll X^{3/4+\varepsilon},
\end{align}
where the last inequality follows from using Lemma \ref{rankinselbound}.
The result \eqref{tildePXbd} now follows. 
\end{proof}

The rest of the paper will be dedicated 
to proving Lemma \ref{trivpoint}, Theorem \ref{type2a}, and Proposition \ref{type1avgbd}.

\section{Voronoi summation formulae for twists}
In this section we develop a Voronoi summation formula for twists of
a cusp form
$f \in L^2(\Gamma_2 \backslash \mathbb{H}^3,\chi)$ with spectral parameter $\tau_f \in 1+i \mathbb{R}$
by appropriate non-Archimedean and Archimedean characters.
Development of this formula requires some care because 
we are working with the group $\Gamma_2=\langle \operatorname{SL}_2(\mathbb{Z}),\Gamma_1(3)  \rangle$
in $\Gamma:=\operatorname{SL}_2(\mathbb{Z}[\omega])$.

\subsection{Twists and Dirichlet series}
We will need consider cubic metaplectic forms on 
groups $\Gamma_1(C)$ with $C \equiv 0 \pmod{9}$
i.e. the spaces $L^2(\Gamma_1(C) \backslash \mathbb{H}^3,\chi,\tau)$
for $\tau \in \mathbb{C}$ with $\Re(\tau) \geq 1$.
To simplify our exposition we focus on the non-exceptional case 
i.e. $\Re(\tau)=1$.
Suppose that 
$\Psi: \lambda^{-1} \mathbb{Z}[\omega] \rightarrow \mathbb{C}$ is periodic modulo  
$(\lambda^{m} r)(\lambda^{-1} \mathbb{Z}[\omega])$.
The $\Psi$-twist (at $\infty$) of a cusp form
$F \in L^2(\Gamma_1(C) \backslash \mathbb{H}^3,\chi,\tau)$ is defined by
\begin{equation} \label{twistdef2}
F(w;\Psi):=\sum_{\substack{0 \neq \nu \in (\lambda C)^{-1} \mathbb{Z}[\omega] }} 
\rho_{F}(\nu) \Psi( C \nu) v K_{\tau-1} (4 \pi \lvert \nu \rvert v) 
\check{e}(\nu z), \quad w=(z,v) \in \mathbb{H}^3,
\end{equation}
also denoted by $(F \otimes \Psi)(w)$.
By \cite[Theorem~0.3.12]{Pro} and its proof 
we have
\begin{equation} \label{cuspform}
F(\cdot;\Psi) \in L^2(\Gamma_1(\lambda^{2m} r^2 C) \backslash \mathbb{H}^3,\chi,\tau) \quad \text{is a cusp form}.
\end{equation}

\begin{remark} \label{twistview}
For the purposes of twisting we view the cusp form $f \in  L^2(\Gamma_2 \backslash \mathbb{H}^3,\chi)$ 
in the larger space $L^2(\Gamma_1(\lambda^4) \backslash \mathbb{H}^3,\chi)$. 
This is immaterial in the final results and only involves extra fixed powers of the prime $\lambda$
in the formulae.
\end{remark}

In what follows it will be instructive to open the definition $\check{e}(z):=e(z+\overline{z})$, $z \in \mathbb{C}$.
We remind the reader that the function $F(w;\Psi)$ in \eqref{twistdef2} is a function in $z,\overline{z}$, and $v$ (although the notation suppresses this).
For $n \in \mathbb{Z}$, we define 
\begin{equation} \label{twistformdiff}
F(w;\Psi,n):=\frac{1}{(2 \pi i)^{|n|}} \cdot 
\begin{cases}
\big( \frac{\partial}{\partial z} \big)^{n} F(w;\Psi) & \text{if} \quad n>0 \\
F(w;\Psi) & \text{if} \quad n=0 \\
\big( \frac{\partial}{\partial \overline{z}} \big)^{|n|} F(w;\Psi) & \text{if} \quad n<0 
\end{cases}, \quad w=(z,v) \in \mathbb{H}^3.
\end{equation}
To complement \eqref{twistdef2}, we have the Fourier expansions (at $\infty$) for $n \in \mathbb{Z} \setminus \{0\}$,
\begin{equation} \label{twistdef2a}
F(w;\Psi,n):=
\sum_{0 \neq \nu \in (\lambda C)^{-1} \mathbb{Z}[\omega]}  
\begin{cases}
\rho_{F}(\nu) \nu^n \Psi(C \nu) v K_{\tau-1} (4 \pi \lvert \nu \rvert v) 
e(\nu z+\overline{\nu z}) & \text{if} \quad n>0 \\
 \rho_{F}(\nu) \overline{\nu}^{|n|} \Psi(C \nu) v K_{\tau-1} (4 \pi \lvert \nu \rvert v) 
e(\nu z+\overline{\nu z}) & \text{if} \quad n<0
\end{cases}.
\end{equation}

Suppose that $\psi: \mathbb{Z}[\omega] \rightarrow \mathbb{C}$ is periodic modulo $\lambda^{m} r$.
The (normalised) Fourier transform $\widehat{\psi}: \lambda^{-1} \mathbb{Z}[\omega] \rightarrow \mathbb{C}$ 
is given by
\begin{equation} \label{fouriertrans}
\widehat{\psi}(x):=\frac{1}{N(\lambda^{m} r)} \sum_{u \pmod{\lambda^{m} r} } \psi(u) \check{e} \Big(\frac{ux}{\lambda^{m} r} \Big), \quad x \in  \lambda^{-1} \mathbb{Z}[\omega],
\end{equation}
and is periodic modulo $(\lambda^{m} r)(\lambda^{-1} \mathbb{Z}[\omega])$.
Fourier inversion asserts that
\begin{equation} \label{fourierinv}
\psi(u):=\sum_{x \in \lambda^{-1} \mathbb{Z}[\omega]/ (\lambda^{m} r)(\lambda^{-1} \mathbb{Z}[\omega]) } 
\widehat{\psi}(x) \check{e} \Big({-\frac{xu}{\lambda^{m} r}} \Big), \quad \text{for} \quad u \in \mathbb{Z}[\omega]. 
\end{equation}

For $n \in \mathbb{Z}$ consider the Dirichlet series
\begin{equation*}
\mathcal{D}(s,F; \Psi,n):=\sum_{\substack{ \nu \neq 0 \\ \nu \in (\lambda C)^{-1} \mathbb{Z}[\omega]}} 
\frac{\rho_{F}(\nu)\Psi( C \nu) \big( \frac{\nu}{|\nu|} \big)^n}{N(\nu)^s}, \quad \Re(s)>1,
\end{equation*}
and the associated Mellin transform
\begin{equation*}
\Lambda(s,F; \Psi,n):= \int_{0}^{\infty} F (v;\Psi,n) v^{2s+|n|-2} dv,
\end{equation*}
where we let $v$ denote $(0,v)$ for $v>0$. 
Let
\begin{equation} \label{Ginfty}
G_{\infty}(s,\tau,n):=\frac{1}{4} (2 \pi)^{-2s-|n|} \Gamma \Big (s+\frac{|n|}{2} -\frac{1}{2}(\tau-1) \Big )
 \Gamma \Big (s+\frac{|n|}{2} + \frac{1}{2}(\tau-1) \Big ), \quad s \in \mathbb{C}.
\end{equation}
\begin{lemma} \label{mellin1}
Let $\tau \in \mathbb{C}$ with $\Re(\tau)=1$, $C \in \mathbb{Z}[\omega]$ with $C \equiv 0 \pmod{9}$,
$F \in L^2(\Gamma_1(C) \backslash \mathbb{H}^3,\chi,\tau)$ be a cusp form,
and $n \in \mathbb{Z}$.
For $\Re(s)>1$ we have
\begin{equation*}
\Lambda(s,F;\Psi,n)=G_{\infty}(s,\tau,n) \mathcal{D}(s,F;\Psi,n),
\end{equation*}
where $G_{\infty}(s,\tau,n)$ is given by \eqref{Ginfty}
\end{lemma}
\begin{proof}
The proofs for the cases $n>0$, $n=0$, and $n<0$ are analogous. We
give details for the case $n>0$.
For $\Re(s)>1$ and $n>0$ we have
\begin{align}
 \Lambda(s,F;\Psi,n) 
&=\int_{0}^{\infty} \sum_{\substack{\nu \neq 0 \\ \nu \in (\lambda C)^{-1} \mathbb{Z}[\omega] }}
\rho_{F}(\nu) \nu^n  \Psi( C \nu)  K_{\tau-1}(4 \pi |\nu| v) v^{2s+n-1} dv \nonumber \\
&=\frac{1}{(4 \pi)^{2s+n}} \sum_{\substack{\nu \neq 0 \\ \nu \in (\lambda C)^{-1} \mathbb{Z}[\omega] }} 
\frac{\rho_{F}(\nu) \Psi(C \nu) \big( \frac{\nu}{|\nu|} \big)^n }{N(\nu)^s}
\int_{0}^{\infty} K_{\tau-1}(T) T^{2s+n-1} dT \nonumber \\
&=\frac{1}{4} (2 \pi)^{-2s-n} \Gamma \Big (s+\frac{n}{2} -\frac{1}{2}(\tau-1) \Big ) \Gamma \Big (s +\frac{n}{2} + \frac{1}{2}(\tau-1) \Big ) \nonumber \\
& \times \sum_{\substack{\nu \neq 0 \\ \nu \in (\lambda C)^{-1} \mathbb{Z}[\omega] }}
\frac{\rho_{F}(\nu) \Psi( C \nu) \big(\frac{\nu}{|\nu|}\big)^{n}}{N(\nu)^s}. 
\label{mellineval}
\end{align}
The interchange of summation and integration above for $\Re(s)>1$ is justified by absolute convergence
(see Lemma \ref{L1bound}) and \cite[(10.25.3),(10.45.7)]{NIST:DLMF}. Furthermore,
\eqref{mellineval} follows from \cite[(10.43.19)]{NIST:DLMF}.
\end{proof}

\subsection{A special case}
Recall that $f \in L^2(\Gamma_2 \backslash \mathbb{H}^3,\chi)$ is a cusp form with spectral parameter $\tau_f \in 1+i \mathbb{R}$.
For $\ell \in \mathbb{Z}_{\geq 0}$, $q \in \mathbb{Z}[\omega]$ with $q \equiv 1 \pmod{3}$, and $\eta \in \mathbb{Z}[\omega]/\lambda^{\ell} q \mathbb{Z}[\omega]$, 
let
\begin{equation}  \label{modqtwist}
f(w;\lambda^{\ell} q,\eta):=
\sum_{\substack{\nu \neq 0 \\ \nu \in \lambda^{-3} \mathbb{Z}[\omega] \\ \lambda^3 \nu \equiv \eta \pmod{\lambda^{\ell} q} }} 
\rho_{f}(\nu) v K_{\tau_f-1} (4 \pi \lvert \nu \rvert v) \check{e}(\nu z). 
\end{equation}
Following \eqref{twistdef2}--\eqref{twistdef2a} we also have the functions 
$f(w;\lambda^{\ell} q,\eta,n)$ and their associated Fourier expansions for each $n \in \mathbb{Z}$.
We ultimately need a Voronoi formulae for the Fourier coefficients of $f(w;\lambda^{\ell} q,\eta,n)$.
Consider the Dirichlet series
\begin{equation*}
\mathcal{D}(s,f;\lambda^{\ell} q,\eta,n):=
\sum_{\substack{\nu \in \lambda^{-3} \mathbb{Z}[\omega] \\ \lambda^{3} \nu \equiv \eta \pmod{\lambda^{\ell} q} }} 
\frac{\rho_{f}(\nu) \big( \frac{\nu}{|\nu|} \big)^{n} }{N(\nu)^s}, \quad \Re(s)>1,
\end{equation*}
and the associated Mellin transform
\begin{equation*}
\Lambda(s,f;\lambda^{\ell} q,\eta,n):= \int_{0}^{\infty} f(v;\lambda^{\ell} q,\eta,n) v^{2s+|n|-2} dv.
\end{equation*}
Then Lemma \ref{mellin1} asserts that
\begin{equation} \label{mellin2}
\Lambda(s,f;\lambda^{\ell} q,\eta,n)=G_{\infty}(s,\tau_f,n)
\mathcal{D}(s,f;\lambda^{\ell} q,\eta,n) \quad \text{for} \quad \Re(s)>1,
\end{equation}
where $G_{\infty}(s,\tau,n)$ is given by \eqref{Ginfty}.

We detect the congruence condition in \eqref{modqtwist} using Fourier transforms.
For $\eta \in \mathbb{Z}[\omega]$, $m\in\mathbb Z_{\geq0}$, and $r\equiv1\pmod3$, let
\begin{equation} \label{modqeta}
\psi_{\lambda^{m} r}(u)_{\eta}:=\mathbf{1}_{\lambda^{m} r}(u) \cdot
\check{e} \Big({-\frac{\eta u}{\lambda^{m} r} } \Big ),
\end{equation} 
where $\mathbf{1}_{\lambda^{m} r}(\cdot)$ is the principal character modulo $\lambda^{m} r$. 
As a shorthand we write $\psi_{\lambda^{m} r}(u):=\psi_{\lambda^{m} r}(u)_0$.
The function $\psi_{\lambda^{m} r}(\cdot)_{\eta}$ is periodic
modulo $\lambda^{m} r$.
The Fourier transform is
\begin{equation} \label{normram}
\widehat{\psi_{\lambda^{m} r}(\cdot)_{\eta}}(k)=\frac{1}{N(\lambda^{m} r)}
\sum_{\substack{u \pmod{\lambda^{m} r} \\  (u,\lambda^{m} r)=1 }} \check{e} \Big(\frac{(k-\eta)u}{\lambda^{m} r} \Big), \quad k \in \lambda^{-1} \mathbb{Z}[\omega].
\end{equation}
As a shorthand we write $\widehat{\psi_{\lambda^{m} r}}(k):=\widehat{\psi_{\lambda^{m} r}(\cdot)_0}(k)$.
A straightforward computation shows the following orthogonality relation.
\begin{lemma} \label{orthogonality}
For $\ell \in \mathbb{Z}_{\geq 0}$ and $k,\eta \in \mathbb{Z}[\omega]/\lambda^{\ell} q \mathbb{Z}[\omega]$ we have 
\begin{equation} \label{orthogidr}
\frac{1}{N(\lambda^{\ell+1} q)} \sum_{r \mid q} \sum_{m=0}^{\ell+1} N(\lambda^{m} r) \widehat{\psi_{\lambda^{m} r}(\cdot)_{\eta}}(k)=
\delta_{k \equiv \eta \pmod{\lambda^{\ell} q}}. 
\end{equation} 
\end{lemma}
\begin{proof}
Grouping residues according to their greatest common divisor with
$\lambda^{\ell+1}q$ identifies the left side with
\[
\frac{1}{N(\lambda^{\ell+1}q)}
\sum_{u\pmod{\lambda^{\ell+1}q}}
\check e\left(\frac{(k-\eta)u}{\lambda^{\ell+1}q}\right).
\]
Since $\mathbb Z[\omega]^*=\lambda^{-1}\mathbb Z[\omega]$, this is $1$
exactly when $k-\eta\in\lambda^\ell q\mathbb Z[\omega]$, and is $0$ otherwise.
\end{proof}

The following Lemma records the standard evaluation of Ramanujan sums.
\begin{lemma} \label{ramstand}
Let $r \in \mathbb{Z}[\omega]$ satisfy $r \equiv 1 \pmod{3}$ and $k \in \mathbb{Z}[\omega]$. Then
we have 
\begin{equation*}
  \widehat{\psi}_r(k) := \frac{1}{N(r)} \sum_{\substack{x \pmod{r} \\ (x,r) = 1}} \check{e} \Big ( \frac{k x}{r} \Big )= \frac{1}{N(r)} \mu \Big( \frac{r}{(r,k)}  \Big) 
  \frac{\varphi(r)}{\varphi \big( \frac{r}{(r,k)} \big)},
\end{equation*}
 where $\varphi(\cdot)$ is the Euler $\varphi$-function on $\mathbb{Z}[\omega]$.
\end{lemma}
\begin{proof}
This follows from the multiplicativity of Ramanujan sums in the modulus $r$, and
the first, fourth, and eighth cases in the evaluation on \cite[pg.~11]{Pro}.
\end{proof}

We next prove a straightforward but crucial Lemma establishing the ``flatness'' of Ramanujan sums 
when averaged over the modulus. 
\begin{lemma} \label{ramflat}
Let $r,k \in \mathbb{Z}[\omega]$ satisfy $r \equiv 1 \pmod{3}$,  
$\widehat{\psi_r}(k)$ be the normalised Ramanujan sum as in the statement of Lemma \ref{ramstand}, and $\varepsilon>0$. Then for $R \geq 1$ we have 
\begin{equation}
\sum_{\substack{ r \in \mathbb{Z}[\omega] \\  N(r) \sim R \\ r \equiv 1 \pmod{3}}} |\widehat{\psi_r}(k)| 
\ll_{\varepsilon}
\delta_{k=0} \cdot R + \delta_{k \neq 0} \cdot (N(k)R)^{\varepsilon}. 
\end{equation}
\end{lemma}
\begin{proof}
When $k=0$ we have the trivial estimate $\ll R$.
When $k \neq 0$ use Lemma \ref{ramstand}, M\"{o}bius inversion, and the triangle inequality 
to obtain
\begin{align} \label{flatcalc}
\sum_{\substack{ r \in \mathbb{Z}[\omega] \\  N(r) \sim R \\ r \equiv 1 \pmod{3}}} |\widehat{\psi_r}(k)| 
& \leq \sum_{\substack{ \gamma \mid k \\ \gamma \equiv 1 \pmod{3}  }}  \sum_{\substack{ r \in \mathbb{Z}[\omega] \\  N(r) \sim R \\ r \equiv 1 \pmod{3} \\ (r,k)=\gamma}} 
\frac{1}{N(r)}
  \frac{\varphi(r)}{\varphi (r/\gamma)}  \nonumber \\
& \leq \sum_{\substack{ \gamma \mid k \\ \gamma \equiv 1 \pmod{3}  }} 
\sum_{\substack{ n,u \in \mathbb{Z}[\omega] \\  N(nu) \sim R/N(\gamma) \\ nu \equiv 1 \pmod{3} }} \frac{1}{N(\gamma nu)} \frac{\varphi(\gamma nu)}{\varphi(nu)} \nonumber  \\
& \ll (N(k) R)^{\varepsilon},
\end{align}
where the last display follows from standard lower bounds for the Euler $\varphi$-function and \eqref{divis}.
This completes the proof.
\end{proof}

Recall the convention for twisting a cusp form $f \in L^2(\Gamma_2 \backslash \mathbb{H}^3,\chi)$
in Remark \ref{twistview}. We replace the congruence $\lambda^3 \nu \equiv \eta \pmod{\lambda^{\ell}q}$
with the equivalent congruence $\lambda^4 \nu \equiv \lambda \eta \pmod{\lambda^{\ell+1}q}$.
We have the immediate consequence.
\begin{lemma} \label{orthog}
Let $\ell \in \mathbb{Z}_{\geq 0}$, $n \in \mathbb{Z}$, $q \in \mathbb{Z}[\omega]$ with $q \equiv 1 \pmod{3}$, and $\eta \in \mathbb{Z}[\omega]/\lambda^{\ell} q \mathbb{Z}[\omega]$.
For $\Re(s)>1$ we have, 
\begin{equation*}
\mathcal{D}(s,f;\lambda^{\ell} q,\eta,n)=\frac{1}{N(\lambda^{\ell+2} q)} \sum_{r \mid q} \sum_{m=0}^{\ell+2} N(\lambda^{m} r) \mathcal{D}(s,f; \widehat{\psi_{\lambda^{m} r}(\cdot)_{\lambda \eta}},n),
\end{equation*}
and 
\begin{equation*}
\Lambda(s,f;\lambda^{\ell} q,\eta,n)=\frac{1}{N(\lambda^{\ell+2} q)} \sum_{r \mid q} \sum_{m=0}^{\ell+2} N(\lambda^{m} r) \Lambda(s,f; \widehat{\psi_{\lambda^{m} r}(\cdot)_{\lambda \eta}},n),
\end{equation*}
where $\widehat{\psi_{\lambda^m r}(\cdot)_{\eta}}$ is given in \eqref{normram}.
\end{lemma}

To obtain a functional equation for $\Lambda(s,f;\lambda^{\ell} q,\eta,n)$ under $s \rightarrow 1-s$ it suffices 
to establish a functional equation for each $\Lambda(s,f; \widehat{\psi_{\lambda^{m} r}(\cdot)_{\eta},n})$.
We have two different cases according to whether $m \in \mathbb{Z}_{\geq 6}$ or $0 \leq m \leq 5$.

\subsection{Functional equation 1: $m \in \mathbb{Z}_{\geq 6}$} 
Suppose that $\psi: \mathbb{Z}[\omega] \rightarrow \mathbb{C}$ is periodic modulo $\lambda^{m} r$ where
$r \equiv 1 \pmod{3}$.

\begin{remark}
The version of the functional equation proved in this section uses the automorphy of $f \in L^2(\Gamma_2 \backslash \mathbb{H}^3, \chi)$ directly.
It requires $m \in \mathbb{Z}_{\geq 6}$, and is useful for large $m$.
\end{remark}

For each $\zeta$ with $\zeta^6=1$, let $\psi^{\#}_{\zeta}: \mathbb{Z}[\omega] \rightarrow \mathbb{C}$ be given by
\begin{equation} \label{Psidefn}
\psi^{\#}_{\zeta}(u):=\frac{1}{N(\lambda^{m} r)} \sum_{\substack{a,d \pmod{\lambda^m r} \\ a,d \equiv 1 \pmod{3} \\ ad \equiv 1 \pmod{\lambda^{m} r} }}
\psi(-\zeta^{-1} d) \Big (\frac{\zeta \lambda^{m-1} r}{d} \Big )_{3}
\check{e} \Big ({\frac{a u}{\zeta \lambda^{m} r}} \Big ), \quad u \in \mathbb{Z}[\omega].
\end{equation}
The function $\psi^{\#}_{\zeta}$ is periodic modulo $\lambda^{m} r$.

\begin{prop} \label{funceqprop}
Let $f \in L^2(\Gamma_2 \backslash \mathbb{H}^3,\chi)$ be a cusp form with spectral parameter $\tau_f \in 1+i \mathbb{R}$, $m \in \mathbb{Z}_{\geq 6}$, $r \in \mathbb{Z}[\omega]$
with $r \equiv 1 \pmod{3}$, and $\psi: \mathbb{Z}[\omega] \rightarrow \mathbb{C}$ be   
a periodic function modulo $\lambda^{m} r$, supported only on residue 
classes coprime to $\lambda^m r$.
We have 
\begin{equation} \label{funceqalpha}
f(w;\widehat{\psi})=\sum_{\zeta} f
\Big ({-\frac{\overline{z}}{(\zeta^{-1} \lambda^{m-4} r)^2(|z|^2+v^2)}}, \frac{v}{|\lambda^{m-4} r|^2(|z|^2+v^2)}; \psi^{\#}_{\zeta^{-1}} \Big), \quad w=(z,v) \in \mathbb{H}^3,
\end{equation}
where $\widehat{\psi}$ and $\psi^{\#}_{\zeta}$ are given by \eqref{fouriertrans} and \eqref{Psidefn} respectively.
\end{prop}
\begin{proof}
We open the definition of the Fourier transform to obtain
\begin{equation} \label{congtwist}
f(w; \widehat{\psi})=\frac{1}{N(\lambda^{m} r)} \sum_{\zeta} \sum_{\substack{d \pmod{\lambda^{m} r}  \\ d \equiv 1 \pmod{3} \\ (d,r)=1 }} \psi(-\zeta d)
f \Big(z-\frac{d}{\zeta^{-1} \lambda^{m-4} r},v \Big).
\end{equation}
Given $\zeta^{-1} \lambda^{m-4} r \in \mathbb{Z}[\omega]$ (with $m \in \mathbb{Z}_{\geq 6}$),
and each $d \equiv 1 \pmod{3}$ 
in \eqref{congtwist} with $(d,r)=1$, 
there exists a matrix
\begin{equation} \label{gammadefn}
\gamma:=\begin{pMatrix}
{d}  {\lambda^4 b} 
{-\zeta^{-1} \lambda^{m-4} r}  {a}
\end{pMatrix} \in \Gamma_1(3). 
\end{equation} 
Note that the determinant equation of this matrix implies that $ad \equiv 1 \pmod{\lambda^m r}$.
A straightforward computation using \eqref{coordinates} shows that
\begin{equation} \label{gammarel1}
\Big(z-\frac{d}{\zeta^{-1} \lambda^{m-4} r},v \Big)=\gamma \Big(\frac{a}{\zeta^{-1} \lambda^{m-4} r}-\frac{\overline{z}}{(\zeta^{-1} \lambda^{m-4} r)^2(|z|^2+v^2)},
\frac{v}{| \lambda^{m-4} r|^2(|z|^2+v^2)} \Big).
\end{equation}
We use the fact that $f \in L^2(\Gamma_2 \backslash \mathbb{H}^3,\chi)$
to obtain
\begin{equation} \label{Fjtransform}
f \Big(z-\frac{d}{\zeta^{-1} \lambda^{m-4} r},v \Big)=\chi(\gamma) 
f \Big(\frac{a}{\zeta^{-1} \lambda^{m-4} r}-\frac{\overline{z}}{(\zeta^{-1} \lambda^{m-4} r)^2(|z|^2+v^2)}, \frac{v}{|\lambda^{m-4} r|^2(|z|^2+v^2)} \Big), 
\end{equation}
where 
\begin{align} \label{tildefac}
\chi(\gamma)=\Big( \frac{-\zeta^{-1} \lambda^{m-4} r}{d} \Big)_3=\Big( \frac{\zeta^{-1} \lambda^{m-1} r}{d} \Big)_3. 
\end{align}
We combine \eqref{Fjtransform}--\eqref{tildefac} 
in \eqref{congtwist}.
We then use the Fourier expansion \eqref{introfourier}
to open $f$, and then
assemble the sum over $d$ (equivalently $a$).
This yields the result.
\end{proof}

\begin{corollary} \label{funceqcor}
Let the notation be as in Proposition \ref{funceqprop} and $n \in \mathbb{Z}$.
For $v>0$ we have
\begin{equation} \label{funceqalphazero}
f(v;\widehat{\psi},n)=\frac{(-1)^n}{N(\lambda^{m-4}r)^{|n|} v^{2 |n|}} \sum_{\zeta} \Big( \frac{\overline{\zeta^{-1} \lambda^{m-4} r}}{\zeta^{-1} \lambda^{m-4} r}  \Big)^{-n} f
\Big( \frac{1}{|\lambda^{m-4} r|^2 v}; \psi^{\#}_{\zeta^{-1}},-n \Big).
\end{equation}
\end{corollary}
\begin{proof}
Setting $z=0$ in Proposition \ref{funceqprop} (in particular, \eqref{funceqalpha}) gives the result for $n=0$. If $n>0$,
we write $|z|^2=z \overline{z}$ and apply the operator
\begin{equation}
\frac{1}{(2 \pi i)^{n}} \Big(\frac{\partial}{\partial z} \Big)^n \Big|_{z=0}
\end{equation}
to both sides of \eqref{funceqalpha}. If $n<0$, we write $|z|^2=z \overline{z}$ 
and apply the operator
\begin{equation}
\frac{1}{(2 \pi i)^{|n|}} \Big(\frac{\partial}{\partial \overline{z}} \Big)^{|n|} \Big \rvert_{z=0}
\end{equation}
to both sides of \eqref{funceqalpha} 
A computation with the chain rule yields the result.
\end{proof}

\begin{prop} \label{entirefunc}
Let the notation be as in Proposition \ref{funceqprop} and $n \in \mathbb{Z}$.
The completed Dirichlet series $\Lambda(s,f;\widehat{\psi},n)$
and $\Lambda(s,f;\psi^{\#}_{\zeta},n)$
both admit meromorphic continuations to entire functions, and satisfy
\begin{equation} \label{levelfunc}
 (-1)^n N(\lambda^{m-4} r)^{2s-1} \Lambda(s,f;\widehat{\psi},n)
 =\sum_{\zeta} \Big( \frac{\overline{\zeta^{-1} \lambda^{m-4} r}}{\zeta^{-1} \lambda^{m-4} r} \Big)^{-n}
 \Lambda(1 - s, f; \psi^{\#}_{\zeta^{-1}},-n).
 \end{equation}
\end{prop}

\begin{proof}
Recall that for $\Re s>1$ we have 
\begin{equation*}
\Lambda (s,f;\widehat{\psi},n)= \int_{0}^{\infty} f(v; \widehat{\psi},n) v^{2s+|n|-2} dv.
\end{equation*}
The function $f(v; \widehat{\psi},n)$ has exponential decay at $0$ and $\infty$ by \eqref{cuspform}, \eqref{twistformdiff},
and termwise differentiation of \eqref{fouriersigma} (with constant term identically zero).
Thus $\Lambda (s,f;\widehat{\psi},n)$ has analytic continuation to an entire function.
The argument for  $f(v; \psi^{\#}_{\zeta},n)$ is analogous.

We now prove \eqref{levelfunc}. We have 
\begin{equation} \label{splitint1}
    \Lambda(s,f; \widehat{\psi},n)= \int_{0}^{N(\lambda^{m-4} r)^{-1}} f(v; \widehat{\psi},n) v^{2s+|n|- 2} dv + \int_{N(\lambda^{m-4} r)^{-1}}^{\infty} f(v; \widehat{\psi},n) v^{2s+|n|- 2} dv.
  \end{equation}
After applying Corollary \ref{funceqcor}, interchanging the order of summation and integration,
and a change of variables, we obtain
  \begin{align}
     & \int_{0}^{N(\lambda^{m-4} r)^{-1}} f(v; \widehat{\psi},n) v^{2s+|n|- 2} dv \nonumber \\
    &  = \frac{(-1)^n}{N(\lambda^{m-4}r)^{|n|}} \sum_{\zeta} \Big( \frac{\overline{\zeta^{-1} \lambda^{m-4} r}}{\zeta^{-1} \lambda^{m-4} r}  \Big)^{-n} \int_{0}^{N(\lambda^{m-4} r)^{-1}} f \Big ( \frac{1}{v |\lambda^{m-4} r|^2} ; \psi^{\#}_{\zeta^{-1}},-n)  v^{2s-|n| - 2} dv
     \nonumber \\ 
    &=(-1)^n N(\lambda^{m-4} r)^{1-2s} \sum_{\zeta} \Big( \frac{\overline{\zeta^{-1} \lambda^{m-4} r}}{\zeta^{-1} \lambda^{m-4} r}  \Big)^{-n} \int_{1}^{\infty} f (v ; \psi^{\#}_{\zeta^{-1}},-n ) v^{- 2s+|n|} dv,  \label{splitintstar}
  \end{align}
 and 
   \begin{align}
   &  \int_{N(\lambda^{m-4} r)^{-1}}^{\infty} f(v; \widehat{\psi},n) v^{2s+|n|- 2} dv \nonumber \\
    &  =  \frac{(-1)^n}{N(\lambda^{m-4}r)^{|n|}} \sum_{\zeta} \Big( \frac{\overline{\zeta^{-1} \lambda^{m-4} r}}{\zeta^{-1} \lambda^{m-4} r}  \Big)^{-n} \int_{N(\lambda^{m-4} r)^{-1}}^{\infty}  f \Big ( \frac{1}{v |\lambda^{m-4} r|^2} ; \psi^{\#}_{\zeta^{-1}},-n \Big )  v^{2s-|n| - 2} dv
     \nonumber \\ 
    &=(-1)^n N(\lambda^{m-4} r)^{1-2s} \sum_{\zeta} \Big( \frac{\overline{\zeta^{-1} \lambda^{m-4} r}}{\zeta^{-1} \lambda^{m-4} r}  \Big)^{-n} \int_{0}^{1} f (v ;\psi^{\#}_{\zeta^{-1}},-n ) v^{- 2s+|n|} dv.  \label{splitintstar2}
  \end{align} 
  The result follows after substituting \eqref{splitintstar} and \eqref{splitintstar2} into \eqref{splitint1}.
  \end{proof} 
  
\subsection{Functional equation 2: $m \in \mathbb{Z}_{\geq 0}$ absolutely bounded (in particular, $0 \leq m \leq 5$)} 
Let $m \in \mathbb{Z}_{\geq 0}$.
The functional equation we prove in this section is valid for all $m \in \mathbb{Z}_{\geq 0}$, but is really only useful when $m$ is bounded
by an absolute constant.

Recall that $\Gamma:=\operatorname{SL}_2(\mathbb{Z}[\omega])$. Let $\Gamma^{\prime}$ 
be a subgroup of $\Gamma$
with $[\Gamma:\Gamma^{\prime}]<\infty$ and $\Gamma^{\prime} \subseteq \Gamma_1(9)$.
Then by \cite[Theorem~0.3.1]{Pro}
each cusp $\sigma \infty$ ($\sigma \in \Gamma$) of 
$\Gamma^{\prime}$ is essential with respect to $\chi$,
and if $\Gamma^{\prime}:=\Gamma_1(C)$ with $C \equiv 0 \pmod{9}$,
then
\begin{equation*}
\Gamma_1(C)_{\sigma}=C \mathbb{Z}[\omega] \quad 
\text{and} \quad \Gamma_1(C)_{\sigma}^{*}=(C \lambda)^{-1} \mathbb{Z}[\omega]. 
\end{equation*}

Suppose that $\psi^{\prime}: \mathbb{Z}[\omega] \rightarrow \mathbb{C}$ is periodic modulo $\lambda^m$, and that
$\psi^{\prime \prime}: \mathbb{Z}[\omega] \rightarrow \mathbb{C}$ is periodic modulo $r$ where
$r \equiv 1 \pmod{3}$. Let
\begin{equation} \label{underlinedefn}
\psi^{\prime \prime, \star}(u):=\frac{1}{N(r)} \sum_{\substack{a,d \pmod{r} \\ (\lambda^{2 m+4}a)(\lambda^{2 m+4} d) \equiv 1 \pmod{r} }}
\psi^{\prime \prime}(-d) \Big (\frac{\lambda^{2 m+4} d}{r} \Big )_{3}
\check{e} \Big ({\frac{a u}{r}} \Big ), \quad u \in \mathbb{Z}[\omega].
\end{equation}
The function $\psi^{\prime \prime,\star}$ is periodic modulo $r$.

Let $\gamma_{m,j} \in \Gamma_2$ for $j=1,\ldots,[\Gamma_2:\Gamma_1(\lambda^{2 m+4})]$,
be a fixed complete set of representatives
for $\Gamma_1(\lambda^{2 m+4}) \backslash \Gamma_2$. We have the convention that
$\gamma_{m,1}:=I$ for all $m \in \mathbb{Z}_{\geq 0}$. For each $j=1,\ldots,[\Gamma_2:\Gamma_1(\lambda^{2 m+4})]$,
let
\begin{equation}  \label{Flj}
(f \otimes \widehat{\psi^{\prime}})_j(w):=(f \otimes \widehat{\psi^{\prime}})(\gamma_{m,j}w), \quad w \in \mathbb{H}^3,
\end{equation}
each having Fourier expansion
\begin{equation} \label{Fljfourierexp}
(f \otimes \widehat{\psi}^{\prime})_j(w):=\sum_{\substack{\nu \neq 0 \\ \nu \in \lambda^{-2 m-3} \mathbb{Z}[\omega] }} 
\rho_{f \otimes \widehat{\psi}^{\prime} ,j}(\nu) v K_{\tau-1} (4 \pi \lvert \nu \rvert v) 
\check{e}(\nu z), \quad w \in \mathbb{H}^3,
\end{equation}
where $\rho_{f \otimes \widehat{\psi}^{\prime} ,j}(\nu) \in \mathbb{C}$.

If $g \in \Gamma_2$, then 
\begin{align} \label{commuterel}
& \gamma_{m,j} g=g_{m,j}(g) \gamma_{m,k_{m,j}(g)},  \\
& \quad \text{for some unique } g_{m,j}(g) \in \Gamma_1(\lambda^{2 m+4}) 
\text{ and } 1 \leq k_{m,j}(g) \leq [\Gamma_2:\Gamma_1(\lambda^{2 m+4})]. \nonumber
\end{align}
For any $g,h \in \Gamma_2$ we have
\begin{equation*}
g_{m,j}(gh)=g_{m,j}(g) g_{m,k_{m,j}(g)}(h) \quad \text{and} \quad k_{m,j}(gh)=k_{m,k_{m,j}(g)}(h).
\end{equation*}

\begin{remark}
Using \eqref{commuterel} we see that for $g \in \Gamma_2$ we have,
\begin{equation} \label{commauto}
(f \otimes \widehat{\psi}^{\prime})_j(gw)=\chi(g_{m,j}(g)) (f \otimes \widehat{\psi}^{\prime})_{k_{m,j}(g)}(w), \quad w \in \mathbb{H}^3.
\end{equation}
Since $\Gamma_1(\lambda^{2m+4})$ is a normal subgroup of $\Gamma_2$ for $m \in \mathbb{Z}_{\geq 0}$
(it is also a normal subgroup of $\Gamma$) 
we have
\begin{equation} \label{normalspec}
k_{m,j}(g)=j \quad \text{for} \quad g \in \Gamma_1(\lambda^{2m+4}) \quad \text{and all} \quad j.
\end{equation}
Then by \cite[Lemma~2.1]{Pat4} we have
\begin{equation*}
\chi(\gamma g \gamma^{-1})=\chi(g) \quad \text{for} \quad g \in \Gamma_1(\lambda^{2m+4}) \quad \text{and} \quad \gamma \in \Gamma_2.
\end{equation*}
Thus for each $j$ we have
\begin{equation} \label{twistcuspj}
(f \otimes \widehat{\psi}^{\prime})_{j} \in L^2(\Gamma_1(\lambda^{2m+4}) \backslash \mathbb{H}^3,\chi,\tau) \quad \text{is a cusp form}.
\end{equation}
\end{remark}

Following \eqref{twistdef2}--\eqref{twistdef2a} we also have the functions 
$(f \otimes \widehat{\psi}^{\prime})_{j}(\cdot,\Psi,n)$ and their associated Fourier expansions for each $n \in \mathbb{Z}$.

\begin{prop} \label{funceqprop2}
Let $f \in L^2(\Gamma_2 \backslash \mathbb{H}^3,\chi)$ be a cusp form with spectral parameter $\tau_f \in 1+i \mathbb{R}$,
$m \in \mathbb{Z}_{\geq 0}$,
$1 \leq j \leq [\Gamma_2:\Gamma_1(\lambda^{2m+4})]$ an integer,
$r \in \mathbb{Z}[\omega]$ with $r \equiv 1 \pmod{3}$, 
and
$\psi^{\prime}:\mathbb{Z}[\omega] \rightarrow \mathbb{C}$ (resp. $\psi^{\prime \prime}:\mathbb{Z}[\omega] \rightarrow \mathbb{C}$) 
be periodic functions
modulo $\lambda^m$ (resp. $r$).
Further assume that $\psi^{\prime \prime}$ is supported only on residue 
classes coprime to $r$. 

Then there exist an integer 
$1 \leq \mathfrak{c}(m,j;r) \leq [\Gamma_2:\Gamma_1(\lambda^{2m+4})]$, 
and cube root of unity $\omega(m,j;r)$ such that,
\begin{align} \label{funceqalpha2}
(f \otimes \widehat{\psi}^{\prime})_j(w;\widehat{\psi}^{\prime \prime})=\omega(m,j;r) (f \otimes \widehat{\psi}^{\prime})_{\mathfrak{c}(m,j;r)}
& \Big ({-\frac{\overline{z}}{r^2(|z|^2+v^2)}}, \frac{v}{|r|^2(|z|^2+v^2)}; \psi^{\prime \prime,\star} \Big), \nonumber  \\
&\quad w=(z,v) \in \mathbb{H}^3,
\end{align}
where $\psi^{\prime \prime,\star}$ is given in \eqref{underlinedefn}.
Both $\mathfrak{c}(m,j;r)$ and $\omega(m,j;r)$ depend only on $m \in \mathbb{Z}_{\geq 0}$, $j \in \mathbb{Z}_{\geq 1}$, 
and the residue class $r \pmod{\lambda^{2 m+4}}$. 
\end{prop}

\begin{remark}
The reason why the functional equation proved in this section is only useful for $m$ bounded by an absolute constant is
because we use the automorphy for each
$(f \otimes  \widehat{\psi}^{\prime})_j \in L^2(\Gamma_1(\lambda^{2m+4}) \backslash \mathbb{H}^3,\chi)$.
\end{remark}

\begin{proof}
We adapt the proof of \cite[Lemma~5.2]{DR}.
We open the definition of the Fourier transform and obtain
\begin{equation} \label{congtwist2}
(f \otimes \widehat{\psi}^{\prime})_j(w; \widehat{\psi}^{\prime \prime})=\frac{1}{N(r)} \sum_{\substack{d \pmod{r}  \\ (d,r)=1 }} \psi^{\prime \prime}(-d)
(f \otimes \widehat{\psi}^{\prime})_j \Big(z-\frac{\lambda^{2 m+4} d}{r},v \Big).
\end{equation}
Given $r \equiv 1 \pmod{3}$ and
each $d \in \mathbb{Z}[\omega]$ in \eqref{congtwist2}, 
we have $(r,\lambda^{2m+4}d)=1$. Thus
there exists a matrix
\begin{equation*}
\begin{pMatrix}
{r} {-\lambda^{2m+4} a}
{\lambda^{2m+4} d } {b} \in \Gamma_1(3),
\end{pMatrix}
\end{equation*}
and hence there exists 
\begin{equation} \label{gammadefn2}
\gamma:=\begin{pMatrix}
{0} {1}
{-1} {0}
\end{pMatrix}
\begin{pMatrix}
{r} {-\lambda^{2m+4} a}
{\lambda^{2m+4} d } {b}
\end{pMatrix}
=
\begin{pMatrix}
{\lambda^{2 m+4} d}  {b} 
{-r}  {\lambda^{2 m+4} a}
\end{pMatrix} \in \Gamma_2. 
\end{equation} 
Note that we implicitly we used \eqref{Gamma2} in the above display.
Also note we have the determinant equation
\begin{equation} \label{deter}
\lambda^{4m+8} ad+br=1.
\end{equation}

A straightforward computation using \eqref{coordinates} shows that
\begin{equation} \label{gammarel1b}
\Big(z-\frac{\lambda^{2 m+4} d}{r},v \Big)=\gamma \Big(\frac{\lambda^{2 m+4} a}{r}-\frac{\overline{z}}{r^2(|z|^2+v^2)},
\frac{v}{|r|^2(|z|^2+v^2)} \Big).
\end{equation}

We now carefully factorise the $\gamma$ in \eqref{gammadefn2} as a word in
$P$, $T$ and $E$ so that \eqref{gammarel1b} and automorphy of $(f \otimes \widehat{\psi}^{\prime})_j$ can be used in 
\eqref{congtwist2}. 
For each $x+y \omega \in \mathbb{Z}[\omega]$, $x,y \in \mathbb{Z}$,
let
\begin{equation*}
A(x+y \omega):=PT^{-x} P T^{-x+y}P
=\begin{pMatrix}
1  {x+y \omega} 
0  1 
\end{pMatrix}.
\end{equation*}
For each $r,b \in \mathbb{Z}[\omega]$ occurring in \eqref{gammadefn2}, let
\begin{equation*}
S(r,b):=E^3 A(r) E A(b) E A(r) = 
\begin{pMatrix}
  {b}  {-1 + br} 
  {1 - br}  {2 r - b r^2}
\end{pMatrix} \in \Gamma_1(3).
\end{equation*}
Then
\begin{align*} \label{gammatildedef}
S(r,b) E \gamma &= 
\begin{pMatrix}
  {-\lambda^{2 m+4} d + br + \lambda^{2 m+4} bdr}  {- b - \lambda^{2 m+4} a b+ b^2 r} 
  {r + 2 \lambda^{2 m+4} dr - br^2 - \lambda^{2 m+4} b d r^2}  {- \lambda^{2 m+4} a + 2 br +\lambda^{2 m+4} a b r - b^2 r^2}
\end{pMatrix} \nonumber \\
&=: \widetilde{\gamma}.
\end{align*}
Using \eqref{deter} we see that $\widetilde{\gamma} \in \Gamma_1(\lambda^{2 m+4})$
and we write 
\begin{equation} \label{gammarel3}
\gamma=E^3 S(r,b)^{-1} \widetilde{\gamma}. 
\end{equation}
We use \eqref{gammarel1b}, \eqref{gammarel3}, \eqref{commauto}, and \eqref{twistcuspj}
to obtain
\begin{align} \label{Fjtransform2}
(f \otimes \widehat{\psi}^{\prime})_j \Big(z-\frac{\lambda^{2 m+4} d}{r},v \Big)&=
\chi \big(g_{m,j}(E^3 S(r,b)^{-1}) \big) \cdot \chi(\widetilde{\gamma}) \nonumber  \\
& \times( f \otimes \widehat{\psi}^{\prime})_{k_{m,j}(E^3 S(r,b)^{-1})}
\Big(\frac{\lambda^{2 m+4} a}{r}-\frac{\overline{z}}{r^2(|z|^2+v^2)}, \frac{v}{|r|^2(|z|^2+v^2)} \Big). 
\end{align}

By \eqref{normalspec} the integer $k_{m,j}(E^3 S(r,b)^{-1})$ depends only on $m,j \in \mathbb{N}$ 
and matrix residue class
\begin{equation*}
E^3 S(r,b)^{-1}=\begin{pMatrix}
{-1+br} {b} {-2r+br^2} {-1+br}
\end{pMatrix}
\pmod{\lambda^{2m+4}}.
\end{equation*}
Thus the integer $k_{m,j}(E^3 S(r,b)^{-1})$ depends only on
$m,j \in \mathbb{N}$ and the residue class $r \pmod{\lambda^{2 m+4}}$, since $b \pmod{\lambda^{2m+4}}$
is determined by \eqref{deter}. By \eqref{commuterel} we have 
\begin{equation*}
g_{m,j}(E^3 S(r,b)^{-1})= \gamma_{m,j} E^3 S(r,b)^{-1} \gamma_{m,k_{m,j}(E^3 S(r,b)^{-1})}^{-1} \in \Gamma_1(\lambda^{2 m+4}),
\end{equation*}
and each matrix in the product on the right side is an element of $\Gamma_2$. 
Thus
\begin{align*}
& \chi(g_{m,j}(E^3 S(r,b)^{-1})) \nonumber \\
&=\chi(\gamma_{m,j}) \chi(E^3) \chi(S(r,b)^{-1}) \chi(\gamma_{m,k_{m,j}(E^3 S(r,b)^{-1})}^{-1}) \nonumber \\
&=\chi(\gamma_{m,j}) \overline{\chi(\gamma_{m,k_{m,j(E^3 S(r,b)^{-1})}})},
\end{align*}
is a cube root of unity depending only on $m \in \mathbb{Z}_{\geq 0}$, $j \in \mathbb{Z}_{\geq 1}$ and the residue class $r \pmod{\lambda^{2 m+4}}$.
For ease of notation we re-label 
\begin{align}
\mathfrak{c}(m,j;r)&:=k_{m,j}(E^3 S(r,b)^{-1}); \label{Qdef} \\
\omega(m,j;r)&:=\chi(\gamma_{m,j}) \overline{\chi(\gamma_{m,k_{m,j(E^3 S(r,b)^{-1})}})}. \label{xidef}
\end{align}
A computation following \cite[pg.~23]{DR} establishes that
\begin{align} \label{tildefac2}
\chi(\widetilde{\gamma})=\Big( \frac{\lambda^{2 m+4} d}{r} \Big)_3. 
\end{align}

We combine \eqref{Fjtransform2}--\eqref{tildefac2} 
in \eqref{congtwist2}.
We then use the Fourier expansion \eqref{Fljfourierexp}
to open $(f \otimes \widehat{\psi}^{\prime})_{\mathfrak{c}(m,j;r)}$,
and assembling the sum over $d$ (equivalently $a$) shows that
\begin{equation*} 
(f \otimes \widehat{\psi}^{\prime})_j(w;\widehat{\psi}^{\prime \prime})=\omega(m,j;r) (f \otimes \widehat{\psi}^{\prime})_{\mathfrak{c}(m,j;r)} \Big ({-\frac{\overline{z}}{r^2(|z|^2+v^2)}}, \frac{v}{|r|^2(|z|^2+v^2)};\psi^{\prime \prime,\star} \Big),
\end{equation*}
as required.
\end{proof}

\begin{corollary} \label{funceqcor2}
Let the notation be as in Proposition \ref{funceqprop2} and $n \in \mathbb{Z}$. For $v>0$ we have 
\begin{equation} \label{funceqalphazero2}
(f \otimes \widehat{\psi}^{\prime})_{j}(v;\widehat{\psi}^{\prime \prime},n)=\frac{(-1)^n  \omega(m,j;r)}{N(r)^{|n|} v^{2 |n|}}
\Big( \frac{\overline{r}}{r}  \Big)^{-n} (f \otimes \widehat{\psi}^{\prime})_{\mathfrak{c}(m,j;r)}
\Big( \frac{1}{|r|^2 v}; \psi^{\prime \prime,\star},-n \Big).
\end{equation}
\end{corollary}
The proof is analogous to that of Corollary \ref{funceqcor} so we omit it.

\begin{prop} \label{entirefunc2}
Let the notation be as in Proposition \ref{funceqprop2} and $n \in \mathbb{Z}$.
The completed Dirichlet series $\Lambda(s,(f \otimes \widehat{\psi}^{\prime})_j;\widehat{\psi}^{\prime \prime},n)$
and $\Lambda(s,(f \otimes \widehat{\psi}^{\prime})_j;\psi^{\prime \prime,\star},n)$ both
admit meromorphic continuation to an entire function, and satisfy
\begin{equation} \label{levlefunc2}
(-1)^n N(r)^{2s-1} \Big( \frac{\overline{r}}{r} \Big)^n \Lambda(s,(f \otimes \widehat{\psi}^{\prime})_j;\widehat{\psi}^{\prime \prime},n) 
 =\omega(m,j;r) \Lambda(1 - s, (f \otimes \widehat{\psi}^{\prime})_{\mathfrak{c}(m,j;r)}; \psi^{\prime \prime,\star},-n).
 \end{equation}
\end{prop}

Proposition \ref{entirefunc2} follows from Corollary \ref{funceqcor2}, and the proof is analogous to 
that of Proposition \ref{entirefunc}. We omit the proof.

\subsection{Level aspect Voronoi formula}
We now prove a Voronoi summation formula 
for the Fourier coefficients for the form $f(w;\lambda^{\ell} q,\eta)$
given in \eqref{modqtwist}.

We recall some basic facts concerning the complex Mellin transform.
Let $\mathbb{C}^{\times}:=\mathbb{C} \setminus \{0\}$. Let $K,M \geq 1$ and
$V_{K,M} \in {C}^{\infty}_{c}(\mathbb{C}^{\times})$ have compact support 
contained in a fixed annulus $c_0\leq |z|\leq100$, with $c_0>0$ independent of $K,M$, and also satisfy
\begin{equation} \label{complexderiv}
\frac{\partial^{i+j}}{{\partial x^{i} \partial y^{j}}} V_{K,M}(z) \ll_{i,j} M K^{i+j} \quad \text{for all} \quad z \in \mathbb{C}^{\times}.
\end{equation}
The complex Mellin transform is given by 
\begin{equation}
\widehat{V}_{K,M}(s,n):=\int_{\mathbb{C^{\times}}} V_{K,M}(z) |z|^{2s} (z/|z|)^{-n} d_{\times} z,
\end{equation}
for $s \in \mathbb{C}$ and $n \in \mathbb{Z}$, where $d_{\times} z:=|z|^{-2} dx dy$. 
Note that $\widehat{V}_{K,M}(s,n)$ is entire with respect to $s$ for each $n \in \mathbb{Z}$.
After making a change of variables $z=r e^{i\theta}$ with $r \in (0,\infty)$ and 
$\theta \in [0,2 \pi)$, we obtain
\begin{equation}
\widehat{V}_{K,M}(s,n)=\int_{0}^{\infty} \int_{0}^{2 \pi} V_{K,M}(re^{i\theta}) r^{2s-1} e^{-in\theta} d \theta dr.
\end{equation}
After repeated integration by parts, we obtain
\begin{equation*}
\widehat{V}_{K,M}(s,n) \ll_{j,k} M \cdot \min \Big \{1,  \frac{K^{j+k}}{|(2s)_j| (1+|n|)^k} \Big \},
\end{equation*}
for $j,k \in \mathbb{Z}_{\geq 0}$, $s \in \mathbb{C}$ in a fixed vertical strip, and
$n \in \mathbb{Z}$,
It follows for $D_1,D_2 \geq 0$, we have 
\begin{equation} \label{Kmellinbound}
\widehat{V}_{K,M}(s,n) \ll_{D_1,D_2} \frac{M K^{D_1+D_2}}{(1+|s|)^{D_1}(1+|n|)^{D_2}},
\end{equation}
for $s \in \mathbb{C}$ in a fixed vertical strip, and $n \in \mathbb{Z}$. The implied constants may depend on this strip and on $c_0$.
The complex Mellin inversion formula is given by 
\begin{equation} \label{mellininvert}
V_{K,M}(z)=\frac{1}{2 \pi^2 i} \sum_{n \in \mathbb{Z}} \int_{(\sigma)} \widehat{V}_{K,M}(s,n) |z|^{-2s} (z/|z|)^{n} ds,
\end{equation}
for $\sigma>0$, $z \in \mathbb{C}$, and $n \in \mathbb{Z}$.

\begin{remark} \label{radial}
Suppose further that $V_{K,M}$ is radial i.e. $V_{K,M}(re(\theta))=V_K(r)$ for all $\theta \in \mathbb{R}$. 
Then 
\begin{equation}
\widehat{V}_{K,M}(s,n)=\delta_{n=0} 2 \pi \cdot \int_{0}^{\infty} V_{K,M}(r) r^{2s-1} dr =\delta_{n=0}  2 \pi \cdot \widehat{V}_{K,M}(2s)=\delta_{n=0}  \pi \cdot \widehat{W}_{K,M}(s)
\end{equation}
where $\widehat{V}_{K,M}(s)$ denotes the usual Mellin transform for functions on $(0,\infty)$,
and $W_{K,M}$ is such that $W_{K,M}(r)=V_{K,M}(\sqrt{r})$.
Then
\eqref{mellininvert} becomes the standard Mellin inversion formula for functions on $(0,\infty)$ after a change of variable in $s$.
\end{remark}

\begin{prop} \label{funceqvor}
Let $f \in L^2(\Gamma_2 \backslash \mathbb{H}^3,\chi)$ be a cusp form with spectral parameter $\tau_f \in 1+i \mathbb{R}$,
$\ell \in \mathbb{Z}_{\geq 0}$, $q \in \mathbb{Z}[\omega]$ with $q \equiv 1 \pmod{3}$, 
$\eta \in \mathbb{Z}[\omega]/\lambda^{\ell} q \mathbb{Z}[\omega]$, 
and 
$V_{K,M} \in C^{\infty}_{c}(\mathbb{C}^{\times})$ be a smooth function 
with compact support in the fixed annulus $c_0\leq|z|\leq100$ satisfying \eqref{complexderiv} for some $K,M \geq 1$.
Then for $X>0$  we have
\begin{align} \label{mastervoronoi}
 & \sum_{\substack{\nu \in \lambda^{-3} \mathbb{Z}[\omega] \\  \lambda^{3} \nu \equiv \eta \pmod{\lambda^{\ell} q} }}
\rho_{f}(\nu)
V_{K,M}(\nu/\sqrt{X}) \nonumber \\
&= \frac{X}{N(\lambda^{\ell+2} q)} 
\sum_{r \mid q} \sum_{n \in \mathbb{Z}} (-1)^n \sum_{m=0}^{\ell+2} \sum_{p=1}^2 Z_{pf}(X,\lambda^m r,\eta,n;\dot{V}_{K,M})
\end{align}
where 
 \begin{align}
& Z_{1f}(X,\lambda^m r,\eta,n;\dot{V}_{K,M}) \nonumber \\
&:= \delta_{0 \leq m \leq \min \{5,\ell+2\}} \cdot N(\lambda^m) \Big( \frac{\overline{r}}{r} \Big)^{-n} \omega(m,1;r) \nonumber \\
& \times \sum_{\nu \in \lambda^{-2m-3} \mathbb{Z}[\omega]}  
\rho_{f \otimes \widehat{\psi_{\lambda^m}(\cdot)_{\lambda \eta}},\mathfrak{c}(m,1;r)}(\nu) \Big( \frac{\nu}{|\nu|} \Big)^{-n} \nonumber \\
& \times \psi^{\star}_r(\cdot)_{\lambda^{2m+1} \eta}(\lambda^{2m+4} \nu ) \dot{V}_{K,M} \Big( \frac{N(\nu)}{N(r)^2/X},n \Big), \label{Z1r} \\
 & Z_{2f}(X,\lambda^m r,\eta,n;\dot{V}_{K,M}) \nonumber \\
&:= \delta_{6 \leq m \leq \ell+2} \cdot N(\lambda^4) \sum_{\zeta}  \Big( \frac{\overline{\zeta^{-1} \lambda^{m-4} r}}{\zeta^{-1} \lambda^{m-4} r}  \Big)^{-n}
\sum_{\nu \in \lambda^{-3} \mathbb{Z}[\omega]} \rho_f(\nu) \Big( \frac{\nu}{|\nu|} \Big)^{-n}  \nonumber \\
& \times \psi^{\#}_{\lambda^m r}(\cdot)_{\lambda \eta,\zeta^{-1}}(\lambda^4 \nu) \dot{V}_{K,M} \Big( \frac{N(\nu)}{N(\lambda^{m-4} r)^2/X},n \Big), \label{Z2r}
\end{align}
where $\psi^{\#}_{\zeta}$ and $\psi^{\star}$ are given in 
\eqref{Psidefn} and \eqref{underlinedefn} (with $\psi^{\prime \prime} \rightarrow \psi$) respectively,
$\dot{V}_{K,M}(\cdot,n): (0,\infty) \rightarrow \mathbb{C}$ is given by
\begin{equation} \label{Wtrans}
\dot{V}_{K,M}(Y,n):=\frac{1}{2 \pi^2 i} \int_{(2)} Y^{-s}
\frac{G_{\infty}(s,\tau_f,n)}{G_{\infty}(1-s,\tau_f,n)} \widehat{V}_{K,M}(1-s,n) ds,
\end{equation}
$G_{\infty}(s,\tau,n)$ is given in \eqref{Ginfty}, and $\omega(m,j,r)$ and $\mathfrak{c}(m,j,r)$
are both as in Proposition \ref{funceqprop2}.
\end{prop}

\begin{remark} \label{radial2}
From Remark \ref{radial} we see that if $V_K$ is radial then only $n=0$ is relevant on the right side
of \eqref{mastervoronoi}. In this case $n$ is omitted from the notation.
\end{remark}

\begin{proof}
Recall the definition of the function $\psi_{\lambda^m r}(\cdot)_{\eta}$ in \eqref{modqeta}, and its 
Fourier transform $\widehat{\psi}_{\lambda^m r}(\cdot)_{\eta}$ in \eqref{normram}.
We apply complex Mellin inversion \eqref{mellininvert} to the smooth function $V_K$,
Lemma \ref{orthog}, and then 
interchange of the order of integration and summation by absolute convergence. 
This yields
\begin{align} \label{mellininversion}
& \sum_{\substack{\nu \in \lambda^{-3} \mathbb{Z}[\omega] \\  \lambda^{3} \nu \equiv \eta \pmod{\lambda^{\ell} q} }}
\rho_{f}(\nu)
V_{K,M}(\nu/\sqrt{X}) \nonumber \\
&=\frac{1}{2 \pi^2 i} \frac{1}{N(\lambda^{\ell+2} q)} \sum_{r \mid q} \Big( \sum_{m=0}^{\min \{5,\ell+2\}}+\sum_{m=6}^{\ell+2} \Big) N(\lambda^m r) \nonumber \\
& \times \sum_{n \in \mathbb{Z}} \int_{(2)} \widehat{V}_{K,M}(s,n) X^{s} \mathcal{D}(s,f; \widehat{\psi_{\lambda^{m} r}(\cdot)_{\lambda \eta}},n) ds.
\end{align}

The Chinese Remainder Theorem implies that
\begin{equation} \label{mult}
\widehat{\psi_{\lambda^m r}(\cdot)_{{\lambda \eta}}}(u)=\widehat{\psi_{\lambda^m}(\cdot)_{\lambda \eta}}(u) \widehat{\psi_{r}(\cdot)_{\lambda \eta}}(u), \quad u \in \mathbb{Z}[\omega],
\end{equation}
and by a change of variables we have
\begin{equation} \label{rescale}
\widehat{\psi_{r}(\cdot)_{\lambda^{2m+1} \eta}}(\lambda^{2m} u)=\widehat{\psi_r(\cdot)_{\lambda \eta}}(u), \quad u \in \mathbb{Z}[\omega].
\end{equation}
Recall the definition of twisting \eqref{twistdef2} and the convention in Remark \ref{twistview}.
Using \eqref{cuspform} we see that $f \otimes  \widehat{\psi_{\lambda^m}(\cdot)_{\lambda \eta}} \in L^2(\Gamma_1(\lambda^{2m+4}) \backslash \mathbb{H}^3,\chi,\tau)$ 
is a cusp form. Using \eqref{mult} and \eqref{rescale} we obtain
\begin{align*}
f(w;\widehat{\psi_{\lambda^m r}(\cdot)_{\lambda \eta}},n)&=f(w;\widehat{\psi_{\lambda^m}(\cdot)_{\lambda \eta}} \widehat{\psi_r (\cdot)_{\lambda^{2m+1} \eta}} \circ \lambda^{2m},n) \\
&=(f \otimes \widehat{\psi_{\lambda^m}(\cdot)_{\lambda \eta}})(w;\widehat{\psi_r(\cdot)_{\lambda^{2m+1} \eta}},n),
\end{align*}
for all $0 \leq m \leq \ell+2$, $r \mid q$, and $\eta \in \mathbb{Z}[\omega]$. The
analogous Dirichlet series identity reads
 \begin{equation} \label{Didentity}
\mathcal{D}(s,f; \widehat{\psi_{\lambda^m r}(\cdot)_{\lambda \eta}},n)
=\mathcal{D}(s,f \otimes \widehat{\psi_{\lambda^m}(\cdot)_{\lambda \eta}}; \widehat{\psi_{r}(\cdot)_{\lambda^{2m+1} \eta}},n), \quad \Re(s)>1.
\end{equation}

After substituting \eqref{Didentity} into \eqref{mellininversion} for each $0 \leq m \leq \min \{5,\ell+2\}$, we deduce that
the right side of \eqref{mellininversion} is equal to 
\begin{align} \label{mellininversion2}
&\frac{ 1}{2 \pi^2 i} \frac{1}{N(\lambda^{\ell+2} q)} \sum_{r \mid q} \sum_{n \in \mathbb{Z}} 
\Big( \sum_{m=0}^{\min \{5,\ell+2\}} N(\lambda^m r) \nonumber \\
& \times \int_{(2)} \widehat{V}_{K,M}(s,n) X^{s} \mathcal{D}(s,f \otimes \widehat{\psi_{\lambda^m}(\cdot)_{\lambda \eta}}; \widehat{\psi_{r}(\cdot)_{\lambda^{2m+1} \eta}},n) ds  \nonumber \\ 
&+\sum_{m=6}^{\ell+2} N(\lambda^m r)
\int_{(2)} \widehat{V}_{K,M}(s,n) X^{s} \mathcal{D}(s,f; \widehat{\psi_{\lambda^m r}(\cdot)_{\lambda \eta}},n) ds \Big).
\end{align}

Both of the integrands in \eqref{mellininversion2} are entire by Proposition \ref{entirefunc}, Proposition \ref{entirefunc2} and Lemma
\ref{mellin1}.
We shift the contour in \eqref{mellininversion}
to $\Re(s)=-1$ and then use
the functional equations \eqref{levelfunc} and \eqref{levlefunc2}. We see that \eqref{mellininversion2} 
is equal to
\begin{align} \label{funcmellin}
& \frac{1}{2 \pi^2 i} \frac{1}{N(\lambda^{\ell+2} q)} \sum_{r \mid q} \sum_{n \in \mathbb{Z}} (-1)^n
\Big( \sum_{m=0}^{\min \{5,\ell+2\}} N(\lambda^m) N(r)^2 \Big( \frac{\overline{r}}{r} \Big)^{-n} \omega(m,1;r) \nonumber \\
& \times \int_{(-1)} \widehat{V}_{K,M}(s,n)  \Big( \frac{X}{N(r)^2} \Big)^{s} 
\frac{G_{\infty}(1-s,\tau_f,-n)}{G_{\infty}(s,\tau_f,n)} \nonumber \\
& \times \mathcal{D}(1-s,(f \otimes \widehat{\psi_{\lambda^m}(\cdot)_{\lambda \eta}})_{\mathfrak{c}(m,1;r)}; \psi^{\star}_{r}(\cdot)_{\lambda^{2m+1} \eta},-n ) ds  \nonumber \\ 
&+\sum_{m=6}^{\ell+2} \sum_{\zeta} N(\lambda^{2m-4}) N(r)^2 \Big( \frac{\overline{\zeta^{-1} \lambda^{m-4} r}}{\zeta^{-1} \lambda^{m-4} r}  \Big)^{-n}
\int_{(-1)} \widehat{V}_{K,M}(s,n) \Big(\frac{X}{N(\lambda^{m-4} r)^2} \Big)^{s} \nonumber \\
& \times \frac{G_{\infty}(1-s,\tau_f,-n)}{G_{\infty}(s,\tau_f,n)} 
\mathcal{D}(1-s,f;\psi^{\#}_{\lambda^m r}(\cdot)_{\lambda \eta,\zeta^{-1}},-n) ds \Big).
\end{align}
We make the change of variable $s \rightarrow 1-s$ in both integrals in \eqref{funcmellin}.
We then open both of the Dirichlet series in the region of absolute convergence,
and then interchange the order of summation and integration to obtain
\eqref{mastervoronoi}, with the transforms given by \eqref{Z1r}--\eqref{Wtrans}.
\end{proof}

We now compute the Archimedean and non-Archimedean transforms on the dual side of the Voronoi
formula in Proposition \ref{funceqvor}. Recall that $K_{\Gamma^{\prime},\sigma,\xi}(m,n,c)$
denotes a cubic Kloosterman attached to the cusp pair $(\sigma,\xi)$ of $\Gamma^{\prime}$, 
see \eqref{cubickloosterman}.

\begin{lemma} \label{unpack1}
Let $m \in \mathbb{Z}_{\geq 6}$, $r \in \mathbb{Z}[\omega]$ with $r \equiv 1 \pmod{3}$, 
$\eta \in \mathbb{Z}[\omega]/\lambda^m r \mathbb{Z}[\omega]$,
$\psi_{\lambda^{m} r}(\cdot)_{\eta}$ be as in \eqref{modqeta}, and
$\zeta$ be such that $\zeta^6=1$. Then for $\nu \in \lambda^{-3} \mathbb{Z}[\omega]$ we have
\begin{equation*}
\psi^{\#}_{\lambda^{m} r}(\cdot)_{\lambda \eta,\zeta}(\lambda^4 \nu)
=\frac{1}{N(\lambda^{m+1} r)} K_{\Gamma_1(3),\sigma,\sigma}(\lambda^3 \nu,\eta,\zeta \lambda^{m-1} r),
\end{equation*}
where $\psi^{\#}_{\zeta}$ is given in \eqref{Psidefn}, $\sigma=\pmatrix 1 0 0 1$, and the cubic Kloosterman sum is given in \eqref{cubickloosterman}.
\end{lemma}
\begin{proof}
By \eqref{Psidefn}, the left side is
\begin{align} \label{cubicintermed}
\frac{1}{N(\lambda^m r)}
\sum_{\substack{a,d\pmod{\lambda^m r}\\a,d\equiv1\pmod3\\ad\equiv1\pmod{\lambda^m r}}}
\Big(\frac{\zeta\lambda^{m-1}r}{d}\Big)_3
\check e\Big(\frac{a\lambda^3\nu+d\eta}{\zeta\lambda^{m-1}r}\Big).
\end{align}
The summand is unchanged when $a,d$ are changed by multiples of
$\lambda^{m-1}r$. For the cubic symbol this follows from reciprocity
and the supplementary laws, since $m\geq6$. Each inverse pair modulo
$\lambda^m r$ has exactly $N(\lambda)=3$ inverse-pair lifts modulo
$\lambda^{m+1}r$. Thus \eqref{cubicintermed} equals
\[
\frac{1}{N(\lambda^{m+1}r)}
\sum_{\substack{a,d\pmod{\lambda^{m+1}r}\\a,d\equiv1\pmod3\\ad\equiv1\pmod{\lambda^{m+1}r}}}
\Big(\frac{\zeta\lambda^{m-1}r}{d}\Big)_3
\check e\Big(\frac{a\lambda^3\nu+d\eta}{\zeta\lambda^{m-1}r}\Big),
\]
which is the asserted Kloosterman sum by Lemma \ref{kloostermanexp1}.
\end{proof}

\begin{lemma} \label{unpack2}
Let the notation be as in Lemma \ref{unpack1} and $m \in \mathbb{Z}_{\geq 0}$.
Then for $\nu \in \lambda^{-2m-3} \mathbb{Z}[\omega]$ we have
\begin{equation*}
\psi^{\star}_{r}(\cdot)_{\lambda^{2m+1} \eta}(\lambda^{2m+4} \nu)
=\frac{1}{N(r)} K_{\Gamma_1(3),\sigma,\xi}(\overline{\lambda^{2m+3}} (\lambda^{2m+3} \nu), \overline{\lambda^3} \eta,r),
\end{equation*}
where $\psi^{\star}$ is given in \eqref{underlinedefn}, $\sigma=\pmatrix 1 0 0 1$, $\xi=\pmatrix 0 {-1} 1 0 $,
and $\overline{\lambda^{\ell}} \in \mathbb{Z}[\omega]$ is such that $\overline{\lambda^{\ell}} \lambda^{\ell} \equiv 1 \pmod{r}$ for $\ell \in \mathbb{Z}_{\geq 0}$.
\end{lemma}
\begin{proof}
By definition
\begin{equation} \label{underlineintermed}
\psi^{\star}_{r}(\cdot)_{\lambda^{2m+1} \eta}(\lambda^{2m+4} \nu) 
=\frac{1}{N(r)} \sum_{\substack{a,d \pmod{r} \\ (\lambda^{2m+4} a) (\lambda^{2m+4} d) \equiv 1 \pmod{r} }}
\Big (\frac{\lambda^{2m+4} d}{r} \Big )_{3}
\check{e} \Big ({\frac{a (\lambda^{2m+4} \nu)+d (\lambda^{2m+1} \eta)}{r} } \Big ).
\end{equation}
The change of variables $a \rightarrow \overline{\lambda^{2m+4}} a \pmod{r}$ and 
$d \rightarrow \overline{\lambda^{2m+4}} d \pmod{r}$ shows that the right side of \eqref{underlineintermed}
is equal to
\begin{equation} \label{underlineintermed2}
\frac{1}{N(r)} \sum_{\substack{a,d \pmod{r} \\ a d \equiv 1 \pmod{r}}}
\Big (\frac{d}{r} \Big )_{3}
\check{e} \Big ({{\frac{\overline{\lambda^{2m+3}} a (\lambda^{2m+3} \nu)+\overline{\lambda^3} \eta d}{r} }} \Big ),
\end{equation}
and we can lift this to the sum
\begin{equation*}
\frac{1}{N(r)} \sum_{\substack{a,d \pmod{3r} \\ a,d \equiv 0 \pmod{3} \\ a d \equiv 1 \pmod{r}}}
\Big (\frac{d}{r} \Big )_{3}
\check{e} \Big (\frac{\overline{\lambda^{2m+3}} a (\lambda^{2m+3} \nu)+\overline{\lambda^3} \eta d}{r} \Big ),
\end{equation*}
and the result now follows from Lemma \ref{kloostermanexp2}.
\end{proof}

\begin{lemma} \label{unpack3}
Let $K,M \geq 1$ and $V_{K,M} \in C^{\infty}_{c}(\mathbb{C}^{\times})$ be a smooth function with compact support in $1\leq|z|\leq2$
whose derivatives satisfy \eqref{complexderiv}. Let $\tau \in 1+i \mathbb{R}$, $n \in \mathbb{Z}$, $G_{\infty}(s,\tau,n)$ be as in $\eqref{Ginfty}$, and 
$\dot{V}_{K,M}(\cdot,n):(0,\infty) \rightarrow \mathbb{C}$
be as in \eqref{Wtrans}. Then for $D_1>1/4$ and $D_2 \geq 0$ we have
\begin{equation*}
\dot{V}_{K,M}(Y,n) \ll_{\tau,D_1,D_2}  M K^{4(D_1+D_2)} Y^{-D_1} (|n|+1)^{4D_1-4D_2-2},
\end{equation*}
for all $Y>0$. When $n=0$, the same estimate holds for every $D_1>0$.
\end{lemma}

\begin{proof}
In the definition \eqref{Wtrans} we move the contour to $\Re(s)=D_1$.
Stirling's formula \cite[(5.11.1)]{NIST:DLMF} implies that
\begin{equation} \label{quotbd}
\frac{G_{\infty}(s,\tau,-n)}{G_{\infty}(1-s,\tau,n)} \asymp |s+|n|/2-\tfrac{1}{2}(\tau-1)|^{2D_1-1} \cdot |s+|n|/2+\tfrac{1}{2}(\tau-1)|^{2D_1-1},
\end{equation}
as $|\Im (s \pm \frac{1}{2}(\tau-1))| \rightarrow \infty$.
Using \eqref{Kmellinbound} (with $D_1 \rightarrow 4D_1$ and $D_2 \rightarrow 4D_2$) and \eqref{quotbd}
in \eqref{Wtrans} we obtain 
\begin{align*}
\dot{V}_{K,M}(Y,n) & \ll_{D_1,D_2} M K^{4(D_1+D_2)} Y^{-D_1} (1+|n|)^{-4D_2}  \nonumber \\
& \times \Big(\int_{(D_1)} \frac{|s+|n|/2-\tfrac{1}{2}(\tau-1)|^{2D_1-1} \cdot |s+|n|/2+\tfrac{1}{2}(\tau-1)|^{2D_1-1}}{(1+|1-s|)^{4D_1}} |ds| \Big).
\end{align*} 
To estimate this integral, put $N=1+|n|$. Stirling's formula and boundedness
on compact sets give the upper bound
\[
\ll_{\tau,D_1}\int_{\mathbb R}(N+|t|)^{4D_1-2}(1+|t|)^{-4D_1}\,dt.
\]
For $D_1>1/4$, the part $|t|\leq N$ is
$O_{D_1}(N^{4D_1-2})$, since $(1+|t|)^{-4D_1}$ is integrable.
The part $|t|>N$ is $O_{D_1}(N^{-1})$, which is bounded by the same quantity.
When $n=0$, the integrand is $(1+|t|)^{-2}$ for every $D_1>0$.
These bounds prove the assertion.
\end{proof}

\subsection{Level aspect Voronoi summation for multiple sums}
Here we record a Voronoi formula that is an iterated version of 
Proposition \ref{funceqvor}. Let
$\boldsymbol{z}=(z_1,z_2)=(x_{11}+i y_{12},x_{21}+i y_{22}) \in (\mathbb{C}^{\times})^2$, $x_{11}, y_{12},x_{21},y_{22} \in \mathbb{R}$. Let
$K,M \geq 1$, and $H_{K,M} \in C^{\infty}_{c}((\mathbb{C}^{\times})^2)$ be a smooth function 
with compact support in $\{\boldsymbol z:c_0\leq|z_i|\leq100\ (i=1,2)\}$, for fixed $c_0>0$, such that for any $\boldsymbol{i}=(i_{11},i_{12},i_{21},i_{22}) \in (\mathbb{Z}_{\geq 0})^4$ we have
\begin{equation} \label{complexderiv2}
\partial^{\boldsymbol{i}} H_{K,M}(\boldsymbol{z}) \ll_{\boldsymbol{i}} M K^{\sum_{1 \leq j,k \leq 2} i_{jk}} \quad
\boldsymbol{z} \in (\mathbb{C}^{\times})^2.
\end{equation}
If $M=1$ then $M$ is omitted from the notation and we write $H_{K}$.
For each $\boldsymbol{n}=(n_1,n_2) \in \mathbb{Z}^2$, consider the double complex Mellin transform
\begin{align} \label{doublemellin}
\widehat{\widehat{H}}_{K,M}(\boldsymbol{s},\boldsymbol{n}) 
&:=\iint_{(\mathbb{C}^{\times})^2} H_{K,M}(\boldsymbol{z}) \Big( \prod_{i=1}^{2}
|z_i|^{2s_i}\Big( \frac{z_i}{|z_i|} \Big)^{-n_i} \Big)  d_{\times} \boldsymbol{z},\quad \boldsymbol{s}=(s_1,s_2) \in \mathbb{C}^2,
\end{align}
where $d_{\times} \boldsymbol{z}:=d x_1 dy_1 dx_2  dy_2 /|z_1 z_2|^2$.
For $\boldsymbol{D}:=(D_{11},D_{12},D_{21},D_{22}) \in (\mathbb{R}_{\geq 0})^4$,
repeated integration by parts using polar coordinates yields the bound
\begin{equation} \label{Kmellinbound2}
\widehat{\widehat{H}}_{K,M}(\boldsymbol{s},\boldsymbol{n}) \ll_{\tau,\boldsymbol{D}} 
M K^{\sum_{1 \leq i,j \leq 2} D_{ij}} \cdot \prod_{i=1}^2 (1+|s_i|)^{-D_{i1}}(1+|n_i|)^{-D_{i2}}.
\end{equation}
Consider the function $\ddot{H}_{K,M}(\cdot,\boldsymbol{n}): (0,\infty)^2 \rightarrow \mathbb{C}$ 
given by 
\begin{align} \label{doubletilde}
 \ddot{H}_{K,M}(\boldsymbol{Y},\boldsymbol{n}) 
&:=\frac{1}{(2 \pi^2 i)^2} \int_{(2)} \int_{(2)} \Big(\prod_{i=1}^2 Y_i^{-s_i}
\frac{G_{\infty}(s_i,\tau_f,-n_i)}{G_{\infty}(1-s_i,\tau_f,n_i)} \Big)   
\widehat{\widehat{H}}_{K,M}(\boldsymbol{1}-\boldsymbol{s},(n_1,-n_2)) d \boldsymbol{s}, \nonumber \\
& \quad \boldsymbol{Y}=(Y_1,Y_2) \in (0,\infty)^2,
\end{align}
where $G_{\infty}(s,\tau,n)$ is given by \eqref{Ginfty} and  $d \boldsymbol{s}=ds_1 ds_2$. 
After moving the contours in \eqref{doubletilde} to $\Re(s_1)=D_{11}>1/4$ and $\Re(s_2)=D_{21}>1/4$,
observe that the integral estimate in the proof of Lemma \ref{unpack3} and \eqref{Kmellinbound2} applied to \eqref{doubletilde}
imply that
\begin{align} \label{doubletildebd}
\ddot{H}_{K,M}(\boldsymbol{Y},\boldsymbol{n}) & \ll_{\tau,\boldsymbol{D}} M K^{4 (\sum_{1 \leq i,j \leq 2} D_{ij})}  
\cdot \prod_{i=1}^{2} Y_i^{-D_{i1}}   (|n_i|+1)^{4D_{i1}-4D_{i2}-2}, \quad \boldsymbol{Y} \in (0,\infty)^2.
\end{align}

Mellin inversion and an iterated application of the functional equation in Proposition 
\ref{funceqvor} yields the following result. We omit the proof for the sake of brevity.

\begin{prop} \label{vor2}
Let $f \in L^2(\Gamma_2 \backslash \mathbb{H}^3,\chi)$ be a cusp form with spectral parameter $\tau_f \in 1+i \mathbb{R}$,
$\boldsymbol{\ell}= (\ell_1,\ell_2) \in (\mathbb{Z}_{\geq 0})^2$, $\boldsymbol{q}=(q_1,q_2) \in (\mathbb{Z}[\omega])^2$
with $q_1,q_2 \equiv 1 \pmod{3}$, and 
$\boldsymbol{\eta}=(\eta_1,\eta_2) \in \mathbb{Z}[\omega]/\lambda^{\ell_1} q_1 \mathbb{Z}[\omega] \times \mathbb{Z}[\omega]/\lambda^{\ell_2} q_2 \mathbb{Z}[\omega]$.
Let $H_{K,M} \in C^{\infty}_{c}((\mathbb{C}^{\times})^2)$ be a smooth function 
with compact support in $\{\boldsymbol z:c_0\leq|z_i|\leq100\ (i=1,2)\}$ satisfying \eqref{complexderiv2} for some $K,M \geq 1$.
Then for $\boldsymbol{X}=(X_1,X_2) \in (0,\infty)^2$  we have
\begin{align}
& \sum_{\substack{ \boldsymbol{\nu} \in (\lambda^{-3} \mathbb{Z}[\omega])^2 \\ \forall i: \lambda^3 \nu_i \equiv \eta_i \pmod{ \lambda^{\ell_i} q_i}  } }
\rho_{f}(\nu_1) \overline{\rho_{f}(\nu_2)} H_{K,M} \Big( \frac{\nu_1}{\sqrt{X_1}}, \frac{\nu_2}{\sqrt{X_2}} \Big) \nonumber \\
& = \frac{X_1 X_2}{N(\lambda^{\ell_1+2} q_1) N(\lambda^{\ell_2+2} q_2) } 
\sum_{\substack{ \boldsymbol{k} \in (\mathbb{Z}[\omega]/\lambda^{14} \mathbb{Z}[\omega])^2 \\ \forall i: k_i \equiv 1 \pmod{3} }}  
\sum_{ \substack{\boldsymbol{m}, \boldsymbol{r} \\ \forall i:  0 \leq m_i \leq \ell_i+2 \\ \forall i: r_i \mid q_i \\ \forall i: r_i \equiv k_i \pmod{\lambda^{14}}  }}  \nonumber \\
& \sum_{\boldsymbol{n} \in \mathbb{Z}^2} (-1)^{n_1+n_2}
\sum_{p=1}^{4} \mathscr{D}_{pf}(\boldsymbol{X},\boldsymbol{\lambda}^{\boldsymbol{m}} \boldsymbol{r}, \boldsymbol{\eta}, \boldsymbol{n};\ddot{H}_{K,M}), 
\end{align}
where 
\begin{align} \label{Y1}
& \mathscr{D}_{1f}(\boldsymbol{X},\boldsymbol{\lambda}^{\boldsymbol{m}} \boldsymbol{r}, \boldsymbol{\eta}, \boldsymbol{n} ;\ddot{H}_{K,M}) \nonumber \\
&:=\delta_{m_1 \in [0, \min \{5,\ell_1+2\}]} \cdot \delta_{m_2 \in [0, \min \{5,\ell_2+2\}]} \cdot
N(\lambda^{m_1}) N(\lambda^{m_2})   \nonumber \\
& \times \Big( \frac{\overline{r_1}}{r_1} \Big)^{-n_1} \Big( \frac{\overline{r_2}}{r_2} \Big)^{n_2}  \omega(m_1;1,k_1) \overline{\omega(m_2,1,k_2)} 
\nonumber \\
& \times \sum_{\substack{ \nu_1 \in \lambda^{-2m_1-3} \mathbb{Z}[\omega] \\ \nu_2 \in \lambda^{-2m_2-3} \mathbb{Z}[\omega] } } 
\rho_{f \otimes \widehat{\psi_{\lambda^{m_1}}(\cdot)_{\lambda \eta_1}},\mathfrak{c}(m_1,1;k_1)}(\nu_1) \nonumber \\
& \times \overline{\rho_{f \otimes \widehat{\psi_{\lambda^{m_2}}(\cdot)_{\lambda \eta_2}},\mathfrak{c}(m_2,1;k_2)}(\nu_2)}
\Big(\frac{\nu_1}{|\nu_1|} \Big)^{-n_1} \Big( \frac{\overline{\nu_2}}{|\nu_2|} \Big)^{-n_2}  \nonumber \\ 
& \times \psi^{\star}_{r_1}(\cdot)_{\lambda^{2m_1+1} \eta_1}(\lambda^{2m_1+4} \nu_1 )
\overline{\psi^{\star}_{r_2}(\cdot)_{\lambda^{2m_2+1} \eta_2}(\lambda^{2m_2+4} \nu_2 )} \nonumber \\
& \times \ddot{H}_{K,M} \Big(\frac{N(\nu_1)}{N(r_1)^2/X_1},\frac{N(\nu_2)}{N(r_2)^2/X_2},\boldsymbol{n} \Big) ;
\end{align}

\begin{align} \label{Y2}
& \mathscr{D}_{2f}(\boldsymbol{X},\boldsymbol{\lambda}^{\boldsymbol{m}} \boldsymbol{r}, \boldsymbol{\eta}, \boldsymbol{n};\ddot{H}_{K,M} ) \nonumber \\
&:=\delta_{m_1 \in [0, \min \{5,\ell_1+2\}]} \cdot \delta_{m_2 \in [6,\ell_2+2]} N(\lambda^{m_1}) N(\lambda^4) 
\Big(\frac{\overline{r_1}}{r_1}  \Big)^{-n_1}  \omega(m_1,1;k_1)  \nonumber \\
& \times \sum_{\substack{ \nu_1 \in \lambda^{-2m_1-3} \mathbb{Z}[\omega] \\ \nu_2 \in \lambda^{-3} \mathbb{Z}[\omega] }}
\rho_{f \otimes \widehat{\psi_{\lambda^{m_1}}(\cdot)_{\lambda \eta_1}},\mathfrak{c}(m_1,1;k_1)}(\nu_1) \overline{\rho_f(\nu_2)}
\Big( \frac{\nu_1}{|\nu_1|} \Big)^{-n_1} \Big( \frac{\overline{\nu_2}}{|\nu_2|} \Big)^{-n_2} \nonumber \\
& \times \sum_{\zeta_2}  \Big( \frac{\overline{\zeta_2^{-1} \lambda^{m_2-4} r_2 }}{\zeta_2^{-1} \lambda^{m_2-4} r_2 }  \Big)^{n_2}  \psi^{\star}_{r_1}(\cdot)_{\lambda^{2m_1+1} \eta_1}(\lambda^{2m_1+4} \nu_1 ) 
\cdot \overline{\psi^{\#}_{\lambda^{m_2} r_2}(\cdot)_{\lambda \eta_2,\zeta_2^{-1}}(\lambda^4 \nu_2)} \nonumber \\
& \times \ddot{H}_{K,M} \Big(\frac{N(\nu_1)}{N(r_1)^2/X_1},\frac{N(\nu_2)}{N(\lambda^{m_2-4} r_2)^2/X_2},\boldsymbol{n} \Big);
\end{align}

\begin{align} \label{Y3}
& \mathscr{D}_{3f}(\boldsymbol{X},\boldsymbol{\lambda}^{\boldsymbol{m}} \boldsymbol{r}, \boldsymbol{\eta}, \boldsymbol{n};\ddot{H}_{K,M}) \nonumber \\
&:= \delta_{m_1 \in [6,\ell_1+2]} \cdot \delta_{m_2 \in [0, \min \{5,\ell_2+2\}]} \cdot 
N(\lambda^4) N(\lambda^{m_2}) \Big( \frac{\overline{r_2}}{r_2} \Big)^{n_2} \overline{\omega(m_2,1;k_2)}   \nonumber \\
& \times \sum_{\substack{ \nu_1 \in \lambda^{-3} \mathbb{Z}[\omega] \\ \nu_2 \in \lambda^{-2m_2-3} \mathbb{Z}[\omega] }}
\rho_f(\nu_1)  \overline{\rho_{f \otimes \widehat{\psi_{\lambda^{m_2}}(\cdot)_{\lambda \eta_2}},\mathfrak{c}(m_2,1;k_2)}(\nu_2)}
\Big( \frac{\nu_1}{|\nu_1|} \Big)^{-n_1} \Big( \frac{\overline{\nu_2}}{|\nu_2|} \Big)^{-n_2} \nonumber \\
& \times \sum_{\zeta_1}  \Big( \frac{\overline{\zeta_1^{-1} \lambda^{m_1-4} r_1 }}{\zeta_1^{-1} \lambda^{m_1-4} r_1 }  \Big)^{-n_1}
 \psi^{\#}_{\lambda^{m_1} r_1}(\cdot)_{\lambda \eta_1,\zeta_1^{-1}}(\lambda^4 \nu_1)
\cdot \overline{\psi^{\star}_{r_2}(\cdot)_{\lambda^{2m_2+1} \eta_2}(\lambda^{2m_2+4} \nu_2 )} \nonumber \\ 
& \times \ddot{H}_{K,M} \Big(\frac{N(\nu_1)}{N(\lambda^{m_1-4} r_1)^2/X_1},\frac{N(\nu_2)}{N(r_2)^2/X_2}, \boldsymbol{n}  \Big);
\end{align}
 
\begin{align} \label{Y4}
& \mathscr{D}_{4f}(\boldsymbol{X},\boldsymbol{\lambda}^{\boldsymbol{m}} \boldsymbol{r}, \boldsymbol{\eta},\boldsymbol{n};\ddot{H}_{K,M}) \nonumber \\
&:= \delta_{m_1 \in [6,\ell_1+2]} \cdot  \delta_{m_2 \in [6,\ell_2+2]} \cdot
N(\lambda^8) \nonumber \\
& \times \sum_{\boldsymbol{\nu} \in (\lambda^{-3} \mathbb{Z}[\omega])^2} \rho_f(\nu_1) \overline{\rho_f(\nu_2)} 
\Big( \frac{\nu_1}{|\nu_1|} \Big)^{-n_1} \Big( \frac{\overline{\nu_2}}{|\nu_2|} \Big)^{-n_2}  \nonumber \\
& \times \mathop {\sum}_{\boldsymbol{\zeta}}  
\Big( \frac{\overline{\zeta_1^{-1} \lambda^{m_1-4} r_1} }{\zeta_1^{-1} \lambda^{m_1-4} r_1 }  \Big)^{-n_1} 
\Big( \frac{\overline{\zeta_2^{-1} \lambda^{m_2-4} r_2 }}{\zeta_2^{-1} \lambda^{m_2-4} r_2 }  \Big)^{n_2} \nonumber \\
& \times \psi^{\#}_{\lambda^{m_1} r_1}(\cdot)_{\lambda \eta_1,\zeta_1^{-1}}(\lambda^4 \nu_1) \cdot
\overline{\psi^{\#}_{\lambda^{m_2} r_2}(\cdot)_{\lambda \eta_2,\zeta_2^{-1}}(\lambda^4 \nu_2)} \nonumber \\
& \times \ddot{H}_{K,M} \Big(\frac{N(\nu_1)}{N(\lambda^{m_1-4} r_1)^2/X_1},\frac{N(\nu_2)}{N(\lambda^{m_2-4} r_2)^2/X_2},\boldsymbol{n}  \Big),
\end{align}
$\psi^{\#}_{\zeta}$ and $\psi^{\star}$ are given in 
\eqref{Psidefn} and \eqref{underlinedefn} (with $\psi^{\prime \prime} \rightarrow \psi$) respectively,
and $\ddot{H}_{K,M}(\cdot, \boldsymbol{n}):(0,\infty)^2 \rightarrow \mathbb{C}$ is given by \eqref{doubletilde},
and $\omega(m,j,r)$ and $\mathfrak{c}(m,j,r)$
are both as in Proposition \ref{funceqprop2}.
\end{prop}

\section{Type-I estimates} \label{type1sec}
Recall the notation from \S \ref{introduction},
in particular \eqref{type1point} and \eqref{type1avg}.

\begin{remark} \label{uniquefac}
We can uniquely factorise $v=\lambda^{e_v} \zeta_{v} v_0$ where $e_v \in \mathbb{Z}_{\geq 2}$, 
$\zeta_{v}^6=1$, and $v_0 \equiv 1 \pmod{3}$. 
In view of the congruence condition $ab \equiv u \pmod{v}$, we can 
assume without loss of generality that $v=\lambda^{e_v} v_0$ with $v_0 \equiv 1 \pmod{3}$.
In particular, since $(u,v)=1$ and $ab \equiv u \pmod{v}$ in \eqref{type1point} and \eqref{type1avg}, we have $(a,v)=1$.
\end{remark}

\begin{proof}[Proof of Lemma \ref{trivpoint}]
The support of $W_K$ allows us to assume $N(a)\leq 54X$, since otherwise the sum is empty.
We write \eqref{type1point} as
\begin{equation} \label{type1trivavg}
\mathscr{S}_f(a,X,v,u;W_K)=\sum_{\substack{\nu \in \lambda^{-3} \mathbb{Z}[\omega] 
\\ \lambda^3 \nu \equiv 0 \pmod{a} \\ \lambda^3 \nu \equiv u \pmod{\lambda^{e_v} v_0}}} \rho_{f} (\nu) W_K \Big( \frac{N(\nu)}{X} \Big).
\end{equation}
Since $(a,\lambda^{e_v}  v_0)=1$, we let
$\overline{a} \in \mathbb{Z}[\omega]$ be such that $a \overline{a} \equiv 1 \pmod{\lambda^{e_v}  v_0}$.
The congruence conditions placed on $\nu$ in \eqref{type1trivavg} are equivalent
to $\lambda^3 \nu \equiv u a \overline{a} \pmod{\lambda^{e_v} v_0 a}$
by the Chinese Remainder Theorem.

\subsection{Application of Voronoi summation}
Applying Voronoi summation (Proposition \ref{funceqvor})
we obtain
\begin{align} \label{Aprimedefn} 
& \mathscr{S}_f(a,X;v,u;W_K) \nonumber \\
&=\frac{X}{N(\lambda^{e_v+2} v_0 a)} 
\sum_{\substack{ k \pmod{\lambda^{14}} \\ k \equiv 1 \pmod{3} }} \sum_{\substack{m,r,t \\ 0 \leq m \leq e_v+2 \\ r \mid a, \hspace{0.05cm} t \mid v_0  \\ rt \equiv k \pmod{\lambda^{14}} }} \sum_{p=1}^{2} Z_{pf}(X,\lambda^m rt,\eta,0;\dot{W}_K),
\end{align}
where $Z_{pf}(\cdots)$ for $p=1,2$ are given in \eqref{Z1r} and \eqref{Z2r} respectively.
The weight functions involved are radial, see Remarks \ref{radial} and \ref{radial2}, so only $n=0$ occurs on the dual side of Voronoi summation.

\subsection{Evaluation and bounds for arithmetic exponential sums}
We now consider the arithmetic exponential sum $\psi^{\#}_{\lambda^m r t}(\cdot)_{\lambda \eta,\zeta^{-1}}(\lambda^4 \nu)$ for $\nu \in \lambda^{-3} \mathbb{Z}[\omega]$
that occurs in $Z_{2f}(\cdots)$.
Throughout this computation we will repeatedly use the facts $\eta \equiv u a \overline{a} \pmod{\lambda^{e_v} v_0 a}$,
$a \overline{a} \equiv 1 \pmod{\lambda^{e_v} v_0}$,
$0 \leq m \leq e_v+2$, $r \mid a$, and $t \mid v_0$, without further reference. 
The replacements of $\eta$ by $u$ below remain valid at $m=e_v+2$ because
the identity $\check e(z/\lambda)=1$ for $z\in\mathbb Z[\omega]$ makes the
second Kloosterman argument periodic modulo $\lambda^{m-2}t$, and the shift
in $\widehat\psi_{\lambda^m}$ periodic modulo $\lambda^{m-1}$ when $m\geq1$.
Using Lemma \ref{unpack1} we have 
\begin{equation} \label{tildeeval}
\psi^{\#}_{\lambda^{m} r t}(\cdot)_{\lambda \eta,\zeta^{-1}}(\lambda^4 \nu)
=\frac{1}{N(\lambda^{m+1} r t)} K_{\Gamma_1(3),\sigma,\sigma}(\zeta(\lambda^3 \nu), \zeta \eta,\lambda^{m-1} rt),
\end{equation}
where $\sigma=\pmatrix 1001$.
After opening the cubic Kloosterman sum in \eqref{tildeeval}, we then perform a computation using the Chinese Remainder Theorem
(with coprime moduli $\lambda^{m-1}t$ and $r$), \eqref{cuberep}, and \eqref{coprimerel}, to obtain
\begin{align} \label{simp1}
\psi^{\#}_{\lambda^{m} rt}(\cdot)_{\lambda \eta,\zeta^{-1}}(\lambda^4 \nu)&=\frac{1}{N(r)^{1/2} N(\lambda^{m+1} t)} \overline{\Big( \frac{\zeta^{-1} \lambda^{m-1} t}{r} \Big)_3} 
\overline{\widetilde{g}(\lambda^3 \nu,r)}  \nonumber \\
& \times K_{\Gamma_1(3),\sigma,\sigma}( \zeta \overline{r} (\lambda^3 \nu),\zeta \overline{r} u,\lambda^{m-1} t).
\end{align}
The bound
\begin{equation} \label{intermedbd1}
|\psi^{\#}_{\lambda^{m} rt}(\cdot)_{\lambda \eta,\zeta^{-1}}(\lambda^4 \nu) | \ll N(\lambda^{m} r t)^{-1/2+\varepsilon} \cdot N((\lambda^3 \nu,r))^{1/2}
\end{equation}
for $\nu \in \lambda^{-3} \mathbb{Z}[\omega]$ and $m \in \mathbb{Z}_{\geq 6}$
follows from using Lemma \ref{sqrootbd} and Lemma \ref{weilbd} in \eqref{simp1}, \eqref{distinctprime}, and the fact
\begin{equation*}
(\zeta \overline{r} (\lambda^3 \nu),\zeta \overline{r} u,\lambda^{m-1} t)=1. 
\end{equation*}

We now give a similar treatment of the
arithmetic sum $\psi^{\star}_{r t}(\cdot)_{\lambda^{2m+1} \eta}(\lambda^{2m+4} \nu )$
that occurs in $Z_{1f}(\cdots)$.
Using Lemma \ref{unpack2} we have 
\begin{equation} \label{simp2}
\psi^{\star}_{rt}(\cdot)_{\lambda^{2m+1} \eta}(\lambda^{2m+4} \nu)
=\frac{1}{N(rt)} K_{\Gamma_1(3),\sigma,\xi}(\overline{\lambda^{2m+3}} (\lambda^{2m+3} \nu), \overline{\lambda^3} \eta,rt),
\end{equation}
where $\sigma$ is as above, and $\xi:=\pmatrix 0 {-1} 1 0$.
After opening the cubic Kloosterman sum in \eqref{simp2}, we then perform a computation using the Chinese Remainder Theorem
(with coprime moduli $t$ and $r$) and \eqref{coprimerel}, to obtain 
\begin{align} \label{simp2b}
\psi^{\star}_{rt}(\cdot)_{\lambda^{2m+1} \eta}(\lambda^{2m+4} \nu)&=\frac{1}{N(r)^{1/2} N(t)} 
\overline{\Big( \frac{\lambda^{2m} t }{r} \Big)_3} \overline{\widetilde{g}(\lambda^{2m+3} \nu,r)}  \nonumber \\
& \times K_{\Gamma_1(3),\sigma,\xi}(\overline{\lambda^{2m+3}} \overline{r} (\lambda^{2m+3} \nu), \overline{\lambda^3} \overline{r} u,t).
\end{align}
The bound
\begin{equation} \label{intermedbd2}
|\psi^{\star}_{rt}(\cdot)_{\lambda^{2m+1} \eta}(\lambda^{2m+4} \nu)| \ll N(rt)^{-1/2+\varepsilon} N((\lambda^{2m+3} \nu,r))^{1/2},
\end{equation}
for $\nu \in \lambda^{-2m-3} \mathbb{Z}[\omega]$ follows from using Lemma \ref{sqrootbd} and Lemma \ref{weilbd} in \eqref{simp2}, \eqref{distinctprime},
and the fact $(\overline{\lambda^{2m+3}} \overline{r} (\lambda^{2m+3} \nu), \overline{\lambda^3} \overline{r} u,t)=1$.

\subsection{Truncations and conclusion}
We substitute \eqref{simp2b} and \eqref{simp1} into 
$Z_{pf}(\cdots)$ for $p=1,2$ respectively.
We recall Remark \ref{radial2}. We then use Lemma \ref{unpack3}
(with $D_1>0$ large and fixed and $D_2=0$), Lemma \ref{L1bound}, and
\eqref{intermedbd2} (resp. \eqref{intermedbd1}) to
truncate the $\nu$-sums in $Z_{pf}(\cdots)$ for $p=1$ (resp. $p=2$) to
\begin{equation} \label{Pdef}
N(\nu) \ll (XKN(v))^{\varepsilon} \cdot K^4 N(\lambda^m rt)^2 X^{-1}=:P,
\end{equation}
with negligible error $O((XKN(v))^{-2000})$. Denote the truncated expressions by
$Z^{\prime}_{pf}(\cdots,P)$ for $p=1,2$.
All nonzero dual frequencies in these two expressions have norm at least $3^{-13}$,
since $m\leq5$ in the first expression and the second uses $\lambda^{-3}\mathbb Z[\omega]$.
Thus the Rankin--Selberg and lattice-counting bounds used below also apply in the
fixed range of small positive norms, after increasing their implied constants.
Without loss of generality, we can restrict our attention to the case 
$P \gg (XKN(v))^{-\varepsilon}$ otherwise both
$Z^{\prime}_{pf}(\cdots,P)$ for $p=1,2$
are $O((XKN(v))^{-2000})$ by the above argument.
Thus 
\begin{align} \label{Aprimedefn2} 
& \mathscr{S}_f(a,X;v,u;W_K) \nonumber \\
&=\frac{X}{N(\lambda^{e_v+2} v_0 a)}
\sum_{\substack{k \pmod{\lambda^{14}} \\ k \equiv 1 \pmod{3} } } 
\sum_{\substack{m,r,t \\ 0 \leq m \leq e_v+2 \\ r \mid a,\hspace{0.05cm} t \mid v_0  \\  rt \equiv k \pmod{\lambda^{14}} \\ P\gg (XKN(v))^{-\varepsilon}  } } 
\sum_{p=1}^{2} Z^{\prime}_{pf}(\cdots,P)  \nonumber \\
&+O((XKN(v))^{-1000}).
\end{align}
Using the triangle inequality and \eqref{intermedbd1}, \eqref{intermedbd2},
and the $n=0$ case of Lemma \ref{unpack3} (with $D_1=\varepsilon$ and $D_2=0$) we obtain
\begin{align}
Z^{\prime}_{1f}(\cdots,P) & \ll (XK)^{\varepsilon} N(\lambda^{m})^{1+\varepsilon}
N(rt)^{-1/2+\varepsilon}  \nonumber \\
& \times \sum_{\substack{ \nu \in \lambda^{-2m-3} \mathbb{Z}[\omega] \\ N(\nu) \ll P}} 
|\rho_{f \otimes \widehat{\psi_{\lambda^m}(\cdot)_{\lambda u}},\mathfrak{c}(m,1;k)}(\nu) |
N((\lambda^{2m+3} \nu,r))^{1/2}; \label{Z1bound} \\
& \text{for} \quad 0 \leq m \leq \min \{5,e_v+2\}, \nonumber
\end{align}
and 
\begin{align}
Z^{\prime}_{2f}(\cdots,P) & \ll (XK)^{\varepsilon} N(\lambda^m rt)^{-1/2+\varepsilon} 
\sum_{\substack{ \nu \in \lambda^{-3} \mathbb{Z}[\omega] \\ N(\nu) \ll P}} | \rho_f(\nu) | N((\lambda^3 \nu,r))^{1/2}  \label{Z2bound} \\
& \text{for} \quad 6 \leq m \leq e_v+2. \nonumber
\end{align}

We apply the Cauchy-Schwarz inequality to the $\nu$-sums in \eqref{Z1bound} and \eqref{Z2bound}.
We then bound \eqref{Z1bound} and \eqref{Z2bound} using
Lemma \ref{gcdflat}, Lemma \ref{rankinselbound}, and \eqref{divis}.  
Substitution of the result into \eqref{Aprimedefn2} gives
\begin{align} \label{Aprimedefn3} 
& \mathscr{S}_f(a,X;v,u;W_K) \nonumber \\
& \ll \frac{X (XK N(v))^{\varepsilon}}{N(\lambda^{e_v+2} v_0 a)} 
\sum_{\substack{m,r,t \\ 0 \leq m \leq e_v+2 \\ r \mid a, \hspace{0.1cm} t \mid v_0  \\ P \gg (XKN(v))^{-\varepsilon}  }} \Big( N(\lambda^m rt)^{-1/2} \cdot \frac{K^4N(\lambda^{m} rt)^2}{X}  \Big)
+ (XKN(v))^{-1000} \nonumber \\
& \ll  (XKN(v))^{\varepsilon} K^{4} N(v)^{1/2} N(a)^{1/2}, 
\end{align}
as required. 
\end{proof}

\begin{proof}[Proof of Proposition \ref{type1avgbd}]
We may assume $A\ll X$, since otherwise \eqref{type1avg} is empty.
Fix $0<\delta<\min\{\varepsilon/100,1/8\}$ for the Mellin contours below.
We multiply \eqref{type1trivavg} by $\mu^2(a) \alpha_a$ 
and sum over $a \in \mathbb{Z}[\omega]$. We repeat the same steps on the $\nu$ sum as in the proof of Lemma \ref{trivpoint}
up to the display \eqref{Aprimedefn}. We then insert a smooth dyadic 
partition of unity in the $r$ variable, using $U$ supported in $[1/2,2]$ as in \S 7.
The ranges $N(r)\asymp R$ and $N(\nu)\asymp S$ below mean the supports
$R/2\leq N(r)\leq2R$ and $S/2\leq N(\nu)\leq2S$, respectively.
We obtain 
\begin{equation} \label{Afdecompose}
\mathscr{A}_f(\cdots )=\sum_{\substack{1 \ll R \ll A  \\ R \text{ dyadic}}} 
\mathscr{A}_f(\cdots,R),
\end{equation}
where
\begin{align} \label{Adefn3}
&\mathscr{A}_f(\cdots,R) \nonumber \\
&:=\frac{X}{N(\lambda^{e_v+2} v_0)} \sum_{a \in \mathbb{Z}[\omega]} \frac{\mu^2(a) \alpha_a}{N(a)} 
\sum_{\substack{k \pmod{\lambda^{14}} \\ k \equiv 1 \pmod{3} } } 
\mathop{\sum}_{\substack{m,r,t \\ 0 \leq m \leq e_v+2 \\ r \mid a,\hspace{0.05cm} t \mid v_0  \\  rt \equiv k \pmod{\lambda^{14}} } } U \Big( \frac{N(r)}{R} \Big) \nonumber \\
& \times \sum_{p=1}^2 Z_{pf}(X,\lambda^m rt,\eta,0;\dot{W}_K),
\end{align}
where the $Z_{pf}(\cdots)$ for $p=1,2$
are given by \eqref{Z1r} and \eqref{Z2r} respectively. We recall Remark \ref{radial2}. We then use Lemma
\ref{unpack3}
(with $D_1>0$ large and fixed and $D_2=0$),
Lemma \ref{L1bound}, and
\eqref{intermedbd2} (resp. \eqref{intermedbd1}) to
truncate the $\nu$-sums in $Z_{pf}(\cdots)$ for $p=1$ (resp. $p=2$)
with $N(r) \asymp R$ to 
\begin{equation} \label{P0def}
N(\nu) \ll (XKN(v))^{\varepsilon} \cdot K^4 R^2 N(\lambda^m t)^2 X^{-1}=:P_0,
\end{equation}
with negligible error $O((XKN(v))^{-2000})$. 
Denote the truncated expressions
by $Z^{\prime}_{pf}(\cdots,P_0)$ for $p=1,2$. 
Here the truncation retains whole dyadic sums with smooth weights, with a fixed dyadic
rounding of the endpoint $P_0$. The discarded sums have the stated negligible error.
This same endpoint is used in the dyadic $S$-sums below.
Without loss of generality, we can restrict our attention to the case that $P_0 \gg (XKN(v))^{-\varepsilon}$
otherwise both
$Z^{\prime}_{pf}(\cdots,P_0)$
are $O((XKN(v))^{-2000})$ by the above argument.
Thus \eqref{Adefn3} becomes
\begin{align} \label{Adefn4}
&\mathscr{A}_f(\cdots,R) \nonumber \\
&=\frac{X}{N(\lambda^{e_v+2} v_0)} \sum_{a \in \mathbb{Z}[\omega]} \frac{\mu^2(a) \alpha_a}{N(a)} 
\sum_{\substack{k \pmod{\lambda^{14}} \\ k \equiv 1 \pmod{3} } } 
\sum_{\substack{m,r,t \\ 0 \leq m \leq e_v+2 \\ r \mid a,\hspace{0.05cm} t \mid v_0 \\ rt \equiv k \pmod{\lambda^{14}} \\ P_0 \gg (XKN(v))^{-\varepsilon} } } U \Big( \frac{N(r)}{R} \Big)  
\sum_{p=1}^{2} Z^{\prime}_{pf}(\cdots,P_0) \nonumber \\
&+O((XKN(v))^{-1000} \| \boldsymbol{\mu^2} \boldsymbol{\alpha} \|_2).
\end{align}

\subsection{Further simplification using the squarefree support of $\boldsymbol{\alpha}$}
We further open each $Z^{\prime}_{pf}(\cdots,P_0)$
in \eqref{Adefn4} and manipulate them 
by further simplifying \eqref{simp1} and \eqref{simp2b}
under the assumption that $r \equiv 1 \pmod{3}$ is squarefree (as is the case in \eqref{Adefn4}).
For $r$ squarefree and $\mu \in \mathbb{Z}[\omega]$,
Lemma \ref{elemlem} guarantees that $g(\mu,r)=0$
unless $(\mu,r)=1$. 
When $(\mu,r)=1$ we note that \eqref{coprimerel} implies that
$\widetilde{g}(\mu,r)=\overline{\big(\tfrac{\mu }{r} \big)_3} \widetilde{g}(r)$.
Thus \eqref{simp1} becomes
\begin{align} \label{simp1a}
& \psi^{\#}_{\lambda^{m} rt}(\cdot)_{\lambda \eta,\zeta^{-1}}(\lambda^4 \nu) \nonumber \\
&=\begin{cases}
N(r)^{-1/2} & \hspace{-0.3cm} \overline{\big( \frac{\zeta^{-1} \lambda^{m-1} t}{r} \big)_3} 
\overline{\widetilde{g}(r)} \big(\frac{\lambda^3 \nu}{r}  \big)_3 
\cdot N(\lambda^{m+1} t)^{-1} K_{\Gamma_1(3),\sigma,\sigma}( \zeta \overline{r}(\lambda^3 \nu),\zeta \overline{r} u,\lambda^{m-1} t) \\
& \text{if} \quad (\lambda^3 \nu,r)=1, \\
0 & \text{otherwise}
\end{cases}
\end{align}
for all $m \in \mathbb{Z}_{\geq 6}$ and $\nu \in \lambda^{-3} \mathbb{Z}[\omega]$. Similarly,
\eqref{simp2b} becomes 
\begin{align} \label{simp2bc}
& \psi^{\star}_{rt}(\cdot)_{\lambda^{2m+1} \eta}(\lambda^{2m+4} \nu) \nonumber \\
&=\begin{cases} 
N(r)^{-1/2} & \hspace{-0.3cm}
\overline{\big( \frac{\lambda^{2m} t }{r} \big)_3} \overline{\widetilde{g}(r)} \big( \frac{\lambda^{2m+3} \nu}{r} \big)_3 
\cdot N(t)^{-1} K_{\Gamma_1(3),\sigma,\xi}(\overline{\lambda^{2m+3}} \overline{r} (\lambda^{2m+3} \nu), \overline{\lambda^3} \overline{r} u,t) \\
&  \text{if} \quad (\lambda^{2m+3} \nu,r)=1, \\
0 & \text{otherwise}
\end{cases},
\end{align}
for all $m \in \mathbb{Z}_{\geq 0}$ and $\nu \in \lambda^{-2m-3} \mathbb{Z}[\omega]$.

\subsection{Preparations for the cubic large sieve}
We substitute \eqref{simp2bc} and \eqref{simp1a} into the expressions for
$Z^{\prime}_{pf}(\cdots,P_0)$ for $p=1$ (resp. $p=2$) in \eqref{Adefn4}.
We then insert a smooth dyadic partition of unity in the $\nu$ variable and
open the transforms $\dot{W}_K(\cdot)$ with \eqref{Wtrans}. We then move the resulting contour integral to 
$\Re(s)=\delta$, 
resolve the $r,\lambda^3 \nu, \lambda^{2m+3} \nu$ variables into congruence classes modulo $\lambda^{\max \{4,m-1\}} t$,
and interchange the order of summation/integration by absolute convergence (see \eqref{Kmellinbound} and \eqref{quotbd}).
Let $S_*>0$ be a fixed dyadic number smaller than $3^{-13}/2$.
Since the forms are cuspidal, zero frequencies contribute nothing, and all nonzero
dual frequencies are covered by the dyadic ranges $S\geq S_*$.
Then \eqref{Adefn4} becomes
\begin{align}  \label{Asum}
\mathscr{A}_f(\cdots,R) &=\mathscr{A}^{\prime}_f(\cdots,R)+\mathscr{A}^{\prime \prime}_f(\cdots,R) 
+O((XKN(v))^{-1000} \| \boldsymbol{\mu^2} \boldsymbol{\alpha} \|_2),
\end{align}
where
\begin{align} \label{Afprime}
& \mathscr{A}^{\prime}_f(\cdots,R) \nonumber \\
&:=\frac{X}{N(\lambda^{e_v+2} v_0)}
\sum_{\substack{m,t \\ 0 \leq m \leq \min \{5,e_v+2\} \\  t \mid v_0  \\ P_0 \gg (XKN(v))^{-\varepsilon} }}    \frac{N(\lambda^m)}{N(t)} 
\sum_{\substack{ k \pmod{\lambda^{14}} \\ k \equiv 1 \pmod{3} }} \omega(m,1;k) 
\sum_{\substack{ \boldsymbol{j} \in (\mathbb{Z}[\omega]/9t \mathbb{Z}[\omega])^2 \\ j_1 \equiv 1 \pmod{3} \\ (j_1,t)=1 }}  
\overline{\Big( \frac{\lambda^{2m} t }{j_1} \Big)_3} \nonumber \\
& \times K_{\Gamma_1(3),\sigma,\xi}(\overline{\lambda^{2m+3}} \overline{j_1} j_2, \overline{\lambda^3} \overline{j_1} u,t) 
\cdot \frac{1}{2 \pi i} \int_{(\delta)} 
\frac{G_{\infty}(s,\tau_f,0)}{G_{\infty}(1-s,\tau_f,0)} \widehat{W}_K(1-s) X^{-s} N(t)^{2s} \nonumber \\
& \times \sum_{\substack{S_* \leq S \ll P_0 \\ S \text{ dyadic}}}
\Big(\sum_{\substack{ \nu \in \lambda^{-2m-3} \mathbb{Z}[\omega] \\  \lambda^{2m+3} \nu \equiv j_2 \pmod{9t} \\ N(\nu) \asymp S}} 
\sum_{\substack{r \in \mathbb{Z}[\omega] \\ r \equiv j_1 \pmod{9t} \\ rt \equiv k \pmod{\lambda^{14}} \\ N(r) \asymp R }}
\Omega^{\prime}_\nu(s,\lambda^m,k,S) \Psi_r(s,R) \Big( \frac{\lambda^{2m+3} \nu}{r} \Big)_3 \Big) ds;
\end{align}
\begin{align} \label{Afprimeprime}
& \mathscr{A}^{\prime \prime}_f(\cdots,R) \nonumber \\
&:=\frac{X}{N(\lambda^{e_v+2} v_0)}
\mathop{ \sum}_ {\substack{\zeta,m,t \\ 6 \leq m \leq e_v+2 \\ t \mid v_0 \\ P_0 \gg (XKN(v))^{-\varepsilon} }} \frac{N(\lambda^2)}{N(\lambda^{m-1} t)} 
\sum_{\substack{ \boldsymbol{j} \in (\mathbb{Z}[\omega]/\lambda^{m-1} t \mathbb{Z}[\omega])^2 \\ j_1 \equiv 1 \pmod{3} \\ (j_1,t)=1 }} 
\overline{\Big( \frac{\zeta^{-1} \lambda^{m-1} t }{j_1} \Big)_3} \nonumber \\
& \times K_{\Gamma_1(3),\sigma,\sigma}(\zeta \overline{j_1} j_2, \zeta \overline{j_1} u, \lambda^{m-1} t) 
\cdot \frac{1}{2 \pi i} \int_{(\delta)} \frac{G_{\infty}(s,\tau_f,0)}{G_{\infty}(1-s,\tau_f,0)} \widehat{W}_K(1-s) X^{-s} N(t)^{2s} N(\lambda^{m-4})^{2s}  \nonumber \\
& \times \sum_{\substack{S_* \leq S \ll P_0  \\ S \text{ dyadic}}} 
\Big(\sum_{\substack{ \nu \in \lambda^{-3} \mathbb{Z}[\omega]  \\ \lambda^{3} \nu \equiv j_2 \pmod{\lambda^{m-1} t} \\ N(\nu) \asymp S }} 
 \sum_{\substack{ r \in \mathbb{Z}[\omega] \\ r \equiv j_1 \pmod{\lambda^{m-1} t} \\ N(r) \asymp R }}
 \Omega^{\prime \prime}_\nu(s,S) \Psi_r(s,R) \Big( \frac{\lambda^{3} \nu}{r} \Big)_3 \Big) ds;
\end{align}
and, with $\Omega'_0=\Omega''_0=0$ and the following definitions for $\nu\ne0$,
\begin{align}
\Psi_r(s,R)&:=N(r)^{-1/2} N(r)^{2s} \overline{\widetilde{g}(r)} U \Big( \frac{N(r)}{R} \Big) \sum_{a \equiv 0 \pmod{r}} \frac{\mu^2(a) \alpha_a}{N(a)}; \label{rweights} \\
\Omega_{\nu}^{\prime}(s,\lambda^m,k,S)
&:= N(\nu)^{-s} U \Big(\frac{N(\nu)}{S}  \Big) \rho_{f \otimes \widehat{\psi_{\lambda^m}(\cdot)_{\lambda u}},\mathfrak{c}(m,1;k)}(\nu); \nonumber \\
\Omega_{\nu}^{\prime \prime}(s,S)&:= N(\nu)^{-s} U \Big(\frac{N(\nu)}{S}  \Big) \rho_{f}(\nu). \nonumber
\end{align}

Observe that the weights $\Psi_r(s,R)$ in \eqref{rweights} are supported on 
squarefree $r$ (see \eqref{sqrootcancel}). 

\subsection{Application of the cubic large sieve and conclusion}
Consider the bilinear form in $\nu$ and $r$ and  in the last display of \eqref{Afprime}.
Splitting the supports into a bounded number of standard dyadic annuli and
using Corollary \ref{cubiccor} (the cubic large sieve), we obtain
\begin{align} \label{sieveapply}
& \sum_{\substack{ \nu \in \lambda^{-2m-3} \mathbb{Z}[\omega] \\  \lambda^{2m+3} \nu \equiv j_2 \pmod{9t} \\ N(\nu) \asymp S}}  \sum_{\substack{r \in \mathbb{Z}[\omega] \\ r \equiv j_1 \pmod{9t} \\ rt \equiv k \pmod{\lambda^{14}} \\ N(r) \asymp R }}
\Omega^{\prime}_\nu(s,\lambda^m,k,S) \Psi_r(s,R) \Big( \frac{\lambda^{2m+3} \nu}{r} \Big)_3 \nonumber \\
& \ll (RS)^{\varepsilon} S^{1/6}(S^{1/2}+R^{1/2}) 
\Big( \sum_{\substack{\nu \in \lambda^{-2m-3} \mathbb{Z}[\omega]}} |\Omega^{\prime}_\nu(s,\lambda^m,k,S)|^2 \Big)^{1/2} 
\Big( \sum_{\substack{r \in \mathbb{Z}[\omega] \\ r \equiv 1 \pmod{3} }} \mu^2(r) | \Psi_r(s,R) |^2  \Big)^{1/2},
\end{align}
where we dropped some of the congruence conditions in the $L^2$-norms by positivity.
Lemma \ref{rankinselbound} gives
\begin{equation} \label{fac1}
\sum_{\substack{\nu \in \lambda^{-2m-3} \mathbb{Z}[\omega]}} |\Omega^{\prime}_\nu(s,\lambda^m,k,S)|^2 \ll S^{1+\varepsilon},
\end{equation}
for each $0 \leq m \leq \min \{5,e_v+2\}$ and $S\geq S_*$.
The bounded range $S_*\leq S\leq100$ is included by increasing the implied constant.
Only finitely many twists and cusps occur here.
Using \eqref{sqrootcancel} and \eqref{normalised} we compute
\begin{align} \label{gammabd}
& \sum_{\substack{r \in \mathbb{Z}[\omega] \\ r \equiv 1 \pmod{3} }} \mu^2(r) | \Psi_r(s,R) |^2  \nonumber \\
&= \sum_{\substack{ r \in \mathbb{Z}[\omega] \\ r \equiv 1 \pmod{3} }} \frac{\mu^2(r)}{N(r)^{1-4 \Re(s)}} \Big | U \Big( \frac{N(r)}{R} \Big) \Big |^2
\cdot
\Big | \sum_{\substack{a \in \mathbb{Z}[\omega] \\ a \equiv 0 \pmod{r}}} \frac{\mu^2(a) \alpha_a}{N(a)} \Big |^2 \nonumber  \\
& \ll R^{-1+\varepsilon} \sum_{\substack{r \in \mathbb{Z}[\omega] \\ r \equiv 1 \pmod{3} \\ N(r) \asymp R }} 
\Big | \sum_{\substack{a \in \mathbb{Z}[\omega] \\ a \equiv 0 \pmod{r}}} \frac{\mu^2(a) \alpha_a}{N(a)} \Big |^2 \nonumber \\
& \ll A^{-1} R^{-2+\varepsilon}   \sum_{\substack{u,r \in \mathbb{Z}[\omega] \\ u,r \equiv 1 \pmod{3}}} \mu^2(ur) |\alpha_{ur}|^2 \nonumber \\
& \ll (AR)^{\varepsilon} A^{-1} R^{-2} \| \boldsymbol{\mu}^2 \boldsymbol{\alpha} \|_2^{2}.
\end{align}
The penultimate display follows from the Cauchy-Schwarz inequality and a change of variables.
The last display follows from \eqref{divis}. We substitute \eqref{fac1} and \eqref{gammabd} 
into \eqref{sieveapply}, and sum both sides of the resulting inequality 
over dyadic values of $S$ (for each $0 \leq m \leq \min \{5,e_v+2\}$).
We obtain
\begin{align} \label{bilinearbd}
& \sum_{\substack{S_* \leq S \ll P_0 \\ S \text{ dyadic}}}
\Big |  \sum_{\substack{ \nu \in \lambda^{-2m-3} \mathbb{Z}[\omega] \\  \lambda^{2m+3} \nu \equiv j_2 \pmod{9t} \\ N(\nu) \asymp S}}  \sum_{\substack{r \in \mathbb{Z}[\omega] \\ r \equiv j_1 \pmod{9t} \\ rt \equiv k \pmod{\lambda^{14}} \\ N(r) \asymp R }} \Omega^{\prime}_\nu(s,\lambda^m,k,S) \Psi_r(s,R) \Big( \frac{\lambda^{2m+3} \nu}{r} \Big)_3 \Big |  \nonumber \\
& \ll (XKN(v))^{\varepsilon}
(K^{14/3} N(t)^{7/3} R^{4/3} A^{-1/2} X^{-7/6}  + K^{8/3} N(t)^{4/3} R^{5/6} A^{-1/2} X^{-2/3} )  \| \boldsymbol{\mu}^2 \boldsymbol{\alpha} \|_2,
\end{align}
where \eqref{P0def} was used to obtain the last display.
We now sum over the residue classes in \eqref{Afprime} and \eqref{Afprimeprime}.
Put $V=N(v)$ and, for each $(m,t)$, put
\begin{equation*}
Q=\begin{cases}
N(t),&0\leq m\leq5,\\
N(\lambda^{m-1}t),&6\leq m\leq e_v+2.
\end{cases}
\end{equation*}
In both cases $Q\ll V$, and the modulus used to separate the residue classes has
norm $O(Q)$. There are therefore $O(Q^2)$ pairs $(j_1,j_2)$.
The second Kloosterman argument is coprime to its modulus, so Lemma \ref{weilbd}
bounds each Kloosterman sum by $O_\varepsilon(Q^{1/2+\varepsilon})$, also for
prime-power moduli. The prefactor outside these sums is $O(X/(VQ))$.
The factors $N(\lambda^m)$ for $m\leq5$ and $N(\lambda^2)$ are absolute constants.

The Mellin integrals contribute no additional power of $K$. Indeed, on
$\Re(s)=\delta$ the transform $\widehat W_K(1-s)$ is bounded, and the gamma quotient
is $O_f((1+|\Im s|)^{4\delta-2})$, which is integrable.
The factors $N(t)^{2s}$ and $N(\lambda^{m-4})^{2s}$ contribute only to the
$(XKV)^\varepsilon$ allowance.
For each pair of residue classes, \eqref{sieveapply}--\eqref{gammabd} give
\begin{equation*}
\ll (XKV)^\varepsilon A^{-1/2}R^{-1}
\big(P_0^{7/6}+R^{1/2}P_0^{2/3}\big)
\|\boldsymbol\mu^2\boldsymbol\alpha\|_2.
\end{equation*}
For $m\geq6$ the same estimate follows by applying Corollary \ref{cubiccor}
to $\lambda^3\nu$ and $r$, and Lemma \ref{rankinselbound} to $\rho_f$.
Since \eqref{P0def} gives
$P_0\ll (XKV)^\varepsilon K^4R^2Q^2/X$, the contribution for a fixed $(m,t)$ is
\begin{equation*}
\ll (XKV)^\varepsilon \frac{X}{V}Q^{3/2}A^{-1/2}R^{-1}
\left(\left(\frac{K^4R^2Q^2}{X}\right)^{7/6}
+R^{1/2}\left(\frac{K^4R^2Q^2}{X}\right)^{2/3}\right)
\|\boldsymbol\mu^2\boldsymbol\alpha\|_2.
\end{equation*}
Summing over $t\mid v_0$ and $m\leq e_v+2$, using the divisor bound and $Q\ll V$,
and absorbing the fixed unit and $k$-sums, we obtain
\begin{align} \label{ar1}
& |\mathscr{A}^{\prime}_f(\cdots,R)|+|\mathscr{A}^{\prime\prime}_f(\cdots,R)| \nonumber \\
& \ll (XKN(v))^{\varepsilon}  
(K^{14/3} N(v)^{17/6} R^{4/3} A^{-1/2} X^{-1/6}  +K^{8/3} N(v)^{11/6} R^{5/6} A^{-1/2} X^{1/3}  )  \| \boldsymbol{\mu}^2 \boldsymbol{\alpha} \|_2.
 \end{align}
After substituting these bounds into \eqref{Asum}, we then substitute 
the result into \eqref{Afdecompose} to obtain 
\begin{align} \label{combinedisplay}
\mathscr{A}_f(\cdots) 
& \ll (XKN(v))^{\varepsilon}( K^{14/3} N(v)^{17/6} X^{-1/6} A^{5/6}  +K^{8/3} N(v)^{11/6} (AX)^{1/3} )
\| \boldsymbol{\mu}^2 \boldsymbol{\alpha} \|_2 \nonumber \\
& \ll (XKN(v))^{\varepsilon} K^{14/3} N(v)^{17/6} (AX)^{1/3} \| \boldsymbol{\mu}^2 \boldsymbol{\alpha} \|_2,
\end{align}
where the last inequality follows since $A \ll X$.
The result follows. 
\end{proof}

\section{Type-II estimates via average (homogenous) convolution } \label{type2convsec}
Recall the notation from \S \ref{introduction},
in particular \eqref{type2}.
The first result in this section bounds the Type-II sum in terms 
of a homogeneous average convolution problem.

\begin{lemma} \label{convolutionsums}
Let the notation be as above and $X \asymp AB$. Then 
\begin{equation*}
|\mathscr{B}_f(\boldsymbol{\alpha},\boldsymbol{\beta},X,v,u;W_K)| \leq \| \boldsymbol{\beta} \|_2 \cdot \Big(  \sum_{\boldsymbol{a}}
\mu^2(a_1) \alpha_{a_1} \mu^2(a_2)
\overline{\alpha_{a_2}}
\mathscr{L}_{f}(\boldsymbol{a},X,v,u;W_K) \Big)^{1/2},
\end{equation*}
where 
\begin{align} \label{Lfdef}
& \mathscr{L}_{f}(\boldsymbol{a},X,v,u;W_K) \nonumber \\
&:=\sum_{\substack{b \in \mathbb{Z}[\omega] \\ a_1 b \equiv u \pmod{v} \\ a_2 b \equiv u \pmod{v} }}
\rho_f(\lambda^{-3} a_1 b) \overline{\rho_f(\lambda^{-3} a_2 b)} W_K \Big(\frac{N(\lambda^{-3} a_1 b)}{X} \Big) 
\overline{W_K \Big(\frac{N(\lambda^{-3} a_2 b)}{X} \Big)}.
\end{align}
\end{lemma}
\begin{proof}
We apply the Cauchy-Schwarz inequality to the $b$-sum in \eqref{type2} to obtain
\begin{equation*}
 |\mathscr{B}_f(\cdots)| 
\leq \|\boldsymbol{\beta} \|_2 \cdot \Big( \sum_{b \in \mathbb{Z}[\omega]} \Big | \sum_{\substack{ a \in \mathbb{Z}[\omega] \\ ab \equiv u \pmod{v} }}
\mu^2(a) \alpha_{a} \rho_f(\lambda^{-3}ab) W_K \Big( \frac{N(\lambda^{-3} ab)}{X} \Big) \Big|^2 \Big)^{1/2}.
\end{equation*}
The result follows from expanding the square modulus in the above expression
and interchanging the order of summation.
\end{proof}

\begin{prop} \label{conv1}
Let the notation be as above and $X \asymp AB$. Then
\begin{align*}
& \sum_{\boldsymbol{a}} \mu^2(a_1) \alpha_{a_1} \mu^2(a_2)
\overline{\alpha_{a_2}} \mathscr{L}_{f}(\boldsymbol{a},X,v,u;W_K) \nonumber \\
& \ll_{\varepsilon,f} (XKN(v))^{\varepsilon} K^{16} N(v)^{8} (AB+ A^{3} B^{1/2}) \|\boldsymbol{\mu^2} \boldsymbol{\alpha}\|^2_{\infty}.
\end{align*}
\end{prop}

\begin{remark}
It will be helpful to remember the normalisation in \eqref{normram} throughout the proof.
We also use the same notation and convention as Remark \ref{uniquefac}.
\end{remark}

\begin{proof}

We begin by separating oscillations using the circle method.

\subsection{Application of the circle method} \label{circlemethodapp}

Re-writing \eqref{Lfdef} we obtain
\begin{align} \label{deltappl}
&   \mathscr{L}_{f}(\boldsymbol{a},\cdots) \nonumber \\
&=\sum_{\substack{ \boldsymbol{\nu} \in (\lambda^{-3} \mathbb{Z}[\omega])^2 \\ \forall i: \lambda^3 \nu_i \equiv 0 \pmod{a_i}  \\ \forall i: \lambda^3 \nu_i \equiv u \pmod{\lambda^{e_v} v_0} }} \rho_f(\nu_1) \overline{\rho_f(\nu_2)} W_K \Big(\frac{N(\nu_1)}{X} \Big)  \overline{W_K \Big(\frac{N(\nu_2)}{X} \Big)} \delta_{\mathbb{Q}(\omega)} \Big(\frac{\lambda^3 \nu_2}{a_2}-\frac{\lambda^3 \nu_1}{a_1} \Big).
\end{align}
After noting Remark \ref{oscillation} we choose $C>0$ such that
\begin{equation} \label{Cchoice}
 C^4:=X/A \asymp B.
\end{equation}
To apply Theorem \ref{DFIcirc} and Remark \ref{elementindicator},
we first specify the additive characters.
For an integral $k$, the primitive characters modulo $d$ are
$k\mapsto\check e(uk/(\lambda d))$, with $u$ ranging over units modulo $d$.
Their normalized sum is $\widehat\psi_d(k/\lambda)$. In particular,
all valuations $\ell\geq0$ occur in the ideal sum. Thus
\begin{align} \label{dfidelta}
 \delta_{\mathbb{Q}(\omega)} \Big( \frac{\lambda^3 \nu_2}{a_2}- \frac{\lambda^3 \nu_1}{a_1} \Big) 
&=\frac{k_C}{C^{4}} \sum_{0 \leq \ell \ll 1+\log C} \sum_{\substack{c \in \mathbb{Z}[\omega] \\ c \equiv 1 \pmod{3}}} N(\lambda^{\ell} c) 
\widehat{\psi}_{\lambda^{\ell} c} \Big( \frac{1}{\lambda}\Big(\frac{\lambda^3 \nu_2}{a_2}-\frac{\lambda^3 \nu_1}{a_1}\Big) \Big) \nonumber \\
& \times h \Big( \frac{N(\lambda^{\ell} c)}{C^2}, \frac{N( \lambda^3 \nu_2/a_2- \lambda^3 \nu_1/a_1 )}{C^4} \Big),
\end{align}
for any $\boldsymbol{\nu} \in (\lambda^{-3} \mathbb{Z}[\omega])^2$ such that $\lambda^3 \nu_i \equiv 0 \pmod{a_i}$
for $i=1,2$, and where $\psi_{\lambda^{\ell} c}$ denotes the principal character modulo $\lambda^{\ell} c$. 
Let $\ell_0:=\max \{\ell,e_v\}$.  We
substitute \eqref{dfidelta} into \eqref{deltappl}, interchange the order of summation,
and resolve $\lambda^3 \nu_i$ 
into congruence classes $\pmod{\lambda^{\ell_0} a_i v_0 c}$ for $i=1,2$.
We obtain 
\begin{align} \label{doublesum}
&  \mathscr{L}_{f}(\boldsymbol{a},\cdots) \nonumber \\
&= \frac{k_C}{C^4} \sum_{0 \leq \ell \ll 1+\log C} \sum_{\substack{c \in \mathbb{Z}[\omega] \\ c \equiv 1 \pmod{3}}} 
\sum_{\substack{\boldsymbol{j} \in \prod_{i=1}^2 \mathbb{Z}[\omega]/ \lambda^{\ell_0} a_i v_0 c \mathbb{Z}[\omega] \\ \forall i: j_i\equiv0\pmod{a_i} \\ \forall i:  j_i \equiv u \pmod{\lambda^{e_v} v_0}  }} N(\lambda^{\ell} c) \widehat{\psi}_{\lambda^{\ell} c} \Big( \frac{1}{\lambda}\Big(\frac{j_2}{a_2}-\frac{j_1}{a_1}\Big) \Big) \nonumber \\
& \times \sum_{\substack{\boldsymbol{\nu} \in (\lambda^{-3} \mathbb{Z}[\omega])^2 \\ \forall i: \lambda^3 \nu_i \equiv j_i \pmod{ \lambda^{\ell_0} a_i v_0 c}
 }}
\rho_{f}(\nu_1) \overline{\rho_{f}(\nu_2)} 
H_{K,C^2/N(\lambda^{\ell} c)}\Big(\frac{\nu_1}{\sqrt{X}},\frac{\nu_2}{\sqrt{X}} \Big),
\end{align}
where $H_{K,C^2/N(\lambda^{\ell}c)}(\boldsymbol{z}):=H_{K,C^2/N(\lambda^{\ell}c)}(\boldsymbol{z};\boldsymbol{a},\lambda^{\ell} c,X,C)$ is given by
\begin{equation} \label{HKdef}
H_{K,C^2/N(\lambda^{\ell}c)}(\boldsymbol{z})
=W_K(|z_1|^2)\overline{W_K(|z_2|^2)}
h\Big(\frac{N(\lambda^{\ell}c)}{C^2},
\frac{X|\lambda^3z_1/a_1-\lambda^3z_2/a_2|^2}{C^4}\Big).
\end{equation}
Recall that $C^4=X/A\asymp B$, $N(a_i)\asymp A$, and $|z_i|\asymp1$ on the support. Thus
\begin{equation} \label{sizeob}
X|\lambda^3z_1/a_1-\lambda^3z_2/a_2|^2/C^4\ll1.
\end{equation}
Together with \eqref{hsupport}, this implies
\begin{equation} \label{hnonvanish}
H_{K,C^2/N(\lambda^\ell c)}\neq0
\quad\text{only if}\quad N(\lambda^\ell c)\ll C^2.
\end{equation}

We require an estimate that retains the concentration of this weight.
Write $x=N(\lambda^\ell c)/C^2$, $H_x=H_{K,1/x}$, and
\[
\xi_i=\frac{\sqrt X\lambda^3}{C^2a_i},\qquad
Y=|\xi_1z_1-\xi_2z_2|^2,\qquad \mathcal U_x=K+x^{-1/2}.
\]
Here $|\xi_i|\asymp1$ and $0<x\ll1$. The chain rule and Lemma \ref{hlemma} give
\begin{equation} \label{HKderivbound}
\partial^{\boldsymbol i}H_x(\boldsymbol z)
\ll_{\boldsymbol i}x^{-1}\mathcal U_x^{|\boldsymbol i|}.
\end{equation}
The factor $x^{-1}$ disappears after integration in the two real directions
transverse to $\xi_1z_1=\xi_2z_2$. More precisely, set
$E=r_1\partial_{r_1}+r_2\partial_{r_2}$ in polar coordinates
$z_i=r_i e^{i\theta_i}$, and let $D^{\boldsymbol\alpha}$ be any product of
the individual logarithmic radial and angular derivatives. Then
\begin{equation} \label{HKmixed}
\|E^jD^{\boldsymbol\alpha}H_x\|_{L^1(d_\times\boldsymbol z)}
\ll_{j,\boldsymbol\alpha}K^j\mathcal U_x^{|\boldsymbol\alpha|}.
\end{equation}
To see this for $x\le1$, use the coordinates
$(u,z_2)=(\xi_1z_1-\xi_2z_2,z_2)$, whose Jacobian is uniformly bounded
above and below. Since $EY=2Y$, applying $E$ to the kernel replaces $h(x,Y)$
by a linear combination of $Y^q\partial_Y^qh(x,Y)$. Lemma \ref{hlemma}
shows that these functions have size $O(x^{-1})$ on $Y\ll x$ and
arbitrary power decay in $Y/x$ outside this region, apart from the
arbitrarily small remainder for $q=0$. Each ordinary derivative in $u$
costs $x^{-1/2}$ in $L^1$, by the substitution $u=x^{1/2}w$.
Derivatives of the cutoffs cost powers of $K$. Finally, $E$ commutes with
the individual logarithmic radial and angular derivatives, which proves
the mixed estimate \eqref{HKmixed} by Leibniz's rule.
For $x$ in a fixed compact subinterval of $[1,\infty)$ the assertion
follows directly from Lemma \ref{hlemma}.

Simultaneous rotation of $z_1,z_2$ leaves $H_x$ invariant, so
$\widehat{\widehat H}_x(\boldsymbol s,(n_1,n_2))=0$ unless $n_1+n_2=0$.
Integration by parts in \eqref{HKmixed} gives, on each fixed pair of real
Mellin contours and for any $B,M\ge0$,
\begin{align} \label{HKjointmellin}
&|\widehat{\widehat H}_x(\boldsymbol\sigma+i\boldsymbol t,(n,-n))| \nonumber\\
&\quad\ll_{B,M,\boldsymbol\sigma}
\left(1+\frac{|t_1+t_2|}{K}\right)^{-B}
\left(1+\frac{|t_1|+|t_2|+|n|}{\mathcal U_x}\right)^{-M}.
\end{align}
The angular convention in \eqref{doubletilde} therefore gives
$\ddot H_x(\boldsymbol Y,(n_1,n_2))=0$ unless $n_1=n_2$.
We claim that, for every fixed $D>0$,
\begin{align} \label{HKtransform}
\sum_{n\in\mathbb Z}|\ddot H_x(\boldsymbol Y,(n,n))|
&\ll_{D,f}K\log(2+\mathcal U_x)(Y_1Y_2)^{-1/4} \nonumber\\
&\quad\times\prod_{i=1}^2
\min\left\{1,\left(\frac{\mathcal U_x^4}{Y_i}\right)^D\right\}.
\end{align}
Indeed, on $\Re(s)=1/4+D$ Stirling's formula gives
\[
\left|\frac{G_\infty(s,\tau_f,-n)}{G_\infty(1-s,\tau_f,n)}\right|
\ll_{D,f}(1+|\Im(s)|+|n|)^{4D-1}.
\]
There are no poles of this quotient in $\Re(s)>0$. First take both
contours in \eqref{doubletilde} at $1/4$. For $S=t_1+t_2$ and
$A_n=1+|n|$, Cauchy--Schwarz gives
\[
\int_{\mathbb R}\frac{dt}{(A_n+|t|)(A_n+|S-t|)}
\le \frac{2}{A_n}.
\]
Use \eqref{HKjointmellin}, integrate $S$ on scale $K$, and sum
$A_n^{-1}(1+|n|/\mathcal U_x)^{-M}$ over $n$ to obtain
$O(K\log(2+\mathcal U_x))$. For each $i$ with $Y_i>\mathcal U_x^4$,
move the $s_i$ contour to $1/4+D$. The additional gamma factor costs
$\mathcal U_x^{4D}$, with its remaining polynomial growth absorbed by
increasing $M$ in \eqref{HKjointmellin}. The four choices of contours
prove \eqref{HKtransform}.

\subsection{Double application of Voronoi summation} \label{doublevorapp}
We apply (double)-Voronoi summation (Proposition \ref{vor2}).
By abuse of notation we denote $(X,X)$ by $X$.
We obtain
\begin{align} \label{nusum}
& \sum_{\substack{\boldsymbol{\nu} \in (\lambda^{-3} \mathbb{Z}[\omega])^2 \\ \forall i: \lambda^3 \nu_i \equiv j_i \pmod{ \lambda^{\ell_0} a_i v_0 c}
 }} 
\rho_{f}(\nu_1) \overline{\rho_{f}(\nu_2)} H_{K,C^2/N(\lambda^{\ell}c)} \Big(\frac{\nu_1}{\sqrt{X}},\frac{\nu_2}{\sqrt{X}} \Big)  \nonumber \\
&= \frac{X^2}{N(\lambda^{\ell_0+2} v_0 c)^2 N(a_1 a_2 ) } 
\sum_{\substack{ \boldsymbol{k} \in (\mathbb{Z}[\omega]/\lambda^{14} \mathbb{Z}[\omega])^2 \\ \forall i: k_i \equiv 1 \pmod{3} }}
\sum_{\substack{\boldsymbol{m},\boldsymbol{r} \\ \forall i: 0 \leq m_i \leq \ell_0+2 \\ \forall i: r_i \mid a_i v_0 c \\  \forall i: r_i \equiv k_i \pmod{\lambda^{14}}   }} \nonumber \\
& \sum_{\boldsymbol{n} \in \mathbb{Z}^2} (-1)^{n_1+n_2}
\sum_{p=1}^{4} \mathscr{D}_{pf}(X,\boldsymbol{\lambda}^{\boldsymbol{m}} \boldsymbol{r},\boldsymbol{j}, \boldsymbol{n} ;\ddot{H}_{K,C^2/N(\lambda^{\ell}c)}), 
\end{align}
where the $\mathscr{D}_{pf}(\cdots)$ are given by \eqref{Y1}--\eqref{Y4}.
We substitute \eqref{nusum} into \eqref{doublesum} to obtain 
\begin{equation} \label{Lsum}
\mathscr{L}_{f}(\boldsymbol{a},\cdots)=\sum_{p=1}^4 \mathscr{M}_{pf}(\boldsymbol{a},\cdots),
\end{equation}
where
\begin{align} \label{resultingexp}
&\mathscr{M}_{pf}(\boldsymbol{a},\cdots) \nonumber \\
&:= \frac{k_C}{C^4} \frac{X^2}{N(a_1 a_2) N(v_0)^2} \sum_{0 \leq \ell \ll 1+\log C} 
\sum_{\substack{c \in \mathbb{Z}[\omega] \\ c \equiv 1 \pmod{3} }} \frac{N(\lambda^{\ell})}{N(\lambda^{\ell_0+2})^2} \frac{1}{N(c)} \nonumber \\
& \times \sum_{\substack{\boldsymbol{j} \in \prod_{i=1}^2 \mathbb{Z}[\omega]/ \lambda^{\ell_0} a_i v_0 c \mathbb{Z}[\omega]  \\ \forall i:  j_i \equiv 0 \pmod{a_i} \\ \forall i: j_i \equiv u \pmod{\lambda^{e_v} v_0}  }}
\widehat{\psi}_{\lambda^{\ell} c} \Big( \frac{1}{\lambda}\Big(\frac{j_2}{a_2}-\frac{j_1}{a_1}\Big) \Big) 
\sum_{\substack{ \boldsymbol{k} \in (\mathbb{Z}[\omega]/ \lambda^{14} \mathbb{Z}[\omega])^2 \\ \forall i: k_i \equiv 1 \pmod{3} }}   
\sum_{\substack{\boldsymbol{m},\boldsymbol{r} \\ \forall i: 0 \leq m_i \leq \ell_0+2 \\  \forall i: r_i \mid a_i v_0 c   \\ \forall i: r_i \equiv k_i \pmod{\lambda^{14}} }} \nonumber \\
& \sum_{\boldsymbol{n} \in \mathbb{Z}^2} (-1)^{n_1+n_2} \mathscr{D}_{pf}(X,\boldsymbol{\lambda}^{\boldsymbol{m}} \boldsymbol{r},\boldsymbol{j},\boldsymbol{n} ;\ddot{H}_{K,C^2/N(\lambda^{\ell}c)}).
\end{align}

We now make a sequence of manipulations to $\mathscr{M}_{pf}(\boldsymbol{a},\cdots)$ in
\eqref{resultingexp}.
First we make a change of variable $j_i \rightarrow a_i j_i$ for $i=1,2$ (the new $j_i$ variables run
$\pmod{\lambda^{\ell_0} v_0 c}$). We then uniquely factorise
each $c \in \mathbb{Z}[\omega]$ with $c \equiv 1 \pmod{3}$ as
$c=t q^{\prime} q^{\prime \prime}$
where  $t, q^{\prime}, q^{\prime \prime} \in \mathbb{Z}[\omega]$ satisfy 
\begin{equation} \label{tqqcond}
t,q^{\prime},q^{\prime \prime} \equiv 1 \pmod{3}, \quad
t \mid \text{rad}(v_0)^{\infty}, \quad q^{\prime} \mid  \text{rad}(a_1 a_2)^{\infty}, 
\text{ and } (q^{\prime \prime},a_1 a_2 v_0)=1.
\end{equation} 
Note this factorisation exists and is unique since $(v_0,a_1 a_2)=1$.
We also uniquely factorise each $r_i \mid a_i v_0 c$ with $r_i \equiv k_i \pmod{\lambda^{14}}$
 as $r_i=t_i r_i^{\prime} r_i^{\prime \prime}$
where $t_i, r_i^{\prime}, r_i^{\prime \prime}$ satisfy
\begin{align} \label{trrcond}
&t_i,r^{\prime}_i, r_i^{\prime \prime} \equiv 1 \pmod{3}, \quad
t_i \mid v_0 t, \quad r_i^{\prime} \mid a_i q^{\prime},
\quad r_i^{\prime \prime} \mid q^{\prime \prime}, \nonumber\\
&t_i r^{\prime}_i r^{\prime \prime}_i \equiv k_i\pmod{\lambda^{14}}
\qquad(i=1,2).
\end{align}

We use the Chinese Remainder Theorem on the new $j_i$ variables (with the pairwise coprime moduli 
$\lambda^{\ell_0} v_0 t$, $q^{\prime}$, and $q^{\prime \prime}$) \text{and} $i=1,2$ to write
\begin{align} \label{jidefn}
 \boldsymbol{j}&:=q^{\prime} q^{\prime \prime} \overline{q^{\prime} q^{\prime \prime}} \boldsymbol{J}+ \lambda^{\ell_0} v_0 t q^{\prime \prime} \overline{\lambda^{\ell_0} v_0
t  q^{\prime \prime}} \boldsymbol{J^{\prime}}+\lambda^{\ell_0} v_0 t q^{\prime} \overline{\lambda^{\ell_0}  v_0 t q^{\prime}} \boldsymbol{J^{\prime \prime}}, \\
& J_i \equiv \overline{a_i} u \pmod{\lambda^{e_v} v_0} \quad \text{for} \quad i=1,2, \nonumber
 \end{align}
where $\overline{a_i}, \overline{q^{\prime} q^{\prime \prime}},\overline{\lambda^{\ell_0} v_0 t q^{\prime \prime}},\overline{\lambda^{\ell_0} v_0 t q^{\prime}}
 \in \mathbb{Z}[\omega]$
are such that $a_i \overline{a_i} \equiv 1 \pmod{\lambda^{e_v} v_0}$, $\overline{q^{\prime} q^{\prime \prime}} q^{\prime} q^{\prime \prime} \equiv 1 \pmod{\lambda^{\ell_0} v_0 t}$,
$\overline{\lambda^{\ell_0} v_0 t q^{\prime}} \lambda^{\ell_0} v_0 t q^{\prime} \equiv 1 \pmod{q^{\prime \prime}}$, and
$\overline{\lambda^{\ell_0} v_0 t q^{\prime \prime}} \lambda^{\ell_0} v_0 t q^{\prime \prime} \equiv 1 \pmod{q^{\prime}}$.
We may assume $e_v\geq14$. Otherwise split the progression into the finitely many progressions modulo $\lambda^{14}v_0$. This changes all asserted estimates by at most an absolute constant.
We further make the change of variable
\begin{equation} \label{subsvariablechange}
\boldsymbol{J} \rightarrow \lambda^{e_v} v_0 \boldsymbol{J}+ (Y_1 u,Y_2 u)
\end{equation}
in \eqref{jidefn}, where $Y_i \in \mathbb{Z}[\omega]$ is such that $Y_i \equiv \overline{a_i} \pmod{\lambda^{e_v} v_0}$.
Observe that the new $J_1,J_2$ variables run $\pmod{\lambda^{\ell_0-e_v} t}$.
We also use the multiplicativity of primitive Ramanujan sums. For integral $k$,
\[
\widehat\psi_{\lambda^\ell c}(k/\lambda)
=\widehat\psi_{\lambda^\ell t}(k/\lambda)
 \widehat\psi_{q'}(k)\widehat\psi_{q''}(k),
\]
since multiplication by $\lambda$ permutes units modulo $q'q''$.
We interchange the order of summation by absolute convergence.
The net result is 
\begin{align} \label{Midef}
& \mathscr{M}_{pf}(\boldsymbol{a},\cdots) \nonumber \\
&:=\frac{k_C}{C^4} \frac{X^2}{N(a_1 a_2) N(v_0)^2} \sum_{0 \leq \ell \ll 1+\log C}
\sum_{\substack{t q^{\prime} q^{\prime \prime} \in \mathbb{Z}[\omega] \\ \eqref{tqqcond}  }}   
\frac{N(\lambda^{\ell})}{N(\lambda^{\ell_0+2})^2} 
\frac{1}{N(t q^{\prime} q^{\prime \prime} )}  \nonumber \\
& \times \sum_{\substack{ \boldsymbol{k} \in (\mathbb{Z}[\omega]/\lambda^{14} \mathbb{Z}[\omega])^2 \\ \forall i: k_i \equiv 1 \pmod{3} }} 
\sum_{\substack{\boldsymbol{m},\boldsymbol{t},\boldsymbol{r}^{\prime},\boldsymbol{r}^{\prime \prime} \\ \forall i:  0 \leq  m_i \leq \ell_0+2 \\ \eqref{trrcond} }}
 \sum_{\boldsymbol{n} \in \mathbb{Z}^2} (-1)^{n_1+n_2} \nonumber \\
& \times S_{pf}(\boldsymbol{a},\lambda^{\ell} t q^{\prime} q^{\prime \prime}, \boldsymbol{\lambda}^{\boldsymbol{m}} \boldsymbol{t} \boldsymbol{r^{\prime}} \boldsymbol{ r^{\prime \prime}},\boldsymbol{n};
\ddot{H}_{K,C^2/N(\lambda^{\ell}t q^{\prime} q^{\prime \prime})});
\end{align}
where
\begin{align} \label{S1defn}
& S_{1f}(\cdots) \nonumber \\
&:=\delta_{\boldsymbol{m} \in [0, \min \{5,\ell_0+2\}]^2 } 
N(\lambda^{m_1}) N(\lambda^{m_2}) \nonumber \\
& \times  \Big( \frac{\overline{r_1}}{r_1} \Big)^{-n_1} \Big( \frac{\overline{r_2}}{r_2} \Big)^{n_2} 
\omega(m_1;1,k_1) \overline{\omega(m_2,1,k_2)} \nonumber \\
& \times \sum_{\substack{ \nu_1 \in \lambda^{-2m_1-3} \mathbb{Z}[\omega] \\ \nu_2 \in \lambda^{-2m_2-3} \mathbb{Z}[\omega] } }
\rho_{f \otimes \widehat{\psi_{\lambda^{m_1}}(\cdot)_{\lambda u}},\mathfrak{c}(m_1,1;k_1)}(\nu_1)
\overline{\rho_{f \otimes \widehat{\psi_{\lambda^{m_2}}(\cdot)_{\lambda u}},\mathfrak{c}(m_2,1;k_2)}(\nu_2)}  \nonumber \\
& \times \Big(\frac{\nu_1}{|\nu_1|} \Big)^{-n_1} \Big( \frac{\overline{\nu_2}}{|\nu_2|} \Big)^{-n_2}  \nonumber \\
& \times \ddot{H}_{K,C^2/N(\lambda^{\ell}t q^{\prime} q^{\prime \prime} ) } \Big(\frac{N(\nu_1)}{N(t_1 r^{\prime}_1 r^{\prime \prime}_1)^2/X},\frac{N(\nu_2)}{N(t_2 r^{\prime}_2 r^{\prime \prime}_2)^2/X},\boldsymbol{n} \Big)  \nonumber \\
& \times \mathscr{C}_1(\boldsymbol{a},\boldsymbol{\nu},\lambda^{\ell} t q^{\prime} q^{\prime \prime}, \boldsymbol{\lambda}^{\boldsymbol{m}} \boldsymbol{t} \boldsymbol{r^{\prime}} \boldsymbol{r^{\prime \prime}});
\end{align}

\begin{align} \label{S2defn}
& S_{2f}(\cdots) \nonumber \\
&:=\delta_{\boldsymbol{m} \in [0, \min \{5,\ell_0+2\}] \times [6,\ell_0+2]}
N(\lambda^{m_1}) N(\lambda^4) \Big(\frac{\overline{r_1}}{r_1}  \Big)^{-n_1} \omega(m_1;1,k_1)   \nonumber \\
& \times \sum_{\substack{ \nu_1 \in \lambda^{-2m_1-3} \mathbb{Z}[\omega] \\ \nu_2 \in \lambda^{-3} \mathbb{Z}[\omega] } }
\rho_{f \otimes \widehat{\psi_{\lambda^{m_1}}(\cdot)_{\lambda u}},\mathfrak{c}(m_1,1;k_1)}(\nu_1)
\overline{\rho_f(\nu_2)}  \nonumber \\
& \times \Big(\frac{\nu_1}{|\nu_1|} \Big)^{-n_1} \Big( \frac{\overline{\nu_2}}{|\nu_2|} \Big)^{-n_2} \nonumber \\
& \times \ddot{H}_{K,C^2/N(\lambda^{\ell}t q^{\prime} q^{\prime \prime}) } \Big(\frac{N(\nu_1)}{N(t_1 r^{\prime}_1 r^{\prime \prime}_1)^2/X},\frac{N(\nu_2)}{N(\lambda^{m_2-4} t_2 r^{\prime}_2 r^{\prime \prime}_2)^2/X},\boldsymbol{n} \Big)  \nonumber \\
& \times \mathscr{C}_2(\boldsymbol{a},\boldsymbol{\nu},\lambda^{\ell} t q^{\prime} q^{\prime \prime}, \boldsymbol{\lambda}^{\boldsymbol{m}} \boldsymbol{t} \boldsymbol{r^{\prime}} \boldsymbol{r^{\prime \prime}});
\end{align}

\begin{align} \label{S3defn}
& S_{3f}(\cdots) \nonumber \\
&:=  \delta_{\boldsymbol{m} \in [6,\ell_0+2] \times [0, \min \{5,\ell_0+2\}]}  
N(\lambda^{4}) N(\lambda^{m_2})
\Big( \frac{\overline{r_2}}{r_2} \Big)^{n_2} \overline{\omega(m_2;1,k_2)} \nonumber \\
& \times 
\sum_{\substack{ \nu_1 \in \lambda^{-3} \mathbb{Z}[\omega] \\ \nu_2 \in \lambda^{-2m_2-3} \mathbb{Z}[\omega] } }
\rho_f(\nu_1)
\overline{\rho_{f \otimes \widehat{\psi_{\lambda^{m_2}}(\cdot)_{\lambda u}},\mathfrak{c}(m_2,1;k_2)}(\nu_2)}  \nonumber \\
& \times \Big(\frac{\nu_1}{|\nu_1|} \Big)^{-n_1} \Big( \frac{\overline{\nu_2}}{|\nu_2|} \Big)^{-n_2}  \nonumber \\
& \times \ddot{H}_{K,C^2/N(\lambda^{\ell} t q^{\prime} q^{\prime \prime} )} \Big(\frac{N(\nu_1)}{N(\lambda^{m_1-4} t_1 r^{\prime}_1 r^{\prime \prime}_1)^2/X},\frac{N(\nu_2)}{N(t_2 r^{\prime}_2 r^{\prime \prime}_2)^2/X},\boldsymbol{n} \Big)  \nonumber \\
& \times  \mathscr{C}_3(\boldsymbol{a}, \boldsymbol{\nu},\lambda^{\ell} t q^{\prime} q^{\prime \prime}, \boldsymbol{\lambda}^{\boldsymbol{m}} \boldsymbol{t} \boldsymbol{r^{\prime}} \boldsymbol{r^{\prime \prime}});
\end{align}

\begin{align} \label{S4defn}
& S_{4f}(\cdots) \nonumber \\
& := \delta_ {\boldsymbol{m} \in [6,\ell_0+2]^2} \cdot 
N(\lambda^8) \nonumber \\
& \times \sum_{\substack{ \nu_1,\nu_2 \in \lambda^{-3} \mathbb{Z}[\omega] } }
\rho_f(\nu_1) \overline{\rho_f(\nu_2)} \Big(\frac{\nu_1}{|\nu_1|} \Big)^{-n_1} \Big( \frac{\overline{\nu_2}}{|\nu_2|} \Big)^{-n_2}  \nonumber \\
& \times \ddot{H}_{K,C^2/N(\lambda^{\ell} t q^{\prime} q^{\prime \prime}) } \Big(\frac{N(\nu_1)}{N(\lambda^{m_1-4} t_1 r^{\prime}_1 r^{\prime \prime}_1)^2/X},\frac{N(\nu_2)}{N(\lambda^{m_2-4} t_2 r^{\prime}_2 r^{\prime \prime}_2)^2/X},\boldsymbol{n} \Big)  \nonumber \\
& \times \mathscr{C}_4(\boldsymbol{a},\boldsymbol{\nu},\lambda^{\ell} t q^{\prime} q^{\prime \prime}, \boldsymbol{\lambda}^{\boldsymbol{m}} \boldsymbol{t} \boldsymbol{r^{\prime}} \boldsymbol{r^{\prime \prime}});
\end{align}

\begin{align} \label{C1def}
& \mathscr{C}_1(\cdots) \nonumber \\
&=\sum_{\substack{ \boldsymbol{J} \in (\mathbb{Z}[\omega]/ \lambda^{\ell_0-e_v} t \mathbb{Z}[\omega])^2  \\ \boldsymbol{J^{\prime}} \in (\mathbb{Z}[\omega]/ q^{\prime} \mathbb{Z}[\omega])^2 \\ \boldsymbol{J^{\prime \prime}} \in (\mathbb{Z}[\omega]/ q^{\prime \prime} \mathbb{Z}[\omega])^2  }} 
\widehat{\psi}_{\lambda^{\ell}t}\Big(\frac{\lambda^{e_v}v _0( J_2-J_1)+u(Y_2-Y_1)}{\lambda}\Big)  
\widehat{\psi}_{q^{\prime}}(J^{\prime}_2- J^{\prime}_1) \widehat{\psi}_{q^{\prime \prime}}(J^{\prime \prime}_2- J^{\prime \prime}_1)   \nonumber \\
& \times \psi^{\star}_{t_1 r^{\prime}_1 r^{\prime \prime}_1}(\cdot)_{\lambda^{2m_1+1} a_1 j_1}(\lambda^{2m_1+4} \nu_1 )
\overline{\psi^{\star}_{t_2 r^{\prime}_2 r^{\prime \prime}_2 }(\cdot)_{\lambda^{2m_2+1} a_2 j_2}(\lambda^{2m_2+4} \nu_2 )}; 
\end{align}

\begin{align} \label{C2def}
& \mathscr{C}_2(\cdots) \nonumber \\
&:=\sum_{\zeta_2}  \Big( \frac{\overline{\zeta_2^{-1} \lambda^{m_2-4} r_2 }}{\zeta_2^{-1} \lambda^{m_2-4} r_2 }  \Big)^{n_2} \nonumber \\
& \times  \sum_{\substack{ \boldsymbol{J} \in (\mathbb{Z}[\omega]/ \lambda^{\ell_0-e_v} t \mathbb{Z}[\omega])^2  \\ \boldsymbol{J^{\prime}} \in (\mathbb{Z}[\omega]/ q^{\prime} \mathbb{Z}[\omega])^2 \\ \boldsymbol{J^{\prime \prime}} \in (\mathbb{Z}[\omega]/ q^{\prime \prime} \mathbb{Z}[\omega])^2  }}   
\widehat{\psi}_{\lambda^{\ell}t}\Big(\frac{\lambda^{e_v} v_0 (J_2- J_1)+u(Y_2-Y_1)}{\lambda}\Big)  
\widehat{\psi}_{q^{\prime}}(J^{\prime}_2 - J^{\prime}_1) 
\widehat{\psi}_{q^{\prime \prime}}( J^{\prime \prime}_2- J^{\prime \prime}_1)   \nonumber \\
& \times \psi^{\star}_{t_1 r^{\prime}_1 r^{\prime \prime}_1}(\cdot)_{\lambda^{2m_1+1} a_1  j_1}(\lambda^{2m_1+4} \nu_1 ) 
\overline{\psi^{\#}_{\lambda^{m_2} t_2 r^{\prime}_2 r^{\prime \prime}_2}(\cdot)_{\lambda a_2 j_2,\zeta_2^{-1}}(\lambda^4 \nu_2)}; 
\end{align}

\begin{align} \label{C3def}
& \mathscr{C}_3(\cdots) \nonumber \\
&:=\sum_{\zeta_1} \Big( \frac{\overline{\zeta_1^{-1} \lambda^{m_1-4} r_1 }}{\zeta_1^{-1} \lambda^{m_1-4} r_1 }  \Big)^{-n_1} \nonumber \\
& \times \sum_{\substack{ \boldsymbol{J} \in (\mathbb{Z}[\omega]/ \lambda^{\ell_0-e_v} t \mathbb{Z}[\omega])^2  \\ \boldsymbol{J^{\prime}} \in (\mathbb{Z}[\omega]/ q^{\prime} \mathbb{Z}[\omega])^2 \\ \boldsymbol{J^{\prime \prime}} \in (\mathbb{Z}[\omega]/ q^{\prime \prime} \mathbb{Z}[\omega])^2  }}  
\widehat{\psi}_{\lambda^{\ell}t}\Big(\frac{\lambda^{e_v} v_0(J_2- J_1)+u(Y_2-Y_1)}{\lambda}\Big)  
\widehat{\psi}_{q^{\prime}}( J^{\prime}_2- J^{\prime}_1) \widehat{\psi}_{q^{\prime \prime}}( J^{\prime \prime}_2- J^{\prime \prime}_1)  \nonumber \\
& \times \psi^{\#}_{\lambda^{m_1} t_1 r^{\prime}_1 r^{\prime \prime}_1}(\cdot)_{\lambda a_1 j_1,\zeta_1^{-1}}(\lambda^4 \nu_1) 
\overline{\psi^{\star}_{t_2 r^{\prime}_2 r^{\prime \prime}_2 }(\cdot)_{\lambda^{2m_2+1} a_2 j_2}(\lambda^{2m_2+4} \nu_2 )}; 
\end{align}

\begin{align} \label{C4def}
& \mathscr{C}_4(\cdots) \nonumber \\
 &:=\sum_{\boldsymbol{\zeta}}
 \Big( \frac{\overline{\zeta_1^{-1} \lambda^{m_1-4} r_1 }}{\zeta_1^{-1} \lambda^{m_1-4} r_1 }  \Big)^{-n_1} 
\Big( \frac{\overline{\zeta_2^{-1} \lambda^{m_2-4} r_2 }}{\zeta_2^{-1} \lambda^{m_2-4} r_2 }  \Big)^{n_2} \nonumber \\
& \times \sum_{\substack{ \boldsymbol{J} \in (\mathbb{Z}[\omega]/ \lambda^{\ell_0-e_v} t \mathbb{Z}[\omega])^2  \\ \boldsymbol{J^{\prime}} \in (\mathbb{Z}[\omega]/ q^{\prime} \mathbb{Z}[\omega])^2 \\ \boldsymbol{J^{\prime \prime}} \in (\mathbb{Z}[\omega]/ q^{\prime \prime} \mathbb{Z}[\omega])^2  }}   
\widehat{\psi}_{\lambda^{\ell}t}\Big(\frac{\lambda^{e_v} v_0(J_2-J_1)+u(Y_2-Y_1)}{\lambda}\Big) 
 \widehat{\psi}_{q^{\prime}}( J^{\prime}_2-  J^{\prime}_1) 
 \widehat{\psi}_{q^{\prime \prime}}( J^{\prime \prime}_2- J^{\prime \prime}_1) \nonumber \\  
& \times \psi^{\#}_{\lambda^{m_1} t_1 r^{\prime}_1 r^{\prime \prime}_1}(\cdot)_{\lambda a_1 j_1,\zeta_1^{-1}}(\lambda^4 \nu_1)  
\overline{\psi^{\#}_{\lambda^{m_2} t_2 r^{\prime}_2 r^{\prime \prime}_2}(\cdot)_{\lambda a_2 j_2,\zeta_2^{-1}}(\lambda^4 \nu_2)},
\end{align}
and $\boldsymbol{j}$ is given by \eqref{jidefn} with subsequent change of variable \eqref{subsvariablechange}.

\begin{remark} \label{mutmand}
Recalling \eqref{Lsum} and the averaging over $\boldsymbol{a}$, we have 
\begin{equation} \label{averagestatement}
\sum_{\boldsymbol{a}} \mu^2(a_1) \alpha_{a_1} \mu^2(a_2)
\overline{\alpha_{a_2}} \mathscr{L}_{f}(\boldsymbol{a},\cdots)=\sum_{p=1}^4 \sum_{\boldsymbol{a}} \mu^2(a_1) \alpha_{a_1} \mu^2(a_2) \overline{\alpha_{a_2}} 
\mathscr{M}_{pf}(\boldsymbol{a},\cdots).
\end{equation}
We first treat $p=4$ in \eqref{averagestatement}. After the local calculations we give the modifications for $p=1,2,3$, including their different support at $\lambda$.
\end{remark}

\subsection{Evaluation and bounds for arithmetic exponential sums} \label{arithexpsum}
We first compute and bound $\mathscr{C}_4(\cdots)$ in \eqref{C4def}.

A computation using Lemma \ref{unpack1}, \eqref{cubicintermed}, \eqref{jidefn},
the Chinese Remainder Theorem 
(with pairwise co-prime moduli $3\zeta_i \lambda^{m_i-1} t_i$, $r^{\prime}_i$ and $r^{\prime \prime}_i$ for $i=1,2$), 
cubic reciprocity, and Lemma \ref{kloostermanexp2} yields
\begin{equation}  \label{exact1}
\mathscr{C}_4(\cdots)=\sum_{\boldsymbol{\zeta}} \Big( \frac{\overline{\zeta_1 \lambda^{m_1-4} r_1 }}{\zeta_1 \lambda^{m_1-4} r_1 }  \Big)^{-n_1} 
\Big( \frac{\overline{\zeta_2 \lambda^{m_2-4} r_2 }}{\zeta_2 \lambda^{m_2-4} r_2 }  \Big)^{n_2}
\prod_{i=1}^{3} G_{4i}(\boldsymbol{a},\boldsymbol{\nu},\lambda^{\ell} t q^{\prime} q^{\prime \prime}, \boldsymbol{\zeta} \boldsymbol{\lambda}^{\boldsymbol{m}} \boldsymbol{t} \boldsymbol{r^{\prime}} \boldsymbol{r^{\prime \prime}}),
\end{equation}
where 
\begin{align} 
& G_{41}(\cdots)   \nonumber \\
&:= \frac{1}{N(\lambda^{m_1+1} t_1) N(\lambda^{m_2+1} t_2)} 
\sum_{\boldsymbol{J} \in (\mathbb{Z}[\omega]/ \lambda^{\ell_0-e_v} t \mathbb{Z}[\omega])^2}
\widehat{\psi}_{\lambda^{\ell}t}\Big(\frac{\lambda^{e_v} v_0(J_2-J_1)+u(Y_2-Y_1)}{\lambda}\Big) \nonumber \\
& \times K_{\Gamma_1(3),\sigma,\sigma}(\overline{r^{\prime}_1 r^{\prime \prime}_1}(\lambda^3 \nu_1), \overline{r^{\prime}_1 r^{\prime \prime}_1} (a_1 \lambda^{e_v} v_0 J_1+a_1 Y_1 u), \zeta_1 \lambda^{m_1-1} t_1)  \nonumber \\
& \times \overline{K_{\Gamma_1(3),\sigma,\sigma}(\overline{r^{\prime}_2 r^{\prime \prime}_2}(\lambda^3 \nu_2), \overline{r^{\prime}_2 r^{\prime \prime}_2}( a_2 \lambda^{e_v} v_0 J_2+a_2 Y_2 u), \zeta_2 \lambda^{m_2-1} t_2)}; \label{Gdef} \\
& G_{42}(\cdots)  \nonumber \\
&:= \frac{1}{N(r^{\prime}_1 r^{\prime}_2) }  \sum_{\boldsymbol{J^{\prime}} \in (\mathbb{Z}[\omega]/q^{\prime} \mathbb{Z}[\omega])^2  }  \widehat{\psi}_{q^{\prime}}( J^{\prime}_2-  J^{\prime}_1) \nonumber \\ & \times K_{\Gamma_1(3),\sigma,\xi}(\overline{\zeta_1 \lambda^{m_1-1} t_1 r^{\prime \prime}_1 } (\lambda^3 \nu_1),\overline{\zeta_1 \lambda^{m_1-1} t_1 r^{\prime \prime}_1 }( a_1 J^{\prime}_1), r^{\prime}_1)  \nonumber \\
& \times \overline{K_{\Gamma_1(3),\sigma,\xi}(\overline{\zeta_2 \lambda^{m_2-1} t_2 r^{\prime \prime}_2 } (\lambda^3 \nu_2), \overline{\zeta_2 \lambda^{m_2-1} t_2 r^{\prime \prime}_2 }(a_2 J^{\prime}_2), r^{\prime}_2)};  \label{Gprimedef} \\
& G_{43}(\cdots)  \nonumber \\
&:=\frac{1}{N(r^{\prime \prime}_1 r^{\prime \prime}_2 )}   \sum_{\substack{ \boldsymbol{J^{\prime \prime}} \in (\mathbb{Z}[\omega]/ q^{\prime \prime} \mathbb{Z}[\omega])^2 }} 
\widehat{\psi}_{q^{\prime \prime}}( J^{\prime \prime}_2- J^{\prime \prime}_1) \nonumber \\
& \times K_{\Gamma_1(3),\sigma,\xi}( \overline{\zeta_1 \lambda^{m_1-1} t_1 r^{\prime}_1 }(\lambda^3 \nu_1), \overline{\zeta_1 \lambda^{m_1-1} t_1 r^{\prime}_1 }(a_1 J^{\prime \prime}_1), r^{\prime \prime}_1) \nonumber \\ 
& \times \overline{K_{\Gamma_1(3),\sigma,\xi}(\overline{\zeta_2 \lambda^{m_2-1} t_2 r^{\prime}_2 }(\lambda^3 \nu_2), \overline{\zeta_2 \lambda^{m_2-1} t_2 r^{\prime}_2 }(a_2 J^{\prime \prime}_2), r^{\prime \prime}_2)}. \label{Gprimeprimedef}
\end{align}
We now evaluate and bound each \eqref{Gdef}--\eqref{Gprimeprimedef}.

\subsubsection{Treatment of \eqref{Gprimedef}} \label{Gprimesec}
For this calculation write $q=q'$, $r_i=r_i'$, $\kappa_i=\lambda^3\nu_i$,
and choose
\[
\beta_i=\overline{\zeta_i\lambda^{m_i-1}t_i r_i''}
\pmod{a_iq},\qquad c_i=a_iq/r_i.
\]
All the $\beta_i$ are units modulo $a_iq$. Opening the Ramanujan sum
and the two Kloosterman sums, the $J_i'$-sum is zero unless
$\beta_i c_i d_i\equiv k\pmod q$, where $k$ is a unit modulo $q$.
Since $r_i\mid a_iq$, this forces $(c_i,q)=1$, hence $q\mid r_i$.
Consequently the exact result is
\begin{align} \label{orthog1}
G_{42}(\cdots)
&=\frac{N(q)}{N(r_1r_2)}\prod_{i=1}^2\delta_{(c_i,q)=1}
\sum_{\substack{d_i\pmod{r_i},\ (d_i,r_i)=1\ (i=1,2)\\
\beta_1c_1d_1\equiv\beta_2c_2d_2\pmod q}}
\Big(\frac{d_1}{r_1}\Big)_3
\overline{\Big(\frac{d_2}{r_2}\Big)_3}
\check e\Big(\frac{\beta_1\kappa_1\overline d_1}{r_1}
-\frac{\beta_2\kappa_2\overline d_2}{r_2}\Big).
\end{align}
Here $d_i\overline d_i\equiv1\pmod{r_i}$. Put $r_i=qs_i$,
so that $s_i\mid a_i$ and $(a_i/s_i,q)=1$. Invert the variables $d_i$
and detect their resulting congruence with additive characters modulo $q$.
Using $(-1/r_i)_3=1$ and the normalization \eqref{normalised}, this gives
\begin{align} \label{Gprimeintermed}
G_{42}(\cdots)
&=\frac{\prod_i\delta_{(a_i/s_i,q)=1}}
 {N(q)N(s_1s_2)^{1/2}}
\sum_{k\pmod q}
\overline{\widetilde g(y_1\kappa_1+kz_1,qs_1)}
\widetilde g(y_2\kappa_2+kz_2,qs_2),
\end{align}
where
\begin{align}
y_1=\beta_1,\qquad z_1=\beta_2(a_2/s_2)s_1,\label{y1z1}\\
y_2=\beta_2,\qquad z_2=\beta_1(a_1/s_1)s_2.\label{y2z2}
\end{align}
Let $g_i=(s_i,q)$ and $h_i=s_i/g_i$. Since the $a_i$ are squarefree,
$(qg_i,h_i)=1$ and $h_i$ is squarefree. Factor each Gauss sum in
\eqref{Gprimeintermed} over $qg_i$ and $h_i$. Lemmas \ref{elemlem}
and \ref{sqsupport} show that it vanishes unless
\begin{equation} \label{Gprimeintermed2}
g_i\mid\kappa_i,\qquad (\kappa_i,h_i)=1\qquad(i=1,2).
\end{equation}
Under these conditions the normalized $h_i$-factor has absolute value
one. Lemma \ref{sqrootbd} bounds the remaining factor by
\[
N(g_i)^{1/2}
N\big((y_i\kappa_i/g_i+kz_i/g_i,q)\big)^{1/2}.
\]
The coefficients $z_i/g_i$ are units modulo $q$. Cauchy--Schwarz,
followed by the divisor bound for the average of $N((k,q))$, therefore gives
\begin{equation} \label{Gprimeintbd}
\sum_{k\pmod q}
|\widetilde g(y_1\kappa_1+kz_1,qs_1)
\widetilde g(y_2\kappa_2+kz_2,qs_2)|
\ll_\varepsilon N(q)^{1+\varepsilon}N(g_1g_2)^{1/2}.
\end{equation}
Returning to the original notation, we obtain
\begin{align} \label{orthog1bd}
|G_{42}(\cdots)|
&\ll_\varepsilon N(q')^{1+\varepsilon}
\prod_{i=1}^2\left(
\delta_{(a_iq'/r_i',q')=1}
\delta_{(\lambda^3\nu_i,(r_i'/q')/(r_i'/q',q'))=1}
\delta_{(r_i'/q',q')\mid\lambda^3\nu_i}
\frac{N((r_i'/q',q'))^{1/2}}{N(r_i')^{1/2}}
\right).
\end{align}
As above, the right side is interpreted as zero unless $q'\mid r_i'$.

\subsubsection{Treatment of \eqref{Gprimeprimedef}} \label{Gprimeprimesec}
We open the normalised Ramanujan sums and cubic Kloosterman sums in \eqref{Gprimeprimedef} and
use orthogonality in the $J^{\prime \prime}_1,J^{\prime \prime}_2$ variables. By a similar argument to the above
we reassemble the result to obtain
\begin{align} \label{orthog2}
& G_{43}(\cdots) \nonumber \\
&=\Big(\prod_{i=1}^{2} \delta_{r^{\prime \prime}_i=q^{\prime \prime} } \Big)
\cdot \overline{\Big(\frac{a_1}{q^{\prime \prime}} \Big)_3} \Big( \frac{\zeta_1 \lambda^{m_1-1} t_1 r^{\prime}_1 }{q^{\prime \prime}} \Big)_3 \Big( \frac{a_2}{q^{\prime \prime}} \Big)_3
\overline{\Big( \frac{\zeta_2 \lambda^{m_2-1} t_2 r^{\prime}_2 }{q^{\prime \prime}} \Big)_3} \nonumber \\
& \times \widehat{\psi}_{q^{\prime \prime}}(P_1 \lambda^3 \nu_1-P_2 \lambda^3 \nu_2),
\end{align}
where
\begin{equation}
P_1:=(\zeta_2 \lambda^{m_2-1} t_2 r^{\prime}_2)^2 a_1 \quad \text{and} \quad P_2:=(\zeta_1 \lambda^{m_1-1} t_1 r^{\prime}_1)^2 a_2.
\end{equation}
We have the bound
\begin{align} \label{orthog2bd}
|G_{43}(\cdots)|
& \leq \Big(\prod_{i=1}^{2} \delta_{r^{\prime \prime}_i=q^{\prime \prime} } \Big) \cdot |\widehat{\psi}_{q^{\prime \prime}}(P_1 \lambda^3 \nu_1-P_2 \lambda^3 \nu_2)|.
\end{align}

\subsubsection{Treatment of \eqref{Gdef}} \label{Gsec}
Recall that $\ell_0:=\max \{\ell,e_v\}$.
We open the normalised Ramanujan sums and cubic Kloosterman sums in \eqref{Gdef} and
use orthogonality in the $J_1,J_2$ variables. We then reassemble the result
to obtain
\begin{align} \label{orthog3}
&G_{41}(\cdots) \nonumber \\
&=\frac{N(\lambda^{\ell_0-e_v} t)^2}{N(\lambda^{\ell} t)} \Big( \prod_{i=1}^2 \frac{1}{N(\lambda^{m_i+1} t_i) } \Big)  
\sum_{\substack{ k \pmod{\lambda^{\ell} t} \\ (k,\lambda^{\ell} t)=1}} \check{e} \Big(\frac{ku(Y_2-Y_1)}{\lambda^{\ell+1}t} \Big)  \nonumber \\
& \sum_{\boldsymbol{x} \in \mathscr{B}_1(k) \times \mathscr{B}_2(k) }
\Big(\frac{\zeta_1 \lambda^{m_1-1} t_1 }{x_1}  \Big)_3 \overline{\Big(\frac{ \zeta_2 \lambda^{m_2-1}t_2 }{x_2}  \Big)_3} \nonumber \\
& \times \check{e} \Big(\frac{ \overline{r^{\prime}_1 r^{\prime \prime}_1} (\lambda^3 \nu_1 \overline{x_1}+a_1 Y_1 u x_1)}{\zeta_1 \lambda^{m_1-1} t_1 }- \frac{\overline{r^{\prime}_2 r^{\prime \prime}_2}(\lambda^3 \nu_2 \overline{x_2}+a_2 Y_2 u x_2)}{\zeta_2 \lambda^{m_2-1} t_2 } \Big),
\end{align}
where for $i=1,2$ we have
\begin{align} \label{Bidef}
\mathscr{B}_i(k):=\{ x_i & \pmod{\lambda^{m_i+1} t_i} : (x_i,\lambda t_i)=1, \quad x_i \equiv 1 \pmod{3}, \nonumber  \\
& \overline{r^{\prime}_i r^{\prime \prime}_i \zeta_i} a_i \lambda^{\ell_0-m_i+2}  (v_0 t/t_i) x_i \equiv k \lambda^{\ell_0-\ell} v_0  \pmod{\lambda^{\ell_0-e_v} t} \}.
\end{align}
Put $d=\ell_0-e_v$. At primes dividing $t$, the congruence in
\eqref{Bidef} is a unit multiple of the map
\[
\mathbb Z[\omega]/t_i\mathbb Z[\omega]\longrightarrow
\mathbb Z[\omega]/t\mathbb Z[\omega],\qquad
x\longmapsto(v_0t/t_i)x.
\]
Its image has size $N(t_i/(v_0,t_i))$, and its kernel has size
$N((v_0,t_i))$. Since $k$ is a unit modulo $t$, its required right
side belongs to the image if and only if
\begin{equation} \label{gcddiv2}
t\mid[v_0,t_i].
\end{equation}
At $\lambda$, if $\ell\leq e_v$ then $d=0$ and there is no congruence
to impose. If $\ell>e_v$, then $\ell_0=\ell$ and the right side of
\eqref{Bidef} is a unit. Since $m_i\leq\ell_0+2$, this forces
$m_i=\ell+2$. In that case it determines $x_i$ modulo $\lambda^d$.
Thus the possible ranges are
\begin{equation} \label{ramifiedsupport}
\begin{cases}
0\leq\ell\leq e_v,\quad 6\leq m_1,m_2\leq e_v+2,\\
\text{or}\quad \ell>e_v,\quad m_1=m_2=\ell+2,
\end{cases}
\qquad t\mid[v_0,t_i]\quad(i=1,2).
\end{equation}
In these ranges the number of inverse variables satisfies
\begin{equation} \label{trivB}
|\mathscr B_i(k)|
\leq\delta_{t\mid[v_0,t_i]}
N(\lambda^{m_i+1-d})N((v_0,t_i)).
\end{equation}
There are at most $N(\lambda^\ell t)$ choices of $k$. Substituting
\eqref{trivB} into \eqref{orthog3}, all powers of $\lambda$ cancel,
and we obtain
\begin{align} \label{orthog3bbd}
|G_{41}(\cdots)|
&\leq\prod_{i=1}^2\left(
\delta_{t\mid[v_0,t_i]}
N\Big(\frac{t(v_0,t_i)}{t_i}\Big)\right)
\leq N(v_0)^2\prod_{i=1}^2\delta_{t\mid[v_0,t_i]}.
\end{align}
The last inequality follows prime by prime from $t_i\mid v_0t$
and $t\mid[v_0,t_i]$. Outside \eqref{ramifiedsupport}, $G_{41}=0$.
In particular, the weaker bound with $N(v_0)^6$ is also available,
but the range $\ell>e_v$ is retained.

\subsubsection{The cases containing a small $m_i$}
For $0\leq m_i\leq5$, use Lemma \ref{unpack2} and put
$\kappa_i=\lambda^{2m_i+3}\nu_i$. For $m_i\geq6$, put
$\kappa_i=\lambda^3\nu_i$. The Fourier coefficients with $m_i\leq5$ in
\eqref{S1defn}--\eqref{S3defn} belong to a fixed finite collection of
cusp forms because $m_i\leq5$, $k_i$ is taken modulo $\lambda^{14}$, and
$a_ij_i\equiv u\pmod{\lambda^{e_v}v_0}$ with $e_v\geq14$.
Factor their arithmetic sums by the same CRT decomposition as
\eqref{exact1}, denoting the three local factors by $G_{p1},G_{p2},G_{p3}$.
When $0\leq m_i\leq5$, the $t$-part is a Kloosterman sum modulo $t_i$, normalized
by $N(t_i)$, and its second argument is a unit multiple of $a_iJ_i$.
The map at primes dividing $t$ is consequently the same map used above,
with kernel of size $N((v_0,t_i))$.
If $\ell>e_v$ and $m_i\leq5$, the sum over $J_i\pmod{\lambda^{\ell-e_v}}$
is zero because the Kloosterman factor has no ramified denominator,
whereas the primitive delta character has nonzero frequency modulo
$\lambda^{\ell-e_v}$. Thus
\begin{equation} \label{smallramifiedbound}
|G_{p1}|\ll\delta_{0\leq\ell\leq e_v}
N(v_0)^2\prod_{i=1}^2\delta_{t\mid[v_0,t_i]},\qquad p=1,2,3.
\end{equation}
For $\ell\leq e_v$, this follows from the preceding kernel count.
For each $i$ with $m_i\geq6$, the number of ramified inverse variables cancels
the corresponding ramified normalization. When $m_i\leq5$, neither factor occurs.

At the primes dividing $q'$ or $q''$, all additional powers of
$\lambda$ are units. The calculation \eqref{orthog1}--\eqref{orthog1bd}
therefore applies with $\lambda^3\nu_i$ replaced by $\kappa_i$.
For clarity, if the local Kloosterman factors are
$K(A_i,B_i a_iJ_i;r_i)$, orthogonality gives
$B_i(a_iq/r_i)d_i\equiv k\pmod q$. For $q=q''$ it forces
$r_1=r_2=q''$, and the remaining normalized sum is
\[
\overline{\Big(\frac{B_1a_1}{q''}\Big)_3}
\Big(\frac{B_2a_2}{q''}\Big)_3
\widehat\psi_{q''}(A_1B_1a_1-A_2B_2a_2).
\]
When $m_i\geq6$,
$A_i=\kappa_i/(\zeta_i\lambda^{m_i-1}t_ir_i')$ and
$B_i=1/(\zeta_i\lambda^{m_i-1}t_ir_i')$ modulo $q''$.
When $0\leq m_i\leq5$,
$A_i=\kappa_i/(\lambda^{2m_i+3}t_ir_i')$ and
$B_i=1/(\lambda^3t_ir_i')$ modulo $q''$.
Consequently, after multiplication of the Ramanujan argument by a unit,
\eqref{orthog2bd} holds in every case with argument
\begin{equation} \label{allcasesdiagonal}
b_2^2a_1\kappa_1-b_1^2a_2\kappa_2,\qquad
b_i=\begin{cases}
\lambda^{m_i+3}t_ir_i',&0\leq m_i\leq5,\\
\zeta_i\lambda^{m_i-1}t_ir_i',&m_i\geq6.
\end{cases}
\end{equation}
The powers of $\lambda$ arising when $m_i\leq5$ are absolutely bounded. This establishes
both the support and the arithmetic bounds needed for all four cases.

\subsection{Further technical manipulations and insertion of smooth dyadic partitions of unity} \label{posturing}
We substitute \eqref{S4defn} into  \eqref{Midef} to obtain 
\begin{align} \label{M4fcompact}
& \mathscr{M}_{4f}(\boldsymbol{a},\cdots) \nonumber \\
&=\frac{N(\lambda^8) k_C}{C^4} \frac{X^2}{N(a_1 a_2) N(v_0)^2} \sum_{0 \leq \ell \ll 1+\log C}  \sum_{\substack{t q^{\prime} q^{\prime \prime} \in \mathbb{Z}[\omega] \\ \eqref{tqqcond}   }}   
\frac{N(\lambda^{\ell})}{N(\lambda^{\ell_0+2})^2} 
\frac{1}{N(t q^{\prime} q^{\prime \prime} )} \sum_{\substack{ \boldsymbol{k} \in (\mathbb{Z}[\omega]/ \lambda^{14} \mathbb{Z}[\omega])^2 
\\ \forall i: k_i \equiv 1 \pmod{3} }} \nonumber \\
& \times \sum_{\substack{\boldsymbol{m},\boldsymbol{t},\boldsymbol{r}^{\prime},\boldsymbol{r}^{\prime \prime} \\ \forall i: 6 \leq  m_i \leq \ell_0+2 \\ \eqref{trrcond} }}   
\sum_{\boldsymbol{n} \in \mathbb{Z}^2} (-1)^{n_1+n_2}
\sum_{\substack{ \boldsymbol{\nu} \in (\lambda^{-3} \mathbb{Z}[\omega])^2 } }
\rho_f(\nu_1) \overline{\rho_f(\nu_2)} \Big(\frac{\nu_1}{|\nu_1|} \Big)^{-n_1} \Big( \frac{\overline{\nu_2}}{|\nu_2|} \Big)^{-n_2}  \nonumber \\
& \times \ddot{H}_{K,C^2/N(\lambda^{\ell} t q^{\prime} q^{\prime \prime}) } \Big(\frac{N(\nu_1)}{N(\lambda^{m_1-4} t_1 r^{\prime}_1 r^{\prime \prime}_1)^2/X},\frac{N(\nu_2)}{N(\lambda^{m_2-4} t_2 r^{\prime}_2 r^{\prime \prime}_2)^2/X},\boldsymbol{n} \Big)  \nonumber \\
& \times \mathscr{C}_4(\boldsymbol{a},\boldsymbol{\nu},\lambda^{\ell} t q^{\prime} q^{\prime \prime}, \boldsymbol{\lambda}^{\boldsymbol{m}} \boldsymbol{t} \boldsymbol{r^{\prime}} \boldsymbol{r^{\prime \prime}}),
\end{align}
where $\mathscr{C}_4(\cdots)$ is given by \eqref{C4def} (and \eqref{exact1}). Note that
the summands $\mathscr{M}_{4f}(\cdots)$ do not depend on the congruence classes $k_i \pmod{\lambda^{14}}$
(unlike the other $\mathscr{M}_p(\cdots)$ for $p=1,2,3$). Thus the 
sum over $\boldsymbol{k}$ in \eqref{M4fcompact}, and the last condition in \eqref{trrcond} can be dropped.
Equality \eqref{exact1} and 
the delta conditions in \eqref{orthog1} (resp. \eqref{orthog2})
imply that we can make the change of variable 
$\boldsymbol{r^{\prime}} \rightarrow q^{\prime} \boldsymbol{s^{\prime}}$ where $s^{\prime}_i \mid a_i$ and $(a_i/s^{\prime}_i,q^{\prime})=1$ (resp.
 $\boldsymbol{r^{\prime \prime}} \rightarrow \boldsymbol{q^{\prime \prime}}$ where $\boldsymbol{q}^{\prime \prime}=(q^{\prime \prime},q^{\prime \prime})$)
 in \eqref{M4fcompact}. The ramified conditions are exactly
\eqref{ramifiedsupport}. In particular, the additional range
$\ell>e_v$, $m_1=m_2=\ell+2$ remains present.
Thus the sums may be written in terms of $\boldsymbol m,\boldsymbol t,
\boldsymbol s'$ with $6\leq m_i\leq\ell_0+2$, subject to
\eqref{ramifiedsupport} and
\begin{equation} \label{newcond}
t_i \mid v_0t, \quad t \mid [v_0,t_i],  \quad s^{\prime}_i \mid a_i, \quad (a_i /s^{\prime}_i,q^{\prime})=1 \quad \text{for} \quad i=1,2.
\end{equation}
We further note that
the delta conditions in \eqref{Gprimeintermed2} imply that $\boldsymbol{\nu}$ sum in \eqref{M4fcompact}
is supported on the conditions
\begin{equation} \label{nucond}
\lambda^{3} \nu_i \equiv 0 \pmod{(s^{\prime}_i,q^{\prime})} \quad \text{and} \quad 
\Big (\lambda^3 \nu_i, \frac{s^{\prime}_i}{(s^{\prime}_i,q^{\prime})} \Big)=1 \quad \text{for} \quad i=1,2.
\end{equation}
We then insert a smooth partition of unity in the variables $t,q^{\prime}$, and $q^{\prime \prime}$ in \eqref{M4fcompact}. 
Thus
\begin{equation} \label{dyadicdecomp}
\mathscr{M}_{4f}(\boldsymbol{a},\cdots) 
=\sum_{\substack{0 \leq \ell \ll 1+\log C \\ 1/2 \leq T,Q^{\prime},Q^{\prime \prime} \text{dyadic} \\ 
N(\lambda^{\ell}) T Q^{\prime} Q^{\prime \prime} \ll C^2 }}
\mathscr{M}_{4f}(\boldsymbol{a},\cdots,N(\lambda^\ell)TQ^{\prime}Q^{\prime \prime}),
\end{equation}
where 
\begin{align} \label{M4fexact1}
& \mathscr{M}_{4f}(\boldsymbol{a},\cdots,N(\lambda^{\ell})TQ^{\prime}Q^{\prime \prime}) \nonumber \\
&:=\frac{N(\lambda^8) k_C}{C^4} \frac{X^2 N(\lambda^{\ell})}{N(a_1 a_2) N(\lambda^{\ell_0+2} v_0)^2}  \nonumber \\
& \times \sum_{\substack{t q^{\prime} q^{\prime \prime} \in \mathbb{Z}[\omega] \\ \eqref{tqqcond}} } \frac{1}{N(t q^{\prime} q^{\prime \prime} )}   
U \Big( \frac{N(t)}{T} \Big)  U \Big( \frac{N(q^{\prime})}{Q^{\prime}} \Big)  U \Big( \frac{N(q^{\prime \prime})}{Q^{\prime \prime}} \Big)  \nonumber \\
& \times \sum_{\substack{\boldsymbol{m}, \boldsymbol{t},\boldsymbol{s}^{\prime}  \\ 6 \leq m_i \leq \ell_0+2 \\ \eqref{ramifiedsupport},\ \eqref{newcond} }}  
\sum_{\boldsymbol{n} \in \mathbb{Z}^2} (-1)^{n_1+n_2}
\sum_{\substack{\boldsymbol{\nu} \in (\lambda^{-3} \mathbb{Z}[\omega])^2 \\ \eqref{nucond}  } }
\rho_f(\nu_1) \overline{\rho_f(\nu_2)} \Big(\frac{\nu_1}{|\nu_1|} \Big)^{-n_1} \Big( \frac{\overline{\nu_2}}{|\nu_2|} \Big)^{-n_2} \nonumber \\
 & \times 
\ddot{H}_{K,C^2/N(\lambda^{\ell} t q^{\prime} q^{\prime \prime} )} \Big(\frac{N(\nu_1)}{N(\lambda^{m_1-4} t_1 s^{\prime}_1 q^{\prime} q^{\prime \prime})^2/X},\frac{N(\nu_2)}{N(\lambda^{m_2-4} t_2 s^{\prime}_2 q^{\prime} q^{\prime \prime})^2/X},\boldsymbol{n} \Big) \nonumber \\
& \times \mathscr{C}_4(\boldsymbol{a},\boldsymbol{\nu},\lambda^{\ell} t q^{\prime} q^{\prime \prime},\boldsymbol{\lambda}^{\boldsymbol{m}} \boldsymbol{t} q^{\prime} \boldsymbol{s^{\prime}} \boldsymbol{q^{\prime \prime}}).
\end{align}
The restriction 
\begin{equation} \label{restriction}
N(\lambda^{\ell})TQ^{\prime} Q^{\prime \prime} \ll C^2
\end{equation}
in 
\eqref{dyadicdecomp} follows from \eqref{hnonvanish}.

We now use the coupled transform estimate \eqref{HKtransform}.
For a fixed dyadic range put
\begin{align} \label{blockparameters}
&\Lambda_\ell:=N(\lambda^\ell),\quad V:=N(v_0),\quad
E_0:=N(\lambda^{\max\{\ell,e_v\}}),\nonumber\\
&q_0:=TQ'Q'',\quad x_0:=\frac{\Lambda_\ell q_0}{C^2},\quad
\mathcal U:=K+x_0^{-1/2}.
\end{align}
Throughout this argument, factors $(XKN(v))^\varepsilon$ absorb logarithmic
losses. The small positive exponents used in intermediate bounds may be
decreased before they are combined.  Since $x_0\ll1$ and $K\geq1$,
\begin{equation} \label{effectivefrequency}
\mathcal U\ll Kx_0^{-1/2}.
\end{equation}
Define the positive scales
\begin{equation} \label{dualscales}
L_i:=\frac{(N(\lambda^{m_i}t_i s_i')Q'Q'')^2}{X},\qquad i=1,2.
\end{equation}
Fixed powers of $\lambda$, including the factor $\lambda^{-4}$ in
\eqref{M4fexact1}, can be absorbed in the implied constants.
For a sufficiently large fixed $D$ and $L>0$, write
\begin{equation} \label{Omegaweight}
\Omega_{L,\mathcal U}(\nu):=
\left(\frac{N(\nu)}{L}\right)^{-1/4}
\min\left\{1,\left(\frac{L\mathcal U^4}{N(\nu)}\right)^D\right\},
\qquad \nu\ne0.
\end{equation}
All scales $L_i\mathcal U^4$ are bounded by a fixed power of $XKN(v)$.
Indeed, \eqref{newcond}, \eqref{restriction}, and
$m_i\leq\max\{\ell,e_v\}+2$ give
\[
L_i\mathcal U^4\ll K^4V^2A\,E_0^2/\Lambda_\ell^2.
\]
The same observation applies to the dyadic scales introduced below.

After summing the angular modes using \eqref{HKtransform}, apply
\eqref{exact1}, \eqref{orthog1bd}, \eqref{orthog2bd}, and
\eqref{orthog3bbd}.  This gives
\begin{align} \label{L4fintermed}
&|\mathscr M_{4f}(\boldsymbol a,\cdots,\Lambda_\ell TQ'Q'')|
\nonumber\\
&\ll (XKN(v))^\varepsilon K\mathcal P
\sum_{\substack{t,q'\\N(t)\asymp T,\ N(q')\asymp Q'\\
t\mid\operatorname{rad}(v_0)^\infty\\q'\mid\operatorname{rad}(a_1a_2)^\infty}}
\sum_{\substack{\boldsymbol\zeta,\boldsymbol m,\boldsymbol t,\boldsymbol s'\\
6\leq m_i\leq\max\{\ell,e_v\}+2\\
\eqref{ramifiedsupport},\ \eqref{newcond}}}
\prod_{i=1}^2\frac{N((s_i',q'))^{1/2}}{N(s_i')^{1/2}}
\nonumber\\
&\qquad\times
\sum_{\substack{0\ne\nu_1,\nu_2\in\lambda^{-3}\mathbb Z[\omega]\\\eqref{nucond}}}
\prod_{i=1}^2|\rho_f(\nu_i)|\Omega_{L_i,\mathcal U}(\nu_i)
\sum_{\substack{q''\equiv1\pmod3\\N(q'')\asymp Q''\\(q'',a_1a_2v_0)=1}}
|\widehat\psi_{q''}(P_1\lambda^3\nu_1-P_2\lambda^3\nu_2)|,
\end{align}
where
\begin{equation} \label{pretransformprefactor}
\mathcal P:=\frac{X^2V^4\Lambda_\ell}{C^4A^2E_0^2q_0},
\end{equation}
and
\begin{equation} \label{hdef}
P_1:=(\zeta_2\lambda^{m_2-1}t_2q's_2')^2a_1,\qquad
P_2:=(\zeta_1\lambda^{m_1-1}t_1q's_1')^2a_2.
\end{equation}
The ranges denoted by $\asymp$ in \eqref{L4fintermed} are the fixed-factor
ranges supplied by the support of $U$.  The bound $V^6$ from
\eqref{orthog3bbd} has been used in \eqref{pretransformprefactor}.

Drop $(q'',a_1a_2v_0)=1$ by positivity.  Lemma \ref{ramflat}, applied to
finitely many ordinary dyadic intervals, gives, for every $\eta>0$,
\begin{align} \label{Vsumbd}
&\sum_{\substack{q''\equiv1\pmod3\\N(q'')\asymp Q''}}
|\widehat\psi_{q''}(P_1\lambda^3\nu_1-P_2\lambda^3\nu_2)|
\nonumber\\
&\ll_\eta
\delta_{P_1\lambda^3\nu_1=P_2\lambda^3\nu_2}\,Q''
+\delta_{P_1\lambda^3\nu_1\ne P_2\lambda^3\nu_2}
(XKN(v))^\eta
\prod_{i=1}^2(1+N(\nu_i))^\eta.
\end{align}
Here the auxiliary factors in $P_i$ are polynomially bounded in $XKN(v)$.
The factors involving $\nu_i$ must be retained because the dual sums are
unbounded.  They will be absorbed by the rapidly decaying weights
\eqref{Omegaweight}, using $\eta$ sufficiently small and $D$ sufficiently large.

Write $\mathscr N_{4f}$ and $\mathscr E_{4f}$ for the positive majorants
obtained by inserting respectively the diagonal and off-diagonal terms
of \eqref{Vsumbd} into \eqref{L4fintermed}.  Thus
\begin{equation} \label{summary}
|\mathscr M_{4f}(\boldsymbol a,\cdots,\Lambda_\ell TQ'Q'')|
\ll\mathscr N_{4f}(\boldsymbol a,\cdots,\Lambda_\ell TQ'Q'')
+\mathscr E_{4f}(\boldsymbol a,\cdots,\Lambda_\ell TQ'Q'').
\end{equation}
It remains to bound
\begin{equation} \label{diagavg}
\sum_{\substack{\ell\in\mathbb Z_{\geq0}\\T,Q',Q''\geq1/2\ {\rm dyadic}\\
\Lambda_\ell TQ'Q''\ll C^2}}
\sum_{\boldsymbol a}\mu^2(a_1)|\alpha_{a_1}|
\mu^2(a_2)|\alpha_{a_2}|\,
\mathscr N_{4f}(\boldsymbol a,\cdots,\Lambda_\ell TQ'Q'')
\end{equation}
and
\begin{equation} \label{offdiagavg}
\sum_{\substack{\ell\in\mathbb Z_{\geq0}\\T,Q',Q''\geq1/2\ {\rm dyadic}\\
\Lambda_\ell TQ'Q''\ll C^2}}
\sum_{\boldsymbol a}\mu^2(a_1)|\alpha_{a_1}|
\mu^2(a_2)|\alpha_{a_2}|\,
\mathscr E_{4f}(\boldsymbol a,\cdots,\Lambda_\ell TQ'Q'').
\end{equation}

\subsection{Off-diagonal: \eqref{offdiagavg}} \label{offdiagbounds}
We first record two weighted consequences of Lemma \ref{rankinselbound}.
For $d\equiv1\pmod3$, Cauchy--Schwarz on each dyadic range $N(\nu)\asymp Z$ gives
\[
\sum_{\substack{0\ne\nu\in\lambda^{-3}\mathbb Z[\omega]\\
N(\nu)\asymp Z,\ d\mid\lambda^3\nu}}|\rho_f(\nu)|
\ll_\varepsilon Z^{1+\varepsilon}N(d)^{-1/2}.
\]
The count of nonzero points in the sublattice is $O(Z/N(d))$. If that
sublattice has no points in the range its contribution is zero.
Multiplying by \eqref{Omegaweight} and summing over the dyadic ranges yields
\begin{equation} \label{rhoibound}
\sum_{\substack{0\ne\nu\in\lambda^{-3}\mathbb Z[\omega]\\d\mid\lambda^3\nu}}
|\rho_f(\nu)|\Omega_{L,\mathcal U}(\nu)
\ll (XKN(v))^\varepsilon\frac{L\mathcal U^3}{N(d)^{1/2}}.
\end{equation}
Similarly,
\begin{equation} \label{weightedRS}
\sum_{0\ne\nu\in\lambda^{-3}\mathbb Z[\omega]}
|\rho_f(\nu)|^2\Omega_{L,\mathcal U}(\nu)^2
\ll (XKN(v))^\varepsilon L\mathcal U^2.
\end{equation}
These estimates hold for the scales under consideration.  The contributions
from each dyadic range before the decay factor are respectively
$L^{1/4}Z^{3/4+\varepsilon}/N(d)^{1/2}$ and
$L^{1/2}Z^{1/2+\varepsilon}$.  Consequently they remain valid, with
$\varepsilon$ enlarged arbitrarily slightly, when the summands contain
the small powers $(1+N(\nu))^\eta$ from \eqref{Vsumbd}.

Apply \eqref{rhoibound} with $d=(s_i',q')$, dropping the other
coprimality conditions by positivity.  The factors $N(d)^{1/2}$ cancel
those in \eqref{L4fintermed}.  The remaining $t,q',t_i,s_i'$ sums cost
$(XKN(v))^\varepsilon$ by \eqref{divis} and Lemma \ref{radlemma}.
Using $N(t_i)\ll VT$, $N(s_i')\ll A$, and
$N(\lambda^{m_i})\ll E_0$, we obtain
\begin{align} \label{offdiagcontr}
\mathscr E_{4f}(\boldsymbol a,\cdots,\Lambda_\ell TQ'Q'')
&\ll (XKN(v))^\varepsilon
K\mathcal P\mathcal U^6
\frac{E_0^4V^4T^4A^3(Q'Q'')^4}{X^2}
\nonumber\\
&\ll (XKN(v))^\varepsilon
K^7V^8AC^2\,\frac{E_0^2}{\Lambda_\ell^2}.
\end{align}
In the last step we used \eqref{effectivefrequency}.
For $\ell\leq e_v$ the sum of $E_0^2/\Lambda_\ell^2$ is
$O(N(\lambda^{e_v})^2)$.  For $\ell>e_v$, we have $E_0=\Lambda_\ell$.
There are only $O(\log(2C))$ such terms by \eqref{restriction}.
The sums over dyadic ranges contribute logarithmic factors.
Finally, $\sum_a\mu^2(a)|\alpha_a|
\ll A^{1/2}\|\boldsymbol\mu^2\boldsymbol\alpha\|_2$, and $C^2\asymp B^{1/2}$.
It follows that \eqref{offdiagavg} is
\begin{align} \label{offdiagres}
&\ll (XKN(v))^\varepsilon K^{16}N(v)^8
A^2B^{1/2}\|\boldsymbol\mu^2\boldsymbol\alpha\|_2^2
\nonumber\\
&\ll (XKN(v))^\varepsilon K^{16}N(v)^8
A^3B^{1/2}\|\boldsymbol\mu^2\boldsymbol\alpha\|_\infty^2.
\end{align}

\subsection{Diagonal: \eqref{diagavg}}  \label{diagbounds}
Consulting \eqref{nucond} we make the change of variable 
\begin{equation} \label{changeofvar}
\lambda^3 \nu_i=(s^{\prime}_i,q^{\prime}) \lambda^3 \mu_i, \quad \text{such that} \quad 0 \neq \mu_i \in \lambda^{-3} \mathbb{Z}[\omega]
\quad \text{and} \quad \Big ((s^{\prime}_i,q^{\prime})  \lambda^3 \mu_i, \frac{s^{\prime}_i}{(s^{\prime}_i,q^{\prime})} \Big)=1
\end{equation}
for $i=1,2$. Since $a_i$ is squarefree and $s^{\prime}_i \mid a_i$, the coprimality condition
in \eqref{changeofvar} is equivalent to
\begin{equation} \label{mucoprimerel}
\Big (\lambda^3 \mu_i, \frac{s^{\prime}_i}{(s^{\prime}_i,q^{\prime})} \Big)=1
\end{equation}
for $i=1,2$. The diagonal equation
$P_1 \lambda^3 \nu_1=P_2 \lambda^3 \nu_2$
with $P_1$ and $P_2$ given in \eqref{hdef}
is equivalent to
\begin{equation} \label{diagreduce}
(\zeta_2 \lambda^{m_2-1} t_2 )^2 \frac{s^{\prime}_2 }{(s^{\prime}_2,q^{\prime})} \frac{a_1}{s^{\prime}_1} \lambda^3 \mu_1 
= (\zeta_1 \lambda^{m_1-1} t_1 )^2  \frac{s^{\prime}_1}{(s^{\prime}_1,q^{\prime})} \frac{a_2}{s^{\prime}_2} \lambda^3 \mu_2,
\end{equation}
where $0 \neq \lambda^3 \mu_i$ satisfies \eqref{mucoprimerel} for $i=1,2$.
The hypothesis that the $a_i$ are squarefree for $i=1,2$ guarantees that
\begin{equation} \label{squarerel}
\Big(\frac{s^{\prime}_1}{(s^{\prime}_1,q^{\prime})},\lambda^{m_2-1} t_2  \frac{a_1}{s^{\prime}_1}  \Big)=
\Big(\frac{s^{\prime}_2}{(s^{\prime}_2,q^{\prime})},\lambda^{m_1-1} t_1  \frac{a_2}{s^{\prime}_2}  \Big)=1.
\end{equation}
Using \eqref{mucoprimerel} and \eqref{squarerel}
we conclude from \eqref{diagreduce} that
\begin{equation} \label{sgcdeq}
\widetilde{s}:=\frac{s^{\prime}_1}{(s^{\prime}_1,q^{\prime})}=\frac{s^{\prime}_2}{(s^{\prime}_2,q^{\prime})} \mid (a_1,a_2),
\end{equation}
and thus \eqref{diagreduce} is equivalent to 
\begin{equation} \label{diagreducereduce}
(\zeta_2 \lambda^{m_2-1} t_2 )^2  \frac{a_1}{s^{\prime}_1} \lambda^3 \mu_1 
= (\zeta_1 \lambda^{m_1-1} t_1 )^2  \frac{a_2}{s^{\prime}_2} \lambda^3 \mu_2
\end{equation}
where $0 \neq \lambda^3 \mu_i$ satisfies \eqref{mucoprimerel} for $i=1,2$.

Set $g_i=(s_i',q')$ and write
$s_i'=g_i\widetilde s$ and $a_i=s_i'r_i$.
To release $(s_i'/g_i,q')=1$, insert
$\sum_{h_i\mid(s_i'/g_i,q')}\mu(h_i)$ and put
$s_i'/g_i=h_id_i$.  Thus
\[
s_i'=h_id_ig_i,\qquad a_i=h_id_ig_ir_i,\qquad h_1d_1=h_2d_2.
\]
After taking absolute values of the M\"obius factors, the remaining
$q'$-sum is restricted by
\begin{equation} \label{diagonalqconditions}
[h_1g_1,h_2g_2]\mid q'
\mid\operatorname{rad}(h_1d_1g_1h_2d_2g_2)^\infty.
\end{equation}
Indeed, every prime of $q'$ belongs to $s_1's_2'$ by
$(a_i/s_i',q')=1$ and \eqref{tqqcond}.
We may discard any remaining coprimality conditions that do not appear below.

For fixed $t_i,h_i,d_i,g_i,r_i$, use $t\mid[v_0,t_i]$ for $i=1,2$
and the divisor bound to estimate the $t$-sum.
The $q'$-sum in \eqref{diagonalqconditions} is bounded using
Lemma \ref{radlemma}.  Together they cost $(XKN(v))^\varepsilon$.
Partition the remaining auxiliary variables into ordinary dyadic intervals,
\[
N(h_i)\sim H_i,\quad N(d_i)\sim D_i,\quad N(g_i)\sim G_i,\quad
N(r_i)\sim R_i,\quad N(t_i)\sim T_i.
\]
Their ranges satisfy
\begin{equation} \label{dyadiclist}
H_iD_iG_iR_i\asymp A,\qquad H_iG_i\ll Q',\qquad
T_i\ll VT,\qquad i=1,2.
\end{equation}
For each such range put
\begin{equation} \label{diagonaldualscales}
L_i^*:=\frac{(N(\lambda^{m_i})T_iH_iD_iG_iQ'Q'')^2}{X}.
\end{equation}
Replacing $L_i$ by $L_i^*$ in the weights changes the bound only by a
constant depending on $D$.  In the original $\nu_i$ variables the
conditions retained from \eqref{changeofvar}--\eqref{diagreducereduce} are
\begin{align}
&(\lambda^3\nu_i,h_id_i)=1,\qquad g_i\mid\lambda^3\nu_i,\qquad i=1,2,
\label{nucond1}\\
&(\zeta_2\lambda^{m_2-1}t_2)^2r_1\frac{\lambda^3\nu_1}{g_1}
=(\zeta_1\lambda^{m_1-1}t_1)^2r_2\frac{\lambda^3\nu_2}{g_2}.
\label{linearcond}
\end{align}
Using $|\mu^2(a)\alpha_a|\leq\|\boldsymbol\mu^2\boldsymbol\alpha\|_\infty$,
the contribution to the averaged diagonal of a fixed auxiliary dyadic
range is at most
\begin{align} \label{netintermed2}
&(XKN(v))^\varepsilon K\mathcal P Q''
\|\boldsymbol\mu^2\boldsymbol\alpha\|_\infty^2
(H_1D_1H_2D_2)^{-1/2}
\nonumber\\
&\quad\times
\sum_{\substack{\boldsymbol\zeta,\boldsymbol m,\boldsymbol t\\
6\leq m_i\leq\max\{\ell,e_v\}+2\\
N(t_i)\sim T_i,\ t_i\mid\operatorname{rad}(v_0)^\infty}}
\sum_{\substack{\boldsymbol r\\N(r_i)\sim R_i}}
\sum_{\substack{\boldsymbol h,\boldsymbol d\\
h_1d_1=h_2d_2\\N(h_i)\sim H_i,\ N(d_i)\sim D_i}}
\sum_{\substack{\boldsymbol g\\N(g_i)\sim G_i}}
\sum_{\substack{0\ne\nu_1,\nu_2\in\lambda^{-3}\mathbb Z[\omega]\\
\eqref{nucond1},\ \eqref{linearcond}}}
\prod_{i=1}^2|\rho_f(\nu_i)|\Omega_{L_i^*,\mathcal U}(\nu_i).
\end{align}
All auxiliary factors prime to $\lambda$ in this display are taken
congruent to $1$ modulo $3$.

For fixed $\boldsymbol t,\boldsymbol r,\boldsymbol h,\boldsymbol d$,
apply Cauchy--Schwarz to the last two sums in \eqref{netintermed2}.
For any $\nu_1,g_1,g_2$, equation \eqref{linearcond} determines $\nu_2$
uniquely.  Thus the first squared bracket is bounded by
\[
G_2\sum_{\nu_1\ne0}
|\rho_f(\nu_1)|^2\Omega_{L_1^*,\mathcal U}(\nu_1)^2
\,\#\{g\equiv1\pmod3:g\mid\lambda^3\nu_1\}.
\]
The divisor bound and \eqref{weightedRS} give
$(XKN(v))^\varepsilon G_2L_1^*\mathcal U^2$.
The second bracket is bounded in the same way, with the indices reversed.
Consequently,
\begin{equation} \label{keydiagbd}
\sum_{\substack{\boldsymbol g\\N(g_i)\sim G_i}}
\sum_{\substack{0\ne\nu_1,\nu_2\in\lambda^{-3}\mathbb Z[\omega]\\
\eqref{nucond1},\ \eqref{linearcond}}}
\prod_{i=1}^2|\rho_f(\nu_i)|\Omega_{L_i^*,\mathcal U}(\nu_i)
\ll (XKN(v))^\varepsilon
(G_1G_2L_1^*L_2^*)^{1/2}\mathcal U^2.
\end{equation}

The number of pairs $(h_1,d_1),(h_2,d_2)$ with $h_1d_1=h_2d_2$
in the prescribed ranges is at most
\[
(XKN(v))^\varepsilon(H_1D_1H_2D_2)^{1/2}.
\]
This follows by counting their common product and applying the divisor
bound twice.  It cancels the corresponding denominator in
\eqref{netintermed2}.  The $\boldsymbol r$-sum costs $O(R_1R_2)$,
and Lemma \ref{radlemma} bounds the $\boldsymbol t$-sum by
$(XKN(v))^\varepsilon$.  The $\boldsymbol m,\boldsymbol\zeta$ sums and
the number of auxiliary dyadic ranges also have this cost.
Using $N(\lambda^{m_1+m_2})\ll E_0^2$, equations
\eqref{pretransformprefactor}, \eqref{diagonaldualscales}, and
\eqref{keydiagbd} therefore give
\begin{align} \label{dyadiccalcs}
&\sum_{\boldsymbol a}\mu^2(a_1)|\alpha_{a_1}|
\mu^2(a_2)|\alpha_{a_2}|
\mathscr N_{4f}(\boldsymbol a,\cdots,\Lambda_\ell TQ'Q'')
\nonumber\\
&\ll (XKN(v))^\varepsilon
\|\boldsymbol\mu^2\boldsymbol\alpha\|_\infty^2
\frac{KV^4X\Lambda_\ell Q'(Q'')^2}{C^4A^2T}\mathcal U^2
\sum_{\substack{\text{auxiliary dyadic ranges}\\\eqref{dyadiclist}}}
\prod_{i=1}^2 H_iD_iR_iT_iG_i^{3/2}
\nonumber\\
&\ll (XKN(v))^\varepsilon
\|\boldsymbol\mu^2\boldsymbol\alpha\|_\infty^2
\frac{KV^6X\Lambda_\ell q_0Q'Q''}{C^4}\mathcal U^2
\nonumber\\
&\ll (XKN(v))^\varepsilon K^3V^6X
\frac{Q'Q''}{C^2}\|\boldsymbol\mu^2\boldsymbol\alpha\|_\infty^2.
\end{align}
For the second inequality, \eqref{dyadiclist} gives
$\prod_iH_iD_iR_iT_iG_i^{3/2}\ll A^2V^2T^2Q'$.
The third uses \eqref{effectivefrequency}.
The restriction $\Lambda_\ell TQ'Q''\ll C^2$ implies
$Q'Q''/C^2\ll(\Lambda_\ell T)^{-1}$.
Summing over the original dyadic ranges and $\ell\geq0$, and recalling
$X\asymp AB$, proves that \eqref{diagavg} is
\begin{equation} \label{diagres}
\ll (XKN(v))^\varepsilon K^8N(v_0)^6AB
\|\boldsymbol\mu^2\boldsymbol\alpha\|_\infty^2.
\end{equation}

For completeness, the cases $p=1,2,3$ use the same weighted estimates.
Put $\kappa_i=\lambda^{2m_i+3}\nu_i$ when $0\leq m_i\leq5$,
and $\kappa_i=\lambda^3\nu_i$ when $m_i\geq6$.
For $0\leq m_i\leq5$, the Fourier coefficients at the resulting cusps
belong to a finite family of cusp forms of bounded level.
Lemma \ref{rankinselbound} applies uniformly to this family after these
fixed dilations, with constants depending only on $f$.
The prime-to-$\lambda$ support conditions and diagonal relation are
those recorded in \eqref{allcasesdiagonal}. The factors $b_i$ are
$\lambda^{m_i+3}t_ir_i'$ when $0\leq m_i\leq5$ and
$\zeta_i\lambda^{m_i-1}t_ir_i'$ when $m_i\geq6$.
The additional powers of $\lambda$ when $m_i\leq5$ are bounded and
coprime to $s_i'$, so the squarefree argument
\eqref{changeofvar}--\eqref{diagreducereduce}, the divisor estimates, and
the uniqueness assertion used in \eqref{keydiagbd} are unchanged.
The local bound \eqref{smallramifiedbound} supplies the same bound
$O(N(v_0)^6)$. Moreover, the terms with $\ell>e_v$ vanish if
$m_i\leq5$ for either $i=1$ or $i=2$. Thus \eqref{offdiagres} and \eqref{diagres} also bound these
three cases.

Combining these estimates in \eqref{averagestatement} and using
$N(v_0)\leq N(v)$ proves Proposition \ref{conv1}.
\end{proof}

\begin{proof}[Proof of Theorem~\ref{type2a}]
This follows immediately from Lemma \ref{convolutionsums} and Proposition \ref{conv1}.
\end{proof}

\bibliographystyle{amsalpha}
\bibliography{weylcubic} 
\end{document}